\documentclass[12pt]{amsart}

\usepackage{amssymb}
\usepackage[scr=boondox]{mathalpha}
\usepackage[all]{xy}
\usepackage{enumerate}
\usepackage{color}
\usepackage{graphicx}
\usepackage[text={6in,9in},centering]{geometry}
\usepackage{hyperref}
\usepackage{relsize}
\usepackage{booktabs}
\usepackage{longtable}
\usepackage{makecell}
\usepackage{multirow}
\usepackage{caption}

\hypersetup{
    colorlinks = true,
    linkcolor = blue,
    anchorcolor = blue,
    citecolor = blue
}

\numberwithin{equation}{section}

\theoremstyle{plain}
  \newtheorem{thm}{Theorem}[section]
  \newtheorem*{thm*}{Theorem}
  \newtheorem{cor}[thm]{Corollary}
  \newtheorem{lem}[thm]{Lemma}
  \newtheorem{prop}[thm]{Proposition}
  \newtheorem*{fact}{Fact}
\theoremstyle{definition}
  \newtheorem{defn}{Definition}

  \newtheorem{claim}{Claim}
\theoremstyle{remark}
  \newtheorem{rmk}{Remark}

\newcommand{\id}[0]{\mathrm{id}}
\newcommand{\e}[0]{\mathrm{e}}

\newcommand{\diam}[0]{\mathrm{diam}}

\newcommand{\proj}{\operatorname{Proj}}

\newcommand{\dimh}[0]{\dim_{\mathrm{H}}}
\newcommand{\dimp}[0]{\dim_{\mathrm{P}}}
\newcommand{\vect}[1]{\boldsymbol{#1}}

\newcommand{\numspan}{\#^{\mathrm{span}}}
\newcommand{\numsep}{\#^{\mathrm{sep}}}
\newcommand{\entlow}{\underline{h}}
\newcommand{\entup}{\overline{h}}
\newcommand{\enttop}{h_{\mathrm{top}}}
\newcommand{\enttopspan}{\enttop^{\mathrm{span}}}
\newcommand{\enttopsep}{\enttop^{\mathrm{sep}}}
\newcommand{\enttopbow}{\enttop^{\mathrm{B}}}
\newcommand{\enttoppac}{\enttop^{\mathrm{P}}}
\newcommand{\enttoplow}{\entlow_{\mathrm{top}}}
\newcommand{\enttopup}{\entup_{\mathrm{top}}}
\newcommand{\entbow}{h^{\mathrm{B}}}
\newcommand{\entpac}{h^{\mathrm{P}}}
\newcommand{\prelow}{\underline{P}}
\newcommand{\preup}{\overline{P}}
\newcommand{\qrelow}{\underline{Q}}
\newcommand{\qreup}{\overline{Q}}
\newcommand{\preL}{P^{\mathrm{L}}}
\newcommand{\preU}{P^{\mathrm{U}}}
\newcommand{\prebpp}{Q^{\mathrm{B}}}
\newcommand{\prebow}{P^{\mathrm{B}}}
\newcommand{\prebw}{P^{\mathrm{BW}}}

\newcommand{\prepac}{P^{\mathrm{P}}}

\newcommand{\carpes}{\mathscr{Q}}
\newcommand{\carpep}{\mathscr{K}}
\newcommand{\carpeslow}{\underline{\mathscr{Q}}}
\newcommand{\carpeplow}{\underline{\mathscr{K}}}
\newcommand{\carpesup}{\overline{\mathscr{Q}}}
\newcommand{\carpepup}{\overline{\mathscr{K}}}
\newcommand{\msrbpp}{\mathscr{M}}

\newcommand{\msrbow}{\mathscr{R}}
\newcommand{\msrpac}{\mathscr{P}}

\newcommand{\R}{\mathbb{R}}

\newcommand{\N}{\mathbb{N}}

\newcommand{\len}[1]{\lvert #1 \rvert}

\begin{document}

\title[Nonautonomous Dynamical Systems]{Nonautonomous Dynamical Systems \uppercase\expandafter{\romannumeral1}: Topological Pressures and Entropies}

\author{Zhuo Chen}
\address{School of Mathematical Sciences, East China Normal University, No. 500, Dongchuan Road, Shanghai 200241, P. R. China}
\email{10211510056@stu.ecnu.edu.cn, charlesc-hen@outlook.com}

\author{Jun Jie Miao}
\address{School of Mathematical Sciences,  Key Laboratory of MEA(Ministry of Education) \& Shanghai Key Laboratory of PMMP,  East China Normal University, Shanghai 200241, China}

\email{jjmiao@math.ecnu.edu.cn}




\begin{abstract}
  Let $\boldsymbol{X}=\{X_{k}\}_{k=0}^{\infty}$ be a sequence of compact metric spaces $X_{k}$ and $\boldsymbol{T}=\{T_{k}\}_{k=0}^{\infty}$ a sequence of continuous mappings $T_{k}:X_{k} \to X_{k+1}$.
  The pair $(\boldsymbol{X},\boldsymbol{T})$ is called a nonautonomous dynamical system. Our main object is to study the variational principles of topological pressures and entropies on nonautonomous dynamical systems.

  In this paper, we introduce a variety of topological pressures ($\underline{Q}$,$\overline{Q}$,$\underline{P}$,$\overline{P}$,$P^{\mathrm{B}}$ and $P^{\mathrm{P}}$) for potentials $\boldsymbol{f}=\{f_{k} \in C(X_{k},\mathbb{R})\}_{k=0}^{\infty}$ on subsets $Z \subset X_{0}$
  analogous to fractal dimensions, and we provide various key properties which are crucial for the study of the variational principles on nonautonomous dynamical systems. Especially, we obtain the power rules and product rules of these pressures, and we also show they are  invariants under equiconjugacies of nonautonomous dynamical systems and equicontinuity on $\boldsymbol{f}$.
  From a fractal dimension point of view, these pressures are kinds of 'dimensions' describing the nonautonomous dynamical systems, and we obtain various properties of pressures analogous to fractal dimensions.

\end{abstract}

\maketitle

  \section{Introduction}
    \subsection{Topological pressures and fractal dimensions}
    By a \emph{topological dynamical system (TDS)}, we mean a pair $(X,T)$, where $X$ is a compact metric space and $T:X \to X$ is a continuous mapping.
    The classic theory of dynamical systems deals with invariant sets under the mapping $T$.
    In the 1970s, the topological pressure was first introduced by Ruelle \cite{Ruelle1972} and Walters \cite{Walters1975}, and it may be regarded as  a nontrivial generalization of the topological entropy. A pivotal definition of the topological entropy using metrics was given by Dinaburg \cite{Dinaburg1970} and Bowen \cite{Bowen1971} independently,
    and this approach became standard for the topological pressure of TDSs (see, for instance, \cite{Walters1982}).
    Topological pressure has become a fundamental concept in thermodynamic formalism, which is known to be of great importance in dynamical systems, ergodic theory, fractal geometry and other fields of mathematics.
    We refer readers to \cite{Barreira2011,Bowen2008,Falconer1997,Pesin1997,Ruelle2004,Walters1982} and references therein.

    Fractal dimensions, on the other hand, is central to fractal geometry as a major tool in studying the structures of various `geometrically irregular' sets encountered in dynamical systems, number theory, stochastic processes, and other areas of mathematics or sciences; see \cite{BHR19,DasSim17,Feng23,FH08,Gu&Miao2022,GM,Hutch81} for various studies on fractal dimensions.
    The idea of defining measures using covers by small sets was introduced by Carathéodory in \cite{Carath}, and in 1918, Hausdorff \cite{Hausdorff1918} used this method to define the   measures and dimensions that now bear his name. Since then,  a wide variety of fractal dimensions have been introduced, and   the most commonly used ones are Hausdorff dimension, packing dimension, and box-counting dimension. Given a non-empty bounded subset $F \subset \R^{n}$, the following dimension inequalities are well known:
    \begin{equation}\label{ineq_dim}
    \dimh{F} \leq \underline{\dim}_{\rm B}{F} \leq \overline{\dim}_{\rm B}{F} \quad \textit{ and }\quad
    \dimh{F} \leq \dimp{F} \leq \overline{\dim}_{\rm B}{F},
    \end{equation}
    where $\dimh{F}$, $\dimp{F}$, $\underline{\dim}_{\rm B}{F}$ and $\overline{\dim}_{\rm B}{F}$ denote the Hausdorff, packing, lower box-counting and upper box-counting dimensions of $F$, respectively.
The dimension inequalities of product fractals are also important and  frequently used in various studies,
\begin{equation} \label{ineq_fdHP}
\begin{split}
&  \dimh{E}+ \dimh{F} \leq \dimh{(E\times F)}\leq \dimh{E}+ \overline{\dim}_{\rm B}{F}; \\
&  \dimh{E}+ \dimp{F} \leq \dimp{(E\times F)}\leq \dimp{E}+   \dimp{F} .
\end{split}
\end{equation}
Note that the inequalities may hold strictly for many sets with high irregularities. We refer readers to   books \cite{Falconer2014, Wen00} for a comprehensive introduction to these dimensions   and their properties. See \cite{Howroyd1996} for the packing dimension inequalities of product sets.

    There is a resemblance between the topological pressures and the fractal dimensions.
    The idea originated in the studies of topological entropies.
    The initial definition of the topological entropy by Adler, Konheim and McAndrew \cite{AKM1965} shared similarities with the (upper) box dimension.
    Bowen in \cite{Bowen1971} extended the  topological entropy to general subsets using a metric formulation of the box dimension, also known as Kolmogorov's entropy dimension \cite{Kolmogorov&Tihomirov1961},
    and he also provided an alternative notion of topological entropy on subsets in \cite{Bowen1973} where the definition was given for sets which need not be invariant via a class of Hausdorff measures, much like the Hausdorff dimension of fractal sets.
    In 2012, Feng and Huang \cite{Feng&Huang2012} introduced the packing topological entropy similar to the packing dimension which was first defined by Tricot in \cite{Tricot1982}.
    These types of topological entropies on subsets may be considered as a type of dimensions or structural complexities sensitive to `dynamically irregular' sets.
    Bowen's Hausdorff dimension-like entropy was generalized to a type of topological pressure by Pesin and Pitskel' \cite{Pesin&Pitskel1984}, and we call it the \emph{Bowen-Pesin-Pitskel' topological pressure}.
    Pesin \cite{Pesin1997} then developed the general Carath\'{e}odory dimension structure theory and gave a definition for the pressure as capacities, which is comparable to the upper box dimension and completely equivalent to the classic pressure on all subsets.
    This also led to the notion of  upper   and  lower capacity topological pressures on non-invariant subsets, and they run parallel to upper and lower box dimensions. See Section \ref{sect:PNDS} for these definitions.

    Given a TDS $(X,T)$ and a non-empty $Z \subset X$, for all $f \in C(X,\R)$,   pressures have the following inequalities similar to dimensions,
    $$
    \prebow(T,f,Z) \leq \prelow(T,f,Z) \leq \preup(T,f,Z),
    $$
    where $\prebow(T,f,Z)$, $\prelow(T,f,Z)$, and $\preup(T,f,Z)$ denote the Bowen-Pesin-Pitskel', lower capacity and upper capacity topological pressures of $T$ for $f$ on $Z$, respectively.
    More recently, Zhong and Chen \cite{Zhong&Chen2023} generalized Feng and Huang's work to  packing topological pressure,
    and they gave a similar  pressure inequalities,
    $$
    \prebow(T,f,Z) \leq \prepac(T,f,Z) \leq \preup(T,f,Z).
    $$
    Strict inequalities may occur on subsets $Z$ that are not invariant; for examples, see \cite[Exmp. 11.1 \& 11.2]{Pesin1997}.
    These different pressures turn out to be useful in more general situations, for example, nonadditive cases, see  \cite{Barreira1996, Barreira2011,Pesin1997}.

    We are interested in the variational principles for topological pressures and entropies in nonautonomous dynamical systems.
    Since this study depends heavily on various properties of topological pressures and entropies, this paper is devoted to verifying the relevant and necessary properties of them for the further study of their variational principles and their applications in \cite{CM2,CM3}. We find that the pressures in nonautonomous dynamical systems have many interesting properties similar to fractal dimensions, and we obtain various pressure inequalities analogous to the fractal dimension inequalities \eqref{ineq_dim} and \eqref{ineq_fdHP}.

    \subsection{Nonautonomous dynamical systems and Nonautonomous fractals }
    Let the pair $(\vect{X},\vect{d})=\{(X_{k},d_{k})\}_{k=0}^{\infty}$ be a sequence of compact metric spaces $X_{k}$ each endowed with the metric $d_{k}$, and let $\vect{T}=\{T_{k}\}_{k=0}^{\infty}$ be a sequence of continuous mappings $T_{k}:X_{k} \to X_{k+1}$.
    We call the pair $(\vect{X},\vect{T})$ a \emph{nonautonomous dynamical system (NDS)}, and we sometimes write the triplet $(\vect{X},\vect{d},\vect{T})$ to emphasize the dependence on the metrics.
    If $\vect{X}$ is constant, i.e., $X_{k} = X$ for all integral $k \geq 0$, we call $(\vect{X},\vect{T})$ a \emph{nonautonomous dynamical system with an identical space} and denote it by $(X,\vect{T})$.
    Furthermore, if $\vect{T}$ is also constant, i.e., $T_{k}=T: X \to X$ for all integral $k \geq 0$, then the nonautonomous system $(\vect{X},\vect{T})$ degenerates into the classic TDS $(X,T)$.

    There has been an appreciable amount of work focusing on NDSs $(X,\vect{T})$ with an identical space.
    We refer readers to the books \cite{APR2023,Kloeden&Potzsche2013,Kloeden&Rasmussen2011} and references therein for a general introduction to the theories and applications of NDSs with an identical space.
    Kolyada and Snoha in \cite{Kolyada&Snoha1996} studied the topological entropies on subsets $Z$ of NDSs $(X,\vect{T})$ with an identical space.

    Nonautonomous dynamical systems are strongly related to nonautonomous fractals in fractal geometry  (see~\cite{Gu&Miao2022,GM, Wen00}).
    A special case of such fractals is called Moran sets, which was first studied by Moran~\cite{Moran} in 1946.
    The dimension theory of nonautonomous fractals has been studied extensively, and the pressure functions and entropies often play important roles in the study of dimensions. In particular, the dimension inequalities \eqref{ineq_dim} often strictly hold  for nonautonomous fractals.   We refer readers to \cite{Gu&Miao2022,GM, Hua1994,HRWW2000,R-G&Urbanski2012,Wen00,Wen01} for various studies.

    Inspired by recent progress in nonautonomous fractals (see \cite{Gu&Miao2022,GM,R-G&Urbanski2012}), we aim to explore the properties of topological pressures and entropies on nonautonomous dynamical systems $(\vect{X},\vect{T})$.
    In 2014, Kawan \cite{Kawan2014} studied topological entropies for general NDSs $(\vect{X},\vect{T})$, and he  generalized the  equivalent definitions of topological entropy via open covers, as well as the so called $(n,\varepsilon)$-spanning and $(n,\varepsilon)$-separated sets.

    Given an NDS $(\vect{X},\vect{d},\vect{T})$, for each integer $k\geq 0$, we write
    $$
    \vect{T}_{k}^{j}=T_{k+(j-1)} \circ \cdots \circ T_{k}: X_{k} \to X_{k+j}
    $$
    for $ j = 1,2,3,\cdots$, and we adopt the convention that  $ \vect{T}_{k}^{0}=\id_{X_{k}}$
    where $\id_{X_{k}}:X_{k} \to X_{k}$ is the identity mapping.
    Since the mappings $T_{k}$ are not necessarily bijective, we write $\vect{T}_{k}^{-j} = (\vect{T}_{k}^{j})^{-1}$ for the preimage of subsets of $X_{k+j}$ under $\vect{T}_{k}^{j}$.
    If the mappings $T_{k},\cdots,T_{k+(j-1)}$ are bijective, we also use $\vect{T}_{k}^{-j} : X_{k+j} \to X_{k}$ for the inverse of $\vect{T}_{k}^{j}$, that is,
    \begin{equation}\label{def_TInv}
      \vect{T}_{k}^{-j}=(T_{k+(j-1)} \circ \cdots \circ T_{k})^{-1}= T_{k}^{-1} \circ \cdots \circ T_{k+(j-1)}^{-1}.
    \end{equation}
    For simplicity, we often write $\vect{T}^{j}=\vect{T}_{0}^{j}$ for $j \in \mathbb{Z}$.

    Bowen metrics are the key essence in the study of nonautonomous dynamical systems.
    Given $k \in \mathbb{N}$, for $n=1,2,3,\ldots,$ we define the {\emph{$n$-th Bowen metric at level $k$}} or the {\emph{$n$-th Bowen metric} on $X_{k}$} by
    $$
    d_{k,n}^{\vect{T}}(x,y)=\max_{0 \leq j \leq n-1}d_{k+j}(\vect{T}_{k}^{j}x,\vect{T}_{k}^{j}y)
    $$
    for all $x,y \in X_{k}$. It is routine to check that each $d_{k,n}^{\vect{T}}$ is a metric on $X_{k}$ topologically equivalent to $d_{k}$ for every $k\in \mathbb{N}$.

    Recall that  $(\vect{X},\vect{d})=\{(X_{k},d_{k})\}_{k=0}^{\infty}$ is a sequence of compact metric spaces. Given a subset $Z \subset X_{0}$, we call $F\subset X_{0}$   a \emph{$(n,\varepsilon)$-spanning set for $Z$ with respect to $\vect{T}$}  if for every $x \in Z$,  there exists $y \in F$ with $d_{0,n}^{\vect{T}}(x,y) \leq \varepsilon$.
    As a dual counterpart, a set $E \subset Z$ is called a \emph{$(n,\varepsilon)$-separated set for $Z$ with respect to $\vect{T}$}
    if  any two distinct points $x,y \in E$  implies $d_{0,n}^{\vect{T}}(x,y)>\varepsilon$.
    We denote by $\numspan_{n,\varepsilon}(\vect{T},Z)$ the smallest cardinality of all $(n,\varepsilon)$-spanning sets for $Z$ with respect to $\vect{T}$
    and by $\numsep_{n,\varepsilon}(\vect{T},Z)$ the largest cardinality of all $(n,\varepsilon)$-separated sets
    for $Z$ with respect to $\vect{T}$, i.e.,
    \[ \begin{split}
    \numspan_{n,\varepsilon}(\vect{T},Z)&=\inf\{\#{F}: F\ \text{is a}\ (n,\varepsilon)\text{-spanning set for}\ Z\ \text{with respect to}\ \vect{T}\},  \\
    \numsep_{n,\varepsilon}(\vect{T},Z) &=\sup\{\#{E}: E\ \text{is a}\ (n,\varepsilon)\text{-separated set for}\ Z\ \text{with respect to}\ \vect{T}\}.
    \end{split} \]
    Write
    $$
    r(\vect{T},\varepsilon,Z)=\varlimsup_{n \to \infty}\frac{1}{n}\log{\numspan_{n,\varepsilon}(\vect{T},Z)}
    \quad \text{and} \quad
    s(\vect{T},\varepsilon,Z)=\varlimsup_{n \to \infty}\frac{1}{n}\log{\numsep_{n,\varepsilon}(\vect{T},Z)}.
    $$
    Note that $\numspan_{n,\varepsilon}(\vect{T},Z)$ and $\numsep_{n,\varepsilon}(\vect{T},Z)$ are both decreasing in $\varepsilon>0$,
    and hence so are $r(\vect{T},\varepsilon,Z)$ and $s(\vect{T},\varepsilon,Z)$. It follows that the limits
    $$
    \enttopspan(\vect{T},Z)=\lim_{\varepsilon \to 0}r(\vect{T},\varepsilon,Z)
    \quad \text{and} \quad
    \enttopsep(\vect{T},Z)=\lim_{\varepsilon \to 0}s(\vect{T},\varepsilon,Z)
    $$
    exist, and we write
    $$
    \enttopspan(\vect{T})=h(\vect{T},X_{0})
    \quad \text{and} \quad
    \enttopsep(\vect{T})=h(\vect{T},X_{0}).
    $$
    It is easy to show that the two entropies $\enttopspan$ and $\enttopsep$ are equivalent on all subsets $Z \subset X_{0}$.
    (See \cite[Lem.1]{Bowen1971} or \cite[Lem.3.1]{Kolyada&Snoha1996} for the argument; see also \cite[\S7.2, Rmk(9)]{Walters1982}.)
    Thus, it is natural to define their  common value as the topological entropy.
    \begin{defn}
      We define the \emph{topological entropy of $\vect{T}$} by
      $$
      \enttop(\vect{T})=\enttopspan(\vect{T})=\enttopsep(\vect{T}).
      $$
    \end{defn}

    Huang et. al.\cite{Huang&Wen&Zeng2008} extended the notion of topological pressure for a potential function $f \in C(X,\R)$ to that in $(X,\vect{T})$,
    and they provided the following nonautonomous version of the classic power rule for the topological pressure in $(X,\vect{T})$.
    \begin{thm*}
      Given an NDS $(X,\vect{T})$ with the identical space $X$, if $\vect{T}$ is equicontinuous, then for all integral $m \geq 1$,
      $$
      P(\{\vect{T}_{k \cdot m}^{m}\}_{k=0}^{\infty},\sum_{j=0}^{m-1}f(\vect{T}^{j}x)) = m \cdot P(\vect{T},f).
      $$
    \end{thm*}
    The particular case of topological entropy was previously obtained by \cite[Lem.4.4]{Kolyada&Snoha1996} in $(X,\vect{T})$, and in 2014, Kawan generalized it  to the case of topological entropy in general NDSs $(\vect{X},\vect{T})$ \cite[Prop.2.5]{Kawan2014}.
    Some other properties of the topological pressure for a potential $f \in C(X,\R)$ on $(X,\vect{T})$ were given by Kong, Cheng, and Li \cite{Kong&Cheng&Li2015}.

    In this paper, we extend the definitions of various topological pressures and entropies to general NDSs $(\vect{X},\vect{T})$ and explore their properties.
    In \S\ref{sect:PNDS}, we give the definitions of the lower, upper, Bowen and packing topological pressures and entropies.  In \S\ref{sect:EquivDefs}, we provide some useful equivalent definitions  for these topological pressures and entropies.  Since we consider these pressures from the fractal dimension point of view, we prove various properties of pressures analogous to fractal dimensions in  \S\ref{sect:PropFracDim}, such as  an equality on certain subsets where the lower and upper topological pressures have improved stability behavior (Theorem~\ref{thm:preeqpremod}),
    pressure inequalities (Proposition~\ref{prop:preneq}) analogous to \eqref{ineq_dim}
    and the relationship between the packing and upper topological pressures (Theorem~\ref{thm:prepaceqpremod} and Corollary~\ref{cor:PupeqPP}).
    In  \S\ref{sect:Prop}, we provide the dynamical properties of pressures on the potentials, and power rules (Theorem~\ref{thm:Pmpoweq}) for pressures are obtained.
    In \S\ref{sect:prod}, we state different kinds of  products rules of   pressures in NDSs, and in particular, we obtain parallel conclusion to dimension inequalities of product fractals \eqref{ineq_fdHP} for pressures in Theorem \ref{thm:prodBBP}, Theorem \ref{thm:prodPPPB} and Corollary \ref{cor:prodeq}.
    In the final section, we show the invariance (Theorem~\ref{thm:invar}) of the pressures under equiconjugacies   and state their `topological' nature up to uniform equivalence (Theorem~\ref{thm:pretopuniequiv}).

  \section{Pressures and entropies of Nonautonomous Dynamical Systems}\label{sect:PNDS}
In the rest of the paper, we always write  $(\vect{X},\vect{T})$ for nonautonomous dynamical systems, where $(\vect{X},\vect{d})=\{(X_{k},d_{k})\}_{k=0}^{\infty}$ is a sequence of compact metric spaces, and $\vect{T} =\{T_k\}_{k=0}^\infty$ is a sequence of continuous mappings.
    \subsection{Bowen balls and  potentials}
    In this subsection, we introduce the Bowen balls and the potentials for the study of pressures in NDSs.

    Given $\varepsilon>0$ and $x \in X_{k}$, let $B_{k,n}^{\vect{T}}(x,\varepsilon)$ denote the
    \emph{(open) Bowen ball} with center  $x$ and  radius $\varepsilon$ at level $k$, i.e.,
    $$
    B_{k,n}^{\vect{T}}(x,\varepsilon)=\{y \in X_{k}: d_{k,n}^{\vect{T}}(x,y)<\varepsilon\},
    $$
    and let $\overline{B}_{k,n}^{\vect{T}}(x,\varepsilon)$ denote the \emph{closed Bowen ball} with
    center $x$ and radius $\varepsilon$ at level $k$. Alternatively, the open and closed Bowen balls  are given by
    \begin{equation}\label{eq:bowenball}
      B_{k,n}^{\vect{T}}(x,\varepsilon)=\bigcap_{j=0}^{n-1}\vect{T}_{k}^{-j}B_{X_{k+j}}(\vect{T}_{k}^{j}x,\varepsilon),
      \qquad
      \overline{B}_{k,n}^{\vect{T}}(x,\varepsilon)=\bigcap_{j=0}^{n-1}\vect{T}_{k}^{-j}\overline{B}_{X_{k+j}}(\vect{T}_{k}^{j}x,\varepsilon).
    \end{equation}

    To study the pressures of nonautonomous dynamical systems, we need a sequence of functions as potentials.
    Let
    $$
    \vect{C}(\vect{X},\mathbb{R})=\prod_{k=0}^{\infty}C(X_{k},\mathbb{R})
    $$
    be the collection of all sequences of continuous functions $f_k: X_k \to \mathbb{R}$.  We  write $\vect{0}$ and $\vect{1} \in \vect{C}(\vect{X},\mathbb{R})$ for the sequence of constant $0$ functions and constant $1$ functions, respectively, and $a\vect{1} \in \vect{C}(\vect{X},\mathbb{R})$ for the sequence of constant $a$ functions where $a \in \mathbb{R}$.
    Given $\vect{f}\in \vect{C}(\vect{X},\mathbb{R})$, we write
    \begin{equation}\label{eq_fnorm}
      \|\vect{f}\|=\sup_{k \in \mathbb{N}}\left\{\|f_{k}\|_{\infty}=\max_{x \in X_{k}}\lvert f_{k}(x) \rvert\right\}.
    \end{equation}
    It is clear that $\|\vect{f}\|<+\infty$ implies that $\vect{f}$ is uniformly bounded. Let $\vect{C}_{b}(\vect{X},\mathbb{R})$ denote the collection of all uniformly bounded function sequences  in $\vect{C}(\vect{X},\mathbb{R})$, i.e.,
    $$\vect{C}_{b}(\vect{X},\mathbb{R})=\{\vect{f}\in \vect{C}(\vect{X},\mathbb{R}): \|\vect{f}\|<+\infty\}.$$
    For each $\vect{f} \in \vect{C}(\vect{X},\R)$, we write
\begin{equation}\label{def_absf}
    \lvert \vect{f} \rvert = \{\lvert f_{k} \rvert\}_{k=0}^{\infty} \in \vect{C}(\vect{X},\R).
\end{equation}
    Given $\vect{f}=\{f_{k}\}_{k=0}^{\infty}$ and $\vect{g}=\{g_{k}\}_{k=0}^{\infty} \in \vect{C}(\vect{X},\mathbb{R})$, we write $\vect{f} \preceq \vect{g}$
    if $f_{k} \leq g_{k}$ for all $k \in \mathbb{N}$.

     We say  $\vect{f} \in \vect{C}(\vect{X},\mathbb{R})$ is \emph{equicontinuous} if for every $\varepsilon>0$,
    there exists $\delta>0$ such that for all $k \in \mathbb{N}$ and all $x^{\prime},x^{\prime\prime} \in X_{k}$ satisfying
    $d_{k}(x^{\prime},x^{\prime\prime})<\delta$, we have that
    $$
    \lvert f_{k}(x^{\prime})-f_{k}(x^{\prime\prime}) \rvert<\varepsilon.
    $$
    Note that if $\vect{X}$ is constant, i.e., $X_k=X$ for all $k \in \mathbb{N}$, then by the compactness of $X$, the equicontinuity of $\vect{f}$ coincides with the conventional definition of equicontinuity.
    In particular, for the nonautonomous dynamical systems  with an identical space $(X,\vect{T})$ and the TDSs $(X,T)$, we usually require $f_{k}=f$ for all $k \in \N$ where $f \in C(X,\R)$, and  in this case, $\vect{f}=\{f\}_{k=0}^{\infty}$ is clearly equicontinuous.  Given $\vect{f} \in \vect{C}(\vect{X},\mathbb{R})$. For each  $k, n \in \mathbb{N}$, we write
  \begin{equation}\label{eq:sumknTf}
    S_{k,n}^{\vect{T}}\vect{f}=\sum_{j=0}^{n-1} f_{k+j} \circ \vect{T}_{k}^{j}.
  \end{equation}
Note that each $S_{k,n}^{\vect{T}}\vect{f}$ is a continuous real-valued function on $X_{k}$.

For $k \in \mathbb{N}$, write $\vect{X}_{k}=\{X_{j}\}_{j=k}^{\infty}$  and $\vect{T}_{k}=\{T_{j}\}_{j=k}^{\infty}$. Then $(\vect{X}_{k},\vect{T}_{k})$ is  also a nonautonomous dynamical system. Note that  $(\vect{X}_{0},\vect{T}_{0})= (\vect{X},\vect{T})$. For $k>0$,  we write $\vect{X}'=\vect{X}_{k}=\{X_{j+k}\}_{j=0}^{\infty}$ and $\vect{T}'=\vect{T}_{k}=\{T_{j+k}\}_{j=0}^{\infty}$, and  it is sufficient to consider the new NDS $(\vect{X}',\vect{T}')$ with initial space at index $0$. Therefore, in this paper,  we only   focus on the NDS $(\vect{X},\vect{T})$ with initial space starting from $k=0$, that is,  $\vect{X} =\{X_{k}\}_{k=0}^{\infty}$  and $\vect{T} =\{T_k\}_{k=0}^\infty.$
For simplicity, we write $d_{n}^{\vect{T}}=d_{0,n}^{\vect{T}}$,
$$
B_{n}^{\vect{T}}(x,\varepsilon)=B_{0,n}^{\vect{T}}(x,\varepsilon), \qquad
 \overline{B}_{n}^{\vect{T}}(x,\varepsilon)=\overline{B}_{0,n}^{\vect{T}}(x,\varepsilon),
$$
 and
  \begin{equation}\label{eq:sum0nTf}
    S_{n}^{\vect{T}}\vect{f}=S_{0,n}^{\vect{T}}\vect{f}=\sum_{j=0}^{n-1} f_{j} \circ \vect{T}^{j}.
  \end{equation}

    \subsection{Pressures and entropies  via  spanning and separated sets}
      In this subsection, we give definitions for the topological pressures entropies of NDSs on subsets via the $(n,\varepsilon)$-spanning and $(n,\varepsilon)$-separated sets.

      Given   $Z \subset X_{0}$,   $\vect{f} \in \vect{C}(\vect{X},\R)$, an integer $n \geq 1$ and a real $\varepsilon>0$, we write
      \begin{equation}\label{eq:defQn}
       Q_{n}(\vect{T},\vect{f},Z,\varepsilon) = \inf\Big\{ \sum_{x \in F}\e^{S_{n}^{\vect{T}}\vect{f}(x)}: F\ \text{is a}\ (n,\varepsilon)\text{-spanning set for}\ Z\  \Big\};
      \end{equation}
      \begin{equation}\label{eq:defPn}
        P_{n}(\vect{T},\vect{f},Z,\varepsilon) = \sup\Big\{ \sum_{x \in E}\e^{S_{n}^{\vect{T}}\vect{f}(x)}: E\ \text{is a}\ (n,\varepsilon)\text{-separated set for}\ Z\  \Big\}.
      \end{equation}
      Since the exponential function is positive, it suffices to take the infimum in \eqref{eq:defQn} over minimal $(n,\varepsilon)$-spanning sets, i.e., those sets which do not have proper subsets that $(n,\varepsilon)$ span $Z$. Similarly,  the supremum in \eqref{eq:defPn} is taken over maximal $(n,\varepsilon)$-separated sets, i.e., those  sets that fail to be $(n,\varepsilon)$ separated when any point of $Z$ is added.

      Unlike autonomous cases, we do not have the subadditivity of $Q_{n}$ and $P_{n}$;  see \cite[2.7]{Bowen2008} for their counterpart of pressures in TDSs and \cite{Kawan2014,Kolyada&Snoha1996} for the ones  of entropies in  NDSs.
      Hence we  consider both lower and upper limits and write
      \begin{equation}\label{eq:Qe}
        \begin{aligned}
          \qrelow(\vect{T},\vect{f},Z,\varepsilon) &= \varliminf_{n \to \infty}\frac{1}{n}\log{Q_{n}(\vect{T},\vect{f},Z,\varepsilon)}, \\
          \qreup(\vect{T},\vect{f},Z,\varepsilon) &= \varlimsup_{n \to \infty}\frac{1}{n}\log{Q_{n}(\vect{T},\vect{f},Z,\varepsilon)},
        \end{aligned}
      \end{equation}
      \begin{equation}\label{eq:Pe}
        \begin{aligned}
          \prelow(\vect{T},\vect{f},Z,\varepsilon) &= \varliminf_{n \to \infty}\frac{1}{n}\log{P_{n}(\vect{T},\vect{f},Z,\varepsilon)}, \\
          \preup(\vect{T},\vect{f},Z,\varepsilon) &= \varlimsup_{n \to \infty}\frac{1}{n}\log{P_{n}(\vect{T},\vect{f},Z,\varepsilon)}.
        \end{aligned}
      \end{equation}
      It is straightforward to verify that both $Q_{n}(\vect{T},\vect{f},Z,\varepsilon)$ and $P_{n}(\vect{T},\vect{f},Z,\varepsilon)$ are decreasing in $\varepsilon$, and  so are $\qrelow(\vect{T},\vect{f},Z,\varepsilon)$, $\qreup(\vect{T},\vect{f},Z,\varepsilon)$, $\prelow(\vect{T},\vect{f},Z,\varepsilon)$, and $\preup(\vect{T},\vect{f},Z,\varepsilon)$.
      \begin{prop}\label{prop:QPe}
        If $\varepsilon_{1}<\varepsilon_{2}$, then
        $$
        \qrelow(\vect{T},\vect{f},Z,\varepsilon_{1}) \geq \qrelow(\vect{T},\vect{f},Z,\varepsilon_{2}),
        \quad
        \qreup(\vect{T},\vect{f},Z,\varepsilon_{1}) \geq \qreup(\vect{T},\vect{f},Z,\varepsilon_{2}),
        $$
        $$
        \prelow(\vect{T},\vect{f},Z,\varepsilon_{1}) \geq \prelow(\vect{T},\vect{f},Z,\varepsilon_{2}),
        \quad
        \preup(\vect{T},\vect{f},Z,\varepsilon_{1}) \geq \preup(\vect{T},\vect{f},Z,\varepsilon_{2}).
        $$
      \end{prop}

      Therefore, the following limits exist,
      \begin{equation}\label{eq:defqre}
        \qrelow(\vect{T},\vect{f},Z) = \lim_{\varepsilon \to 0}\qrelow(\vect{T},\vect{f},Z,\varepsilon),
        \quad
        \qreup(\vect{T},\vect{f},Z) = \lim_{\varepsilon \to 0}\qreup(\vect{T},\vect{f},Z,\varepsilon),
      \end{equation}
      \begin{equation}\label{eq:defpre}
        \prelow(\vect{T},\vect{f},Z) = \lim_{\varepsilon \to 0}\prelow(\vect{T},\vect{f},Z,\varepsilon),
        \quad
        \preup(\vect{T},\vect{f},Z) = \lim_{\varepsilon \to 0}\preup(\vect{T},\vect{f},Z,\varepsilon).
      \end{equation}

      \begin{defn}
        Given  a subset $Z \subset X_{0}$ and $\vect{f} \in \vect{C}(\vect{X},\R)$, we call
        $\qrelow(\vect{T},\vect{f},Z)$ and $\qreup(\vect{T},\vect{f},Z)$ the \emph{lower} and \emph{upper spanning topological pressures of $\vect{T}$ for $\vect{f}$ on $Z$}, respectively;
        and we call $\prelow(\vect{T},\vect{f},Z)$ and $\preup(\vect{T},\vect{f},Z)$ the \emph{lower} and \emph{upper separated topological pressures of $\vect{T}$ for $\vect{f}$ on $Z$}, respectively.

       If  $\qrelow(\vect{T},\vect{f},Z)=\prelow(\vect{T},\vect{f},Z)$, we call it  the \emph{lower topological pressure of $\vect{T}$ for $\vect{f}$ on $Z$} and denote it  by $\preL(\vect{T},\vect{f},Z)$.
Similarly, if   $\qreup(\vect{T},\vect{f},Z)=\preup(\vect{T},\vect{f},Z)$, we call it the \emph{  upper topological pressure of $\vect{T}$ for $\vect{f}$ on $Z$} and denote it by $\preU(\vect{T},\vect{f},Z)$.

      \end{defn}

     Different to TDS $(X, T)$ and NDS $(X, \vect{T})$,  we require mild conditions on the potential $\vect{f}$ to obtain the equivalence of the corresponding pressures from $(n,\varepsilon)$-spanning and $(n,\varepsilon)$-separated sets.
     See \cite{Kong&Cheng&Li2015} for the case in NDSs $(X, \vect{T})$ with an identical space, and see \cite[Thm.1.1]{Walters1975} or \cite[\S9.1 Rmk.(12)]{Walters1982} for the case in TDSs $(X, T)$.
      \begin{prop}\label{prop:PeqQ}
        Given  a subset $Z \subset X_{0}$, let $\vect{f} \in \vect{C}(\vect{X},\R)$. Then
        $$
        \qrelow(\vect{T},\vect{f},Z) \leq \prelow(\vect{T},\vect{f},Z)
        \quad \text{and} \quad
        \qreup(\vect{T},\vect{f},Z) \leq \preup(\vect{T},\vect{f},Z).
        $$
       If $\vect{f} \in \vect{C}(\vect{X},\R)$ is equicontinuous, then $\preL(\vect{T},\vect{f},Z)$ and $\preU(\vect{T},\vect{f},Z)$ exist, i.e.,
        $$
       \preL(\vect{T},\vect{f},Z)= \qrelow(\vect{T},\vect{f},Z) = \prelow(\vect{T},\vect{f},Z);
        \quad
        \preU(\vect{T},\vect{f},Z)=\qreup(\vect{T},\vect{f},Z) = \preup(\vect{T},\vect{f},Z).
        $$
      \end{prop}
      \begin{proof}[Proof of Proposition~\ref{prop:PeqQ}]
        It is straightforward that
        $$
        \qrelow(\vect{T},\vect{f},Z) \leq \prelow(\vect{T},\vect{f},Z)
        \quad\text{and}\quad
        \qreup(\vect{T},\vect{f},Z) \leq \preup(\vect{T},\vect{f},Z).
        $$
        Suppose that  $\vect{f} \in \vect{C}(\vect{X},\R)$ is equicontinuous. By \eqref{eq:Qe}, \eqref{eq:Pe}, \eqref{eq:defqre}, and \eqref{eq:defpre},
        it suffices to show that for all $\alpha>0$,
        \begin{equation}\label{eq:PleqnaQ}
          P_{n}(\vect{T},\vect{f},Z,\varepsilon) \leq \e^{n\alpha}Q_{n}(\vect{T},\vect{f},Z,\frac{\varepsilon}{2})
        \end{equation}
        for all sufficiently small $\varepsilon$.

        For each given $\alpha>0$, by the equicontinuity of $\vect{f}$, there exists $\varepsilon_{0}>0$ such that for all $\varepsilon$ with $0<\varepsilon\leq\varepsilon_{0}$, all integers $k \in \N$ and all points $x$ and $y \in X_{k}$ with $d_{X_{k}}(x,y)<\frac{\varepsilon}{2}$,
        we have $\lvert f_{k}(x)-f_{k}(y) \rvert<\alpha$. It implies that  for each $n \geq 1$ and all  $x$ and $y \in X_{0}$ satisfying $d_{n}^{\vect{T}}(x,y)<\frac{\varepsilon}{2}$,
        $$
        \left\vert S_{n}^{\vect{T}}\vect{f}(x)-S_{n}^{\vect{T}}\vect{f}(y) \right\vert<n\alpha.
        $$

Arbitrarily choose  a $(n,\varepsilon)$-separated set $E$  and  a $(n,\frac{\varepsilon}{2})$-spanning set $F$ for $Z$.
        Note that $F$ is also a $(n,\varepsilon)$ spanning set of  $Z$ and $E \subset Z$. We define a mapping $\phi: E \to F$ by choosing  a point $\phi(x) \in F$ with $d_{n}^{\vect{T}}(x,\phi(x))<\frac{\varepsilon}{2}$ for each $x \in E$.
        The mapping $\phi$ is injective since $E$ is $(n,\varepsilon)$ separated. It follows that
        \begin{align*}
          \sum_{y \in F}\e^{S_{n}^{\vect{T}}\vect{f}(y)} & \geq \sum_{y \in \phi(E)}\e^{S_{n}^{\vect{T}}\vect{f}(y)}  = \sum_{x \in E}\e^{S_{n}^{\vect{T}}\vect{f}(\phi(x))} \geq \e^{-n\alpha}\sum_{x \in E}\e^{S_{n}^{\vect{T}}\vect{f}(x)},
        \end{align*}
        and by equations \eqref{eq:defQn} and \eqref{eq:defPn}, we have $Q_{n}(\vect{T},\vect{f},Z,\frac{\varepsilon}{2}) \geq \e^{-n\alpha}P_{n}(\vect{T},\vect{f},Z,\varepsilon)$.
        This implies \eqref{eq:PleqnaQ}, and we complete the proof.
      \end{proof}

Since $\vect{f}=\vect{0}$ is equicontinuous,  we are able to define the corresponding entropies of lower and upper topological pressures.
      \begin{defn}
        We call
        $$
        \enttoplow(\vect{T},Z)=\preL(\vect{T},\vect{0},Z). \qquad \textit{and \qquad}
        \enttopup(\vect{T},Z)=\preU(\vect{T},\vect{0},Z)
        $$
        the \emph{lower} and \emph{upper topological entropy of $\vect{T}$ on $Z$}, respectively.
      \end{defn}

    \subsection{Bowen pressures and entropies}\label{ssect:prebppetentbow}
      In \cite{Pesin&Pitskel1984}, Pesin and Pitskel' used an approach similar to Hausdorff measures and dimensions to define a type of pressure on topological dynamical systems, and they studied the relevant variational principle.
      Inspired by their work, we define a type of topological pressures on nonautonomous dynamical systems by constructing certain Hausdorff measures with Bowen balls; see \cite{Falconer2014,Rogers1998} for the  details of Hausdorff measures.

      Given  a subset $Z \subset X_{0}$, real $N>0$ and  real $\varepsilon>0$, we say that a collection $\{B_{n_{i}}^{\vect{T}}(x_{i},\varepsilon)\}_{i\in \mathcal{I}}$ of Bowen balls is a \emph{$(N,\varepsilon)$-cover} of $Z$
      if $\bigcup_{i\in \mathcal{I}}B_{n_{i}}^{\vect{T}}(x_{i},\varepsilon) \supset Z$ where $n_{i} \geq N$ for each $i \in \mathcal{I}$.

      Given $\vect{f} \in \vect{C}(\vect{X},\mathbb{R})$ and $s \in \mathbb{R}$, for  reals $N >0$ and $\varepsilon>0$, we define
      \begin{equation}\label{eq:defmsrbow}
        \msrbow_{N,\varepsilon}^{s}(\vect{T},\vect{f},Z)=\inf\Big\{\sum_{i=1}^{\infty} \exp{\big(-n_{i}s+S_{n_{i}}^{\vect{T}}\vect{f}(x_{i})\big)}  \Big\},
      \end{equation}
      where the infimum is taken over all countable $(N,\varepsilon)$-covers $\{B_{n_{i}}^{\vect{T}}(x_{i},\varepsilon)\}_{i=1}^\infty$ of $Z$.
      Note that $\msrbow_{N,\varepsilon}^{s}(\vect{T},\vect{f}, \cdot)$ is an outer measure on $X_0$ constructed by Method I in \cite{Rogers1998}.

      Since $\msrbow_{N,\varepsilon}^{s}(\vect{T},\vect{f},Z)$ increases as $N$ tends to $\infty$ for every given $Z \subset X_{0}$, we write
      $$
      \msrbow_{\varepsilon}^{s}(\vect{T},\vect{f},Z)=\lim_{N \to \infty} \msrbow_{N,\varepsilon}^{s}(\vect{T},\vect{f},Z).
      $$
      Note that if $t > s$, then $\msrbow_{\varepsilon}^{t}(\vect{T},\vect{f},Z)=0$ whenever $\msrbow_{\varepsilon}^{s}(\vect{T},\vect{f},Z)<\infty$. Thus, there is a critical value of $s$ at which $\msrbow_{\varepsilon}^{s}(\vect{T},\vect{f},Z)$ `jumps' from $\infty$ to $0$. Formally, the critical value is denoted by
      \begin{equation}\label{def_PBep}
        \prebow(\vect{T},\vect{f},Z,\varepsilon)=\inf\{s: \msrbow_{\varepsilon}^{s}(\vect{T},\vect{f},Z)=0\}
                                                =\sup\{s: \msrbow_{\varepsilon}^{s}(\vect{T},\vect{f},Z)=+\infty\}.
      \end{equation}

      The following facts are direct consequences of the definitions.
      \begin{prop}\label{prop:msrbowepsilon}
        Given $\vect{f} \in \vect{C}(\vect{X},\mathbb{R})$, $Z \subset X_{0}$ and $s \in \mathbb{R}$, for $\varepsilon_{1}>\varepsilon_{2}>0$, we have that
        $$
        \msrbow_{N,\varepsilon_{1}}^{s}(\vect{T},\vect{f},Z) \leq \msrbow_{N,\varepsilon_{2}}^{s}(\vect{T},\vect{f},Z)
        $$
        for all $N >0$, and
        $$
        \msrbow_{\varepsilon_{1}}^{s}(\vect{T},\vect{f},Z) \leq \msrbow_{\varepsilon_{2}}^{s}(\vect{T},\vect{f},Z).
        $$
        Moreover
        $$
        \prebow(\vect{T},\vect{f},Z,\varepsilon_{1}) \leq \prebow(\vect{T},\vect{f},Z,\varepsilon_{2}).
        $$
      \end{prop}

      By the monotonicity of $\prebow(\vect{T},\vect{f},Z,\varepsilon)$ with respect to  $\varepsilon$, we are able to define the Bowen type of topological pressures on nonautonomous dynamical systems.
      \begin{defn}\label{def:prebpp}
        Given  $\vect{f} \in \vect{C}(\vect{X},\mathbb{R})$  and $Z \subset X_{0}$, we call
        $$
        \prebow(\vect{T},\vect{f},Z)=\lim_{\varepsilon \to 0}\prebow(\vect{T},\vect{f},Z,\varepsilon)
        $$
        the \emph{Bowen-Pesin-Pitskel' topological pressure} (\emph{Bowen pressure} for short) \emph{of $\vect{T}$ for $\vect{f}$ on $Z$}. We call
        $$
        \enttopbow(\vect{T},Z)=\prebow(\vect{T},\vect{0},Z)
        $$
        the \emph{Bowen topological entropy} (\emph{Bowen entropy} for short)  \emph{of $\vect{T}$  on $Z$}.
      \end{defn}
      We end this subsection with a useful remark.
      \begin{rmk}
        It is equivalent to use closed Bowen balls for $(N,\varepsilon)$-covers in \eqref{eq:defmsrbow}.
      \end{rmk}

    \subsection{Packing  pressures and entropies}\label{ssect:prepacetentpac}
      Coverings and packings play a dual role in many areas of mathematics.
      Bowen pressures  are defined in terms of Hausdorff measures by using covers, and there is another typical measure construction by using packings, namely, packing measure,
      that is in a sense `dual' to Hausdorff measure. Hausdorff and  packing measures and dimensions are fundamental concepts in fractal geometry \cite{Falconer2014,Tricot1982}.
      Thus, it is natural to look for a type of pressure defined in terms of `packings' by large collections of disjoint Bowen balls of large iterations and small
      radii with centres in the set under consideration. Such analogues of entropies and pressures have been considered on TDSs; see \cite{Feng&Huang2012,Wang&Chen2012,Zhong&Chen2023}. In this subsection, we extend these notions to nonautonomous dynamical systems.

      Given  a subset $Z \subset X_{0}$, we say that a collection $\{\overline{B}_{n_{i}}^{\vect{T}}(x_{i},\varepsilon)\}_{i\in \mathcal{I}}$ of closed Bowen balls is a \textit{$(N,\varepsilon)$-packing of $Z$} if $\{\overline{B}_{n_{i}}^{\vect{T}}(x_{i},\varepsilon)\}_{i\in \mathcal{I}}$ is disjoint where $x_{i} \in Z$ and $n_{i} \geq N$ for all $i \in \mathcal{I}$.
      Given $\vect{f} \in \vect{C}(\vect{X},\R)$ and $s \in \mathbb{R}$, for each $N>0$ and $\varepsilon>0$, we define
      \begin{equation}\label{eq:defmsrPack}
        \msrpac_{N,\varepsilon}^{s}(\vect{T},\vect{f},Z)=\sup\Big\{\sum_{i=1}^{\infty}\exp{\left(-n_{i}s+S_{n_{i}}^{\vect{T}}\vect{f}(x_{i})\right)} \Big\},
      \end{equation}
      where the supremum is taken over all countable $(N,\varepsilon)$-packings $\{\overline{B}_{n_{i}}^{\vect{T}}(x_{i},\varepsilon)\}_{i=1}^\infty$ of $Z$. Since $\msrpac_{N,\varepsilon}^{s}(\vect{T},\vect{f},Z)$ is non-increasing as $N$ tends to $\infty$,
      we write
      $$
      \msrpac_{\infty,\varepsilon}^{s}(\vect{T},\vect{f},Z)=\lim_{N \to \infty} \msrpac_{N,\varepsilon}^{s}(\vect{T},\vect{f},Z).
      $$
      Note that $\msrpac_{\infty,\varepsilon}^{s}$ is not a measure, of which the problem is similar to that encountered with the classic packing measures; see \cite{Falconer2014}. Hence, we modify the definition by decomposing $Z$ into a countable collection of sets and define
      \begin{equation}\label{def_pes}
        \msrpac_{\varepsilon}^{s}(\vect{T},\vect{f},Z)=\inf\Big\{\sum_{i=1}^{\infty}\msrpac_{\infty,\varepsilon}^{s}(\vect{T},\vect{f},Z_{i}): \bigcup_{i=1}^{\infty}Z_{i} \supset Z\Big\}.
      \end{equation}
      Similarly, we denote the jump value of $s$  by
      \begin{equation}\label{def_PP}
        \prepac(\vect{T},\vect{f},Z,\varepsilon) =\inf\{s:\msrpac_{\varepsilon}^{s}(\vect{T},\vect{f},Z)=0\}
                                =\sup\{s: \msrpac_{\varepsilon}^{s}(\vect{T},\vect{f},Z)=+\infty\}.
      \end{equation}

      The following monotone properties are straightforward from the definitions.
      \begin{prop}\label{prop:msrpacepsilon}
        Given $\vect{f} \in \vect{C}(\vect{X},\mathbb{R})$, $Z \subset X_{0}$ and $s \in \mathbb{R}$, for $\varepsilon_{1}>\varepsilon_{2}>0$, we have that
        $$\msrpac_{N,\varepsilon_{1}}^{s}(\vect{T},\vect{f},Z) \leq \msrpac_{N,\varepsilon_{2}}^{s}(\vect{T},\vect{f},Z),$$
        for all $N >0$, and
        $$\msrpac_{\varepsilon_{1}}^{s}(\vect{T},\vect{f},Z) \leq \msrpac_{\varepsilon_{2}}^{s}(\vect{T},\vect{f},Z).$$
        Moreover
        $$
        \prepac(\vect{T},\vect{f},Z,\varepsilon_{1}) \leq \prepac(\vect{T},\vect{f},Z,\varepsilon_{2}).
        $$
      \end{prop}
      Since  $\prepac(\vect{T},\vect{f},Z,\varepsilon)$ is monotone with respect to  $\varepsilon$, we define the packing pressure as follows.
      \begin{defn}\label{def:prepac}
        Given  $\vect{f} \in \vect{C}(\vect{X},\mathbb{R})$ and $Z \subset X_{0}$, we define the \emph{packing topological pressure} (\emph{packing pressure} for short) \emph{of $\vect{T}$ for $\vect{f}$ on  $Z$} by
        $$\prepac(\vect{T},\vect{f},Z)=\lim_{\varepsilon \to 0}\prepac(\vect{T},\vect{f},Z,\varepsilon).$$
        We define the \emph{packing topological entropy} (\emph{packing entropy} for short) \emph{of $\vect{T}$ on $Z$} by
        $$\enttoppac(\vect{T},Z)=\prepac(\vect{T},\vect{0},Z).$$
      \end{defn}

  \section{Equivalent Definitions of Pressures and entropies}\label{sect:EquivDefs}
In this section, we study the equivalent definitions of pressures and entropies, and these definitions and their properties are important to investigate the variational principles of pressures and entropies.
See \cite{CM2} for details.
    \subsection{Upper and lower spanning and separated pressures as capacities}
      In this subsection, we provide an alternative formulation of upper and lower pressures following the general Carath\'{e}odory dimension structures (see \cite{Pesin1997}).
      First, we give a simple fact which is  frequently used in our proofs.
\begin{fact}
 Given an NDS $(\vect{X},\vect{T})$ and a subset $Z \subset X_{0}$,
      \begin{enumerate}
        \item a set $F \subset X_{0}$ $(n,\varepsilon)$ spans $Z$ iff the family $\{\overline{B}_{n}^{\vect{T}}(x,\varepsilon)\}_{x \in F}$ is a cover of $Z$;
        \item a set $E \subset Z$ is $(n,\varepsilon)$ separated iff the family $\{\overline{B}_{n}^{\vect{T}}(x,\frac{\varepsilon}{2})\}_{x \in E}$ is disjoint.
      \end{enumerate}
 \end{fact}

Given  $\vect{f} \in \vect{C}(\vect{X},\R)$ and  $s \in \R$, for every integer $n \geq 1$ and real $\varepsilon>0$, we define
      \begin{equation}\label{eq:defcarpes}
        \carpes_{n,\varepsilon}^{s}(\vect{T},\vect{f},Z) = \inf\Big\{\sum_{i=1}^{\infty}\exp{\left(-ns+S_{n}^{\vect{T}}\vect{f}(x_{i})\right)}: \bigcup_{i=1}^{\infty}\overline{B}_{n}^{\vect{T}}(x_{i},\varepsilon) \supset Z\Big\},
      \end{equation}
      and
      \begin{equation}\label{eq:defcarpep}
        \begin{split}
        \carpep_{n,\varepsilon}^{s}(\vect{T},\vect{f},Z) = \sup\Big\{&\sum_{i=1}^{\infty}\exp{\left(-ns+S_{n}^{\vect{T}}\vect{f}(x_{i})\right)}: \\
                                                                    &\{\overline{B}_{n}^{\vect{T}}(x_{i},\varepsilon)\}_{i=1}^{\infty}\ \text{is disjoint and}\ x_{i} \in Z\Big\}.
        \end{split}
      \end{equation}
      The Carath\'{e}odory functions are given by
      $$
      \carpeslow_{\varepsilon}^{s}(\vect{T},\vect{f},Z) = \varliminf_{n \to \infty}\carpes_{n,\varepsilon}^{s}(\vect{T},\vect{f},Z),
      \qquad
      \carpesup_{\varepsilon}^{s}(\vect{T},\vect{f},Z) = \varlimsup_{n \to \infty}\carpes_{n,\varepsilon}^{s}(\vect{T},\vect{f},Z),
      $$
      $$
      \carpeplow_{\varepsilon}^{s}(\vect{T},\vect{f},Z) = \varliminf_{n \to \infty}\carpep_{n,\varepsilon}^{s}(\vect{T},\vect{f},Z),
      \qquad
      \carpepup_{\varepsilon}^{s}(\vect{T},\vect{f},Z) = \varlimsup_{n \to \infty}\carpep_{n,\varepsilon}^{s}(\vect{T},\vect{f},Z).
      $$

We have the following alternative formulae for the upper and lower pressures.
      \begin{prop}\label{prop:prepesincap}
        For all $\vect{f} \in \vect{C}(\vect{X},\R)$, we have
        \begin{align*}
          \qrelow(\vect{T},\vect{f},Z) &= \lim_{\varepsilon \to 0}\sup\{s:\carpeslow_{\varepsilon}^{s}(\vect{T},\vect{f},Z)=+\infty\} = \lim_{\varepsilon \to 0}\inf\{s:\carpeslow_{\varepsilon}^{s}(\vect{T},\vect{f},Z)=0\}, \\
          \qreup(\vect{T},\vect{f},Z) &= \lim_{\varepsilon \to 0}\sup\{s:\carpesup_{\varepsilon}^{s}(\vect{T},\vect{f},Z)=+\infty\} = \lim_{\varepsilon \to 0}\inf\{s:\carpesup_{\varepsilon}^{s}(\vect{T},\vect{f},Z)=0\},
        \end{align*}
        and
        \begin{align*}
            \prelow(\vect{T},\vect{f},Z) &= \lim_{\varepsilon \to 0}\sup\{s:\carpeplow_{\varepsilon}^{s}(\vect{T},\vect{f},Z)=+\infty\} = \lim_{\varepsilon \to 0}\inf\{s:\carpeplow_{\varepsilon}^{s}(\vect{T},\vect{f},Z)=0\}, \\
            \preup(\vect{T},\vect{f},Z) &= \lim_{\varepsilon \to 0}\sup\{s:\carpepup_{\varepsilon}^{s}(\vect{T},\vect{f},Z)=+\infty\} = \lim_{\varepsilon \to 0}\inf\{s:\carpepup_{\varepsilon}^{s}(\vect{T},\vect{f},Z)=0\}.
        \end{align*}
      \end{prop}
      \begin{proof}
       Dimension structures are given by similar arguments as in \cite{Pesin1997}.
      Since the arguments are similar, we only prove the first  equality. Given $\vect{f} \in \vect{C}(\vect{X},\R)$, by \eqref{eq:defqre} and \eqref{eq:defpre}, it suffices to prove that for all $\varepsilon>0$,
        \begin{equation}\label{eq:Qepes}
            \qrelow(\vect{T},\vect{f},Z,\varepsilon) = \sup\{s:\carpeslow_{\varepsilon}^{s}(\vect{T},\vect{f},Z)=+\infty\} = \inf\{s:\carpeslow_{\varepsilon}^{s}(\vect{T},\vect{f},Z)=0\}.
        \end{equation}

        Note that by Fact (1), it is clear that for all $s \in \R$ and integers $n \geq 1$,
        \begin{equation} \label{Eq_Dneqn}
        \carpes_{n,\varepsilon}^{s}(\vect{T},\vect{f},Z)=\e^{-ns}Q_{n}(\vect{T},\vect{f},Z,\varepsilon).
        \end{equation}
        Let
$$
\underline{r}= \sup\{s:\carpeslow_{\varepsilon}^{s}(\vect{T},\vect{f},Z)=+\infty\} = \inf\{s:\carpeslow_{\varepsilon}^{s}(\vect{T},\vect{f},Z)=0\}.
$$

        We first show that $\underline{r} \leq \qrelow(\vect{T},\vect{f},Z,\varepsilon)$.
        For each $s<\underline{r}$, we have $\carpeslow_{\varepsilon}^{s}(\vect{T},\vect{f},Z)=+\infty$.
        By the definition of $\carpeslow_{\varepsilon}^{s}$, for all strictly increasing sequences $\{n_{l}\}_{l=1}^{\infty}$ of positive integers $n_{l}$,
        we have that $\carpes_{n_{l},\varepsilon}^{s}(\vect{T},\vect{f},Z)$ tends to $\infty$ as $l \to \infty$.
        Hence
        $$
        \frac{1}{n_{l}}\log{Q_{n_{l}}(\vect{T},\vect{f},Z,\varepsilon)}-s \to \infty
        $$
        as $l \to \infty$, and we have
        $$
        \varliminf_{n \to \infty} \frac{1}{n}\log{Q_{n}(\vect{T},\vect{f},Z,\varepsilon)}>s,
        $$
        which implies $\qrelow(\vect{T},\vect{f},Z,\varepsilon) \geq s$.
        It follows that $\underline{r} \leq \qrelow(\vect{T},\vect{f},Z,\varepsilon)$.

        Next, we show that $\underline{r} \geq \qrelow(\vect{T},\vect{f},Z,\varepsilon)$.        For each $s>\underline{r}$, we have $\carpeslow_{\varepsilon}^{s}(\vect{T},\vect{f},Z)=0$.
Similarly, by Fact (1) and \eqref{Eq_Dneqn}, there exists a subsequence $\carpes_{n_{l},\varepsilon}^{s}(\vect{T},\vect{f},Z)$ which tends to $0$ as $l \to \infty$.
        Consequently,
        $$
        \frac{1}{n_{l}}\log{Q_{n_{l}}(\vect{T},\vect{f},Z,\varepsilon)}-s \to -\infty
        $$
        as $l \to \infty$, and hence
        $$
        \varliminf_{n \to \infty}\frac{1}{n}\log{Q_{n}(\vect{T},\vect{f},Z,\varepsilon)}<s.
        $$
        This implies that $\qrelow(\vect{T},\vect{f},Z,\varepsilon) \leq s$, and therefore $\qrelow(\vect{T},\vect{f},Z,\varepsilon) \leq \underline{r}$.

      \end{proof}

      Inspired by the relation between upper box dimension and packing contents in fractal geometry, we obtain the following proposition which provides an equivalent definition of upper separated pressure $\preup$ via the packing content $\msrpac_{\infty,\varepsilon}^{s}$; see Tricot \cite[Cor.2]{Tricot1982} and Falconer \cite[Lem.3.8]{Falconer2014} for the analogous result in fractal geometry.
      \begin{prop}\label{prop:preupeqpac}
        Given $\vect{f} \in \vect{C}(\vect{X},\R)$ and $Z \subset X_{0}$, we have
        $$
        \preup(\vect{T},\vect{f},Z) = \lim_{\varepsilon\to0}\inf\{s: \msrpac_{\infty,\varepsilon}^{s}(\vect{T},\vect{f},Z)=0\} = \lim_{\varepsilon\to0}\sup\{s: \msrpac_{\infty,\varepsilon}^{s}(\vect{T},\vect{f},Z)=+\infty\}.
        $$
      \end{prop}
      \begin{proof}
We write
        $$
        s_0 = \inf\{s: \msrpac_{\infty,\varepsilon}^{s}(\vect{T},\vect{f},Z)=0\} = \sup\{s: \msrpac_{\infty,\varepsilon}^{s}(\vect{T},\vect{f},Z)=+\infty\}.
        $$
        By \eqref{eq:defpre} and Proposition \ref{prop:prepesincap}, it suffices to show that
        $$
        \preup(\vect{T},\vect{f},Z,\varepsilon) = s_0.
        $$

        Note that for all integral $n \geq 1$ and real $\varepsilon>0$, a disjoint family $\{\overline{B}_{n}^{\vect{T}}(x_{i},\varepsilon)\}_{i=1}^{\infty}$ with $x_{i} \in Z$ is also a $(n,\varepsilon)$-packing of $Z$, and by \eqref{eq:defcarpep}, it is clear that
$$
\carpep_{n,\varepsilon}^{s}(\vect{T},\vect{f},Z) \leq \msrpac_{n,\varepsilon}^{s}(\vect{T},\vect{f},Z).
$$
        It follows that $\preup(\vect{T},\vect{f},Z) \leq s_{0}$.

        It remains to show  $\preup(\vect{T},\vect{f},Z) \geq s_{0}$.
        If $s_{0}=-\infty$, the result is obvious. Otherwise, choose any two real numbers $s$ and $t$ with $s<t<s_{0}$.
        By the definition of $s_{0}$, $\msrpac_{\infty,\varepsilon}^{t}(\vect{T},\vect{f},Z)=+\infty$.
        Thus, by the definition of $\msrpac_{\infty,\varepsilon}^{t}$ and \eqref{eq:defmsrPack}, there exists an integer $N_{0} \geq 1$ such that for each $N \geq N_{0}$, there exists a $(N,\varepsilon)$-packing $\{\overline{B}_{n_{i}}^{\vect{T}}(x_{i},\varepsilon)\}_{i=1}^{\infty}$ of $Z$
        such that
        $$
        \sum_{i=1}^{\infty}\exp{\Big(-n_{i}t+S_{n_{i}}^{\vect{T}}\vect{f}(x_{i})\Big)} > 1.
        $$
        Let $\mathcal{I}_{n}=\{i: n_{i}=n\}$ for each $n \geq N$. Then
        $$
        \sum_{n=1}^{\infty}\sum_{i \in \mathcal{I}_{n}}\exp{\Big(-n_{i}t+S_{n_{i}}^{\vect{T}}\vect{f}(x_{i})\Big)} > 1.
        $$
        This implies that there exists some integer $n \geq N$ such that
        $$
        \sum_{i \in \mathcal{I}_{n}}\exp{\Big(-nt+S_{n}^{\vect{T}}\vect{f}(x_{i})\Big)} \geq \e^{-n(t-s)}(1-\e^{s-t}).
        $$

Meanwhile, since   $\{\overline{B}_{n}^{\vect{T}}(x_{i},\varepsilon)\}_{i \in \mathcal{I}_{n}}$ with $x_{i} \in Z$ is disjoint,
        it follows by \eqref{eq:defPn} that
        $$
        \e^{-nt}P_{n}(\vect{T},\vect{f},Z,2\varepsilon) \geq \sum_{i \in \mathcal{I}_{n}}\exp{\Big(-nt+S_{n}^{\vect{T}}\vect{f}(x_{i})\Big)} \geq \e^{-n(t-s)}(1-\e^{s-t}).
        $$
        Hence, for every $N \geq N_{0}$, there exists an integer $n \geq N$ such that
        $$
        P_{n}(\vect{T},\vect{f},Z,2\varepsilon) \geq \e^{ns}(1-\e^{s-t}).
        $$
        By \eqref{eq:Pe}, it follows that $ \preup(\vect{T},\vect{f},Z,2\varepsilon) \geq s$  for all $\varepsilon>0$,  and  we have  $\preup(\vect{T},\vect{f},Z) \geq s$.
        Since this is true for all $-\infty<s<s_{0}$,   we conclude that $\preup(\vect{T},\vect{f},Z) \geq s_{0}$.
      \end{proof}

    \subsection{Upper and lower pressures via open covers}\label{subsec_PC}
    In this subsection, we  give another method to define upper and lower  pressures  by open covers.

    Given a sequence $\vect{\mathscr{U}}=\{\mathscr{U}_{k}\}_{k=0}^{\infty}$ of open covers $\mathscr{U}_{k}$ of $X_{k}$,
    for all integers $k \geq 0$ and $n \geq 1$,
    we write $\vect{\mathscr{U}}_{k}^{n}$ for the set of all strings $\mathbf{U}$ of length $n=\len{\mathbf{U}}$ at level $k$,
    i.e.,
    $$
    \vect{\mathscr{U}}_{k}^{n} = \{\mathbf{U}=U_{k}U_{k+1} \cdots U_{k+n-1}: U_{j} \in \mathscr{U}_{j}, j = k,\ldots,k+n-1\},
    $$
    and for every $\mathbf{U} \in \vect{\mathscr{U}}_{k}^{n}$,
    \begin{equation}\label{def_XkbU}
    X_{k}[\mathbf{U}] = \bigcap_{j=0}^{n-1}\vect{T}_{k}^{-j}U_{k+j}
    = \{x \in X_{k}: \vect{T}_{k}^{j}x \in U_{k+j}, j=0,\dots,n-1\}.
    \end{equation}
    Write
    $$
    \vee_{k,n}^{\vect{T}}\vect{\mathscr{U}} = \bigvee_{j=0}^{n-1}\vect{T}_{k}^{-j}\mathscr{U}_{j+k} = \{X_{k}[\mathbf{U}]: \mathbf{U} \in \vect{\mathscr{U}}_{k}^{n}\}.
    $$
    We say that $\Gamma \subset \vect{\mathscr{U}}_{0}^{n}$ \emph{covers} $Z \subset X_{0}$ if $Z \subset \bigcup_{\mathbf{U} \in \Gamma}X_{0}[\mathbf{U}]$.
    For brevity, we write $\vee_{n}^{\vect{T}}\vect{\mathscr{U}}=\vee_{0,n}^{\vect{T}}\vect{\mathscr{U}}$.

    For every $\mathbf{U} \in \vect{\mathscr{U}}_{k}^{n}$, we write
    \begin{equation}
      \underline{S}_{k,n}^{\vect{T}}\vect{f}(\mathbf{U})=\inf_{x \in X_{k}[\mathbf{U}]}S_{k,n}^{\vect{T}}\vect{f}(x),
      \quad\text{and}\quad
      \overline{S}_{k,n}^{\vect{T}}\vect{f}(\mathbf{U})=\sup_{x \in X_{k}[\mathbf{U}]}S_{k,n}^{\vect{T}}\vect{f}(x);
    \end{equation}
   where $\underline{S}_{k,n}^{\vect{T}}\vect{f}(\mathbf{U})=\overline{S}_{k,n}^{\vect{T}}\vect{f}(\mathbf{U})=-\infty$ if $X_{k}[\mathbf{U}] = \emptyset$. We write $\underline{S}_{n}^{\vect{T}}\vect{f}(\mathbf{U})=\underline{S}_{0,n}^{\vect{T}}\vect{f}(\mathbf{U})$ and $ \overline{S}_{n}^{\vect{T}}\vect{f}(\mathbf{U})= \overline{S}_{0,n}^{\vect{T}}\vect{f}(\mathbf{U})$ for simplicity.
    Similar to \eqref{eq:defQn} and \eqref{eq:defPn}, we define
    \begin{equation}\label{eq:QnPncov}
      \begin{split}
        Q_{n}(\vect{T},\vect{f},Z,\vect{\mathscr{U}}) &= \inf\Big\{\sum_{\mathbf{U} \in \Gamma}\exp{(\underline{S}_{n}^{\vect{T}}\vect{f}(\mathbf{U}))}: \Gamma \subset \vect{\mathscr{U}}_{0}^{n}\ \text{covers}\ Z\Big\}, \\
        P_{n}(\vect{T},\vect{f},Z,\vect{\mathscr{U}}) &= \inf\Big\{\sum_{\mathbf{U} \in \Gamma}\exp{(\overline{S}_{n}^{\vect{T}}\vect{f}(\mathbf{U}))}: \Gamma \subset \vect{\mathscr{U}}_{0}^{n}\ \text{covers}\ Z\Big\}.
      \end{split}
    \end{equation}
Because of  the lack of subadditivity,  we  consider their lower and upper limits,
\begin{equation}\label{def_QQPPCV}
    \begin{split}
   &   \qrelow(\vect{T},\vect{f},Z,\vect{\mathscr{U}})=\varliminf_{n \to \infty}\frac{1}{n}\log{Q_{n}(\vect{T},\vect{f},Z,\vect{\mathscr{U}})}, \\
   &   \qreup(\vect{T},\vect{f},Z,\vect{\mathscr{U}})=\varlimsup_{n \to \infty}\frac{1}{n}\log{Q_{n}(\vect{T},\vect{f},Z,\vect{\mathscr{U}})},\\
   & \prelow(\vect{T},\vect{f},Z,\vect{\mathscr{U}})=\varliminf_{n \to \infty}\frac{1}{n}\log{P_{n}(\vect{T},\vect{f},Z,\vect{\mathscr{U}})}, \\
   & \preup(\vect{T},\vect{f},Z,\vect{\mathscr{U}})=\varlimsup_{n \to \infty}\frac{1}{n}\log{P_{n}(\vect{T},\vect{f},Z,\vect{\mathscr{U}})}.
    \end{split}
\end{equation}

    Given a sequence $\vect{\mathscr{U}}$ of open covers $\mathscr{U}_{k}$ of $X_{k}$,
    we write
    $$
    \diam(\vect{\mathscr{U}}) = \sup_{k \in \N}\sup\{\diam(U): U \in \mathscr{U}_{k}\}.
    $$
    If there exists $\delta>0$ such that  $\delta$ is a Lebesgue number for $\mathscr{U}_{k}$ for all integral $k \geq 0$, then we say $\delta$ is a \emph{Lebesgue number} for $\vect{\mathscr{U}}$.
    A sequence $\vect{\mathscr{U}}$ has a Lebesgue number iff
    $$
    \inf_{k \in \N}\sup\{\delta>0: \delta\ \text{is a Lebesgue number for}\ \mathscr{U}_{k}\} > 0.
    $$
    \begin{prop}\label{prop:qrecov}
      Given  $Z \subset X_{0}$, if $\vect{f} \in \vect{C}(\vect{X},\R)$ is equicontinuous, then
      $$
    \preL(\vect{T},\vect{f},Z) = \sup_{\vect{\mathscr{U}}}\qrelow(\vect{T},\vect{f},Z,\vect{\mathscr{U}}) = \lim_{\diam(\vect{\mathscr{U}}) \to 0}\qrelow(\vect{T},\vect{f},Z,\vect{\mathscr{U}})
      $$
      and
      $$
      \preU(\vect{T},\vect{f},Z) = \sup_{\vect{\mathscr{U}}}\qreup(\vect{T},\vect{f},Z,\vect{\mathscr{U}}) = \lim_{\diam(\vect{\mathscr{U}}) \to 0}\qreup(\vect{T},\vect{f},Z,\vect{\mathscr{U}}),
      $$
      where $\vect{\mathscr{U}}$ ranges over all sequences of open covers of $X_{k}$ with a Lebesgue number.
    \end{prop}
    \begin{proof}
       Let $\vect{\mathscr{U}}$ be  a sequence of open covers $\mathscr{U}_{k}$ of $X_k$ with a  Lebesgue number $\delta$.    We first show that
      \begin{equation}\label{eq:QnUleqQnd2}
        Q_{n}(\vect{T},\vect{f},Z,\vect{\mathscr{U}}) \leq Q_{n}(\vect{T},\vect{f},Z,\frac{\delta}{2}).
      \end{equation}

  Given  an integer $n>0$, for every  $(n,\frac{\delta}{2})$-spanning set $F$ for $Z$ with respect to $\vect{T}$, we have
      $$
      Z \subset \bigcup_{x \in F}\overline{B}_{n}^{\vect{T}}\Big(x,\frac{\delta}{2}\Big) = \bigcup_{x \in F}\bigcap_{j=0}^{n-1}\vect{T}^{-j}\overline{B}\Big(\vect{T}^{j}x,\frac{\delta}{2}\Big).
      $$
  For every $x \in F$, since $\overline{B}\big(\vect{T}^{j}x,\frac{\delta}{2}\big)$ is contained in a  member of $\mathscr{U}_{j}$ for every $j>0$, there exists  $\mathbf{U}_{x} \in \vect{\mathscr{U}}_{0}^{n}$ such that  $\overline{B}_{n}^{\vect{T}}\big(x,\frac{\delta}{2}\big) \subset X_{k}[\mathbf{U}_{x}]$.
      Hence $\Phi=\{\mathbf{U}_{x}\}_{x \in F} \subset \vect{\mathscr{U}}_{0}^{n}$ covers $Z$, and
      $$
      Q_{n}(\vect{T},\vect{f},Z,\vect{\mathscr{U}}) \leq \sum_{\mathbf{U} \in \Phi}\exp{(\underline{S}_{n}^{\vect{T}}\vect{f}(\mathbf{U}))} \leq \sum_{x \in F}\e^{S_{n}^{\vect{T}}\vect{f}(x)}
      $$
      for all $(n,\frac{\delta}{2})$-spanning sets $F$ for $Z$, which implies the inequality \eqref{eq:QnUleqQnd2}.

      Next, given $\alpha>0$ and  $n \geq 1$, if $\vect{f}$ is equicontinuous, we show that
      \begin{equation}\label{eq:QnUgeqPne}
        Q_{n}(\vect{T},\vect{f},Z,\vect{\mathscr{U}}) \geq \e^{-n\alpha}P_{n}(\vect{T},\vect{f},Z,\varepsilon)
      \end{equation}
      for all small $\varepsilon>0$ and all sequences $\vect{\mathscr{U}}$ of open covers of $X_{k}$ with $\diam(\vect{\mathscr{U}}) \leq \varepsilon$.

      Fix $\alpha>0$. By the equicontinuity of $\vect{f}$, there exists $\varepsilon_{0}$ such that for all $\varepsilon$ with $0<\varepsilon<\varepsilon_{0}$,
      all integers $k \in \N$, and all points $x$ and $y \in X_{k}$ with $d_{X_{k}}(x,y)<\frac{\varepsilon}{2}$,
      we have $\lvert f_{k}(x_{k})-f_{k}(y_{k}) \rvert<\alpha$.  This implies that for each $n \geq 1$,
      $$
      \left\vert S_{n}^{\vect{T}}\vect{f}(x)-S_{n}^{\vect{T}}\vect{f}(y) \right\vert<n\alpha
      $$
      whenever $d_{n}^{\vect{T}}(x,y)<\frac{\varepsilon}{2}$, where $x$ and $y \in X_{0}$.

      Choose a sequence $\vect{\mathscr{U}}$ of open covers of $X_{k}$ with $\diam(\vect{\mathscr{U}}) \leq \varepsilon$, and let $E \subset Z$ be $(n,\varepsilon)$ separated with respect to $\vect{T}$.
      Since $\diam(\vect{\mathscr{U}}) \leq \varepsilon$, the images of every member of $\vee_{n}^{\vect{T}}\vect{\mathscr{U}}$ under $\vect{T}^{j}$ have diameters no more than $\varepsilon$ at level j where  $0 \leq j \leq n-1$,
      but any two distinct points $x$ and $y$ from $E$ satisfy $d_{j}(\vect{T}^{j}x,\vect{T}^{j}y) > \varepsilon$ for some $0 \leq j \leq n-1$,
      so no member of $\vee_{n}^{\vect{T}}\vect{\mathscr{U}}$ contains two distinct points in $E$.

      Suppose that $\Gamma \subset \vect{\mathscr{U}}_{0}^{n}$ covers $Z$. Then $\Gamma$ covers $E \subset Z$, and hence for every $x \in E$, there is a string $\mathbf{U}_{x} \in \Gamma$ such that $x \in X_{0}[\mathbf{U}_{x}]$. Moreover, $x$ is the only point from $E$ contained in $X_{0}[\mathbf{U}_{x}]$, and it follows that
      $$
      \sum_{x \in E}\e^{S_{n}^{\vect{T}}\vect{f}(x)} \leq \sum_{x \in E}\exp{(\underline{S}_{n}^{\vect{T}}\vect{f}(\mathbf{U}_{x}) + n\alpha)} \leq \sum_{\mathbf{U} \in \Gamma}\exp{(\underline{S}_{n}^{\vect{T}}\vect{f}(\mathbf{U}) + n\alpha)}.
      $$
      Since the above inequalities hold for all $(n,\varepsilon)$-separated sets $E$ for $Z$ and all $\Gamma \subset \vect{\mathscr{U}}_{0}^{n}$ that covers $Z$, we have the inequality \eqref{eq:QnUgeqPne}.

      Finally,   the proofs for the upper pressure $\preU$ and lower pressure $\preL$ are almost identical, and we only give the proof for  $\preL$.

 Since  $\vect{f}$ is equicontinuous,  by \eqref{eq:QnUleqQnd2} and \eqref{eq:QnUgeqPne}, for all $\alpha>0$, we have that
    \begin{equation}\label{ineq_PQQ}
      \prelow(\vect{T},\vect{f},Z,\varepsilon)-\alpha \leq \qrelow(\vect{T},\vect{f},Z,\vect{\mathscr{U}}) \leq \qrelow(\vect{T},\vect{f},Z,\frac{\delta}{2})
     \end{equation}
      for all sufficiently small $\varepsilon>0$ and sequences $\vect{\mathscr{U}}$ of open covers of $X_{k}$ with $\diam(\vect{\mathscr{U}})<\varepsilon$ and a Lebesgue number  $\delta>0$.
      Combining this with Proposition~\ref{prop:QPe}, it follows that
      \begin{align*}
        \prelow(\vect{T},\vect{f},Z) -\alpha \leq \sup_{\vect{\mathscr{U}}}\qrelow(\vect{T},\vect{f},Z,\vect{\mathscr{U}})
                                             \leq \sup_{\vect{\mathscr{U}}}\qrelow(\vect{T},\vect{f},Z,\frac{\delta}{2})
                                             \leq \qrelow(\vect{T},\vect{f},Z).
      \end{align*}
     Since  $\vect{f}$ is equicontinuous, by Proposition~\ref{prop:PeqQ} and the arbitrariness of $\alpha>0$, we conclude that
      $$
      \preL(\vect{T},\vect{f},Z) =\sup_{\vect{\mathscr{U}}}\qrelow(\vect{T},\vect{f},Z,\vect{\mathscr{U}}),
      $$
      where the supremum is taken over all sequences $\vect{\mathscr{U}}$ of open covers of $X_{k}$ with a Lebesgue number.

    Meanwhile, since $\vect{f}$ is equicontinuous, by Proposition~\ref{prop:PeqQ}, we have that
$$
\preL(\vect{T},\vect{f},Z) =\lim_{\delta \to 0}\qrelow(\vect{T},\vect{f},Z,\frac{\delta}{2})=\lim_{\varepsilon \to 0}\prelow(\vect{T},\vect{f},Z,\varepsilon),
$$
and  by \eqref{ineq_PQQ},  it follows that
      \begin{align*}
      \varlimsup_{\diam(\vect{\mathscr{U}}) \to 0}\qrelow(\vect{T},\vect{f},Z,\vect{\mathscr{U}}) \leq  \preL(\vect{T},\vect{f},Z)  \leq \varliminf_{\diam(\vect{\mathscr{U}}) \to 0}\qrelow(\vect{T},\vect{f},Z,\vect{\mathscr{U}}),
      \end{align*}
which implies that
      $$
      \preL(\vect{T},\vect{f},Z) =\lim_{\diam(\vect{\mathscr{U}}) \to 0}\qrelow(\vect{T},\vect{f},Z,\vect{\mathscr{U}}).
      $$
      \end{proof}

    \begin{prop}\label{prop:precov}
      Given an NDS $(\vect{X},\vect{T})$ and $Z \subset X_{0}$, if $\vect{f} \in \vect{C}(\vect{X},\R)$ is equicontinuous, then
      $$
      \preL(\vect{T},\vect{f},Z) = \lim_{\diam(\vect{\mathscr{U}}) \to 0}\prelow(\vect{T},\vect{f},Z,\vect{\mathscr{U}}); \quad
      \preU(\vect{T},\vect{f},Z) = \lim_{\diam(\vect{\mathscr{U}}) \to 0}\preup(\vect{T},\vect{f},Z,\vect{\mathscr{U}}),
      $$
      where $\vect{\mathscr{U}}$ ranges over all sequences of open covers of $X_{k}$ with a Lebesgue number.
    \end{prop}
\begin{proof}
The proof is  similar but simpler to Proposition \ref{prop:qrecov}, and we omit it.
\end{proof}
    \begin{rmk}
      Let $P$ denote one of $\prelow$ and $\preup$.
      Even for autonomous systems, it may happen that
       $$
      \sup_{\vect{\mathscr{U}}}P(\vect{T},\vect{f},Z,\vect{\mathscr{U}}) >P(\vect{T},\vect{f},Z),
      $$
      where $\vect{\mathscr{U}}$ is a sequence of open covers of $X_{k}$ with a Lebesgue number. See  \cite{Walters1975,Walters1982}.
    \end{rmk}

    A canonical class of sequences $\vect{\mathscr{U}}$  consists of balls, namely, for every $\varepsilon>0$, let
    $$
    \mathscr{B}_{k}(\varepsilon)=\{ B_{X_{k}}(x,\varepsilon): x \in X_{k}\}
    $$
    for each $k \in \N$. We write $\vect{\mathscr{B}}(\varepsilon)=\{\mathscr{B}_{k}(\varepsilon)\}_{k=0}^{\infty}$.
    It is clear that $\diam(\vect{\mathscr{B}}(\varepsilon))=2\varepsilon$ and that the sequence $\vect{\mathscr{B}}(\varepsilon)$ of covers $\mathscr{B}_{k}(\varepsilon)$ of $X_{k}$ has  a Lebesgue number $\delta=\varepsilon>0$.
 We may replace the covers in Definition \ref{eq:QnPncov} by using Bowen balls and taking $\inf$' or `$\sup$' of $S_{n}^{\vect{T}}\vect{f}$ on the Bowen balls instead of the values at the centers, and the quantities in \eqref{def_QQPPCV} remain the same.

    \subsection{Bowen pressures via open covers}
    In this subsection, we present the Bowen pressures using open covers, and we obtain formulations similar to the lower and upper pressures using open covers.

   Given a sequence $\vect{\mathscr{U}}=\{\mathscr{U}_{k}\}_{k=0}^{\infty}$ of open covers $\mathscr{U}_{k}$ of $X_{k}$, recall that
    $$
    \vect{\mathscr{U}}_{0}^{n} = \{\mathbf{U}=U_{0}U_{1} \cdots U_{n-1}: U_{j} \in \mathscr{U}_{j}, j = 0,\ldots,n-1\}
    $$
 and $\len{\mathbf{U}}=n $ is the   length of the string  $\mathbf{U}\in \vect{\mathscr{U}}_0^{n}$.     We say that $\Gamma \subset \cup_{n=0}^\infty \vect{\mathscr{U}}_{0}^{n}$ \emph{covers} $Z \subset X_{0}$ if $Z \subset \bigcup_{\mathbf{U} \in \Gamma}X_{0}[\mathbf{U}]$ where $X_{0}[\mathbf{U}]$ is defined by \eqref{def_XkbU}.

Given an NDS $(\vect{X},\vect{T})$, $\vect{f} \in \vect{C}(\vect{X},\R)$ and a sequence $\vect{\mathscr{U}}$ of open covers of $X_{k}$,
    for each $s \in \R$ and $N>0$, we define a measure $\msrbpp_{N}^{s}(\vect{T},\vect{f},\cdot,\vect{\mathscr{U}})$ by
    \begin{equation}\label{eq:defmsrbpp}
      \msrbpp_{N}^{s}(\vect{T},\vect{f},Z,\vect{\mathscr{U}}) = \inf_{\Gamma}\Big\{\sum_{\mathbf{U} \in \Gamma}\exp{\Big(-\len{\mathbf{U}}s + \overline{S}_{\len{\mathbf{U}}}^{\vect{T}}\vect{f}(\mathbf{U})\Big)}\Big\},
    \end{equation}
    where the infimum is taken over all countable covers  $\Gamma$ of  $Z$ satisfying that $\len{\mathbf{U}} \geq N$ for every $\mathbf{U} \in \Gamma$.
    Clearly $\msrbpp_{N}^{s}(\vect{T},\vect{f},Z,\vect{\mathscr{U}})$ is non-decreasing as $N$ tends to $\infty$ for every given $Z$,
    and we write
    $$
    \msrbpp^{s}(\vect{T},\vect{f},Z,\vect{\mathscr{U}}) = \lim_{N \to \infty}\msrbpp_{N}^{s}(\vect{T},\vect{f},Z,\vect{\mathscr{U}}).
    $$
    A similar dimension structure is given by the `jump point' in $s$,  denoted by
    \begin{equation}\label{def_PBPPep}
      \begin{aligned}
        \prebpp(\vect{T},\vect{f},Z,\vect{\mathscr{U}}) &= \sup\{s \in \R: \msrbpp^{s}(\vect{T},\vect{f},Z,\vect{\mathscr{U}}) = +\infty\} \\
                                                        &= \inf\{s \in \R: \msrbpp^{s}(\vect{T},\vect{f},Z,\vect{\mathscr{U}}) = 0\}.
      \end{aligned}
    \end{equation}

Next, we provide an equivalent description of $\prebow$ for equicontinuous potentials.
    \begin{prop}\label{prop:prebppcov}
      Given  $Z \subset X_{0}$, if $\vect{f} \in \vect{C}(\vect{X},\R)$ is equicontinuous, then
      $$
      \prebow(\vect{T},\vect{f},Z) = \lim_{\diam(\vect{\mathscr{U}}) \to 0}\prebpp(\vect{T},\vect{f},Z,\vect{\mathscr{U}}),
      $$
      where $\vect{\mathscr{U}}$ ranges over all sequences of open covers of $X_{k}$ with a Lebesgue number.
    \end{prop}
    \begin{proof}
      It is sufficient to show that
      $$
      \varlimsup_{\diam(\vect{\mathscr{U}}) \to 0}\prebpp(\vect{T},\vect{f},Z,\vect{\mathscr{U}}) \leq \prebow(\vect{T},\vect{f},Z) \leq \varliminf_{\diam(\vect{\mathscr{U}}) \to 0}\prebpp(\vect{T},\vect{f},Z,\vect{\mathscr{U}}),
      $$
      where $\vect{\mathscr{U}}$ ranges over all sequences of open covers of $X_{k}$ with a Lebesgue number.

      Suppose that $\vect{\mathscr{U}}$ is a sequence of open covers of $X_{k}$ with $\diam(\vect{\mathscr{U}}) < \varepsilon$. Then  for each $\mathbf{U}$, the nonempty set $X_{0}[\mathbf{U}]$  is contained in   $B_{n}^{\vect{T}}(x,2\varepsilon)$ where $x \in X_{0}[\mathbf{U}]$ and $n=\len{\mathbf{U}}$.
      Given $N>0$, for every $\Gamma=\{\mathbf{U}_{i}\}_{i=1}^{\infty}$ covering $Z$ with $\len{\mathbf{U}_{i}} \geq N$ for all $i$, choosing $x_{i} \in X_{0}[\mathbf{U}_{i}]$ for each $i$,
      the family $\{B_{\len{\mathbf{U}_{i}}}^{\vect{T}}(x_{i},2\varepsilon)\}_{i=1}^{\infty}$ is a $(N,2\varepsilon)$-cover of $Z$, and
      $$
      \sum_{i=1}^{\infty}\exp{\Big(-\len{\mathbf{U}_{i}}s + \overline{S}_{\len{\mathbf{U}_{i}}}^{\vect{T}}\vect{f}(\mathbf{U}_{i})\Big)} \geq \sum_{i=1}^{\infty}\exp{\Big(-\len{\mathbf{U}_{i}}s + S_{\len{\mathbf{U}_{i}}}^{\vect{T}}\vect{f}(x_{i})\Big)}.
      $$
      Since this holds for all $\Gamma$ covering $Z$ with $\len{\mathbf{U}} \geq N$ for all $\mathbf{U} \in \Gamma$, by \eqref{eq:defmsrbow} and \eqref{eq:defmsrbpp},
      $$
      \msrbow_{N,2\varepsilon}^{s}(\vect{T},\vect{f},Z) \leq \msrbpp_{N}^{s}(\vect{T},\vect{f},Z,\vect{\mathscr{U}})
      $$
      for all $s \in \R$ and $N \geq 1$, and letting $N \to \infty$ gives $\msrbow_{2\varepsilon}^{s}(\vect{T},\vect{f},Z) \leq \msrbpp^{s}(\vect{T},\vect{f},Z,\vect{\mathscr{U}})$. By \eqref{def_PBep} and \eqref{def_PBPPep}, it implies  that
      $$
      \prebow(\vect{T},\vect{f},Z,2\varepsilon) \leq \prebpp(\vect{T},\vect{f},Z,\vect{\mathscr{U}}).
      $$
      Since $\diam(\vect{\mathscr{U}})$ tends to $0$ as $\varepsilon$ goes to $0$, we conclude that
     \begin{equation} \label{ineq_PB2EQB}
      \prebow(\vect{T},\vect{f},Z) \leq \varliminf_{\diam(\vect{\mathscr{U}}) \to 0}\prebpp(\vect{T},\vect{f},Z,\vect{\mathscr{U}}).
      \end{equation}

      On the other hand, arbitrarily choosing   $\alpha>0$, by the equicontinuity of $\vect{f}$, there exists $\varepsilon_{0}>0$ such that for all $\varepsilon$ with $0<\varepsilon\leq\varepsilon_{0}$, all integral $n \geq 1$ and all $x,y \in X_{0}$,
      $$
      \lvert S_{n}^{\vect{T}}\vect{f}(x) - S_{n}^{\vect{T}}\vect{f}(y) \rvert < n\alpha
      $$
      whenever $y \in B_{n}^{\vect{T}}(x,\varepsilon)$.
      Let $\vect{\mathscr{U}}$ be a sequence of open covers of $X_{k}$ with Lebesgue number $\delta>0$. We   assume that $\frac{\delta}{2}<\varepsilon_{0}$.
      Then every Bowen ball $B_{n}^{\vect{T}}(x,\frac{\delta}{2})$  is contained in $X_{0}[\mathbf{U}_{x}]$ for some $\mathbf{U}_{x} \in \vect{\mathscr{U}}_{0}^{n}$.
      It follows that for every countable $(N,\frac{\delta}{2})$-cover $\{B_{n_{i}}^{\vect{T}}(x_{i},\frac{\delta}{2})\}_{i=1}^{\infty}$ of $Z$,
      the family $\Gamma=\{\mathbf{U}_{x_{i}}\}_{i=1}^{\infty}$ covers $Z$ and $\len{\mathbf{U}_{x_{i}}}=n_{i} \geq N$ for all $i$.
      Hence, by \eqref{eq:defmsrbow} and \eqref{eq:defmsrbpp}, we obtain that
      $$
      \sum_{i=1}^{\infty}\exp{\Big(-\len{\mathbf{U}_{x_{i}}}s + \overline{S}_{\len{\mathbf{U}_{x_{i}}}}^{\vect{T}}\vect{f}(\mathbf{U}_{x_{i}}) \Big)} \leq \sum_{i=1}^{\infty}\exp{\Big(-n_{i}(s-\alpha)+S_{n_{i}}^{\vect{T}}\vect{f}(x_{i})\Big)},
      $$
      and this implies that  for all sufficiently small $\delta>0$,
      $$
      \msrbpp_{N}^{s}(\vect{T},\vect{f},Z,\vect{\mathscr{U}}) \leq \msrbow_{N,\frac{\delta}{2}}^{s-\alpha}(\vect{T},\vect{f},Z).
      $$
      It follows that $\msrbpp^{s}(\vect{T},\vect{f},Z,\vect{\mathscr{U}}) \leq \msrbow_{\frac{\delta}{2}}^{s-\alpha}(\vect{T},\vect{f},Z)$,
      and hence
      $$
      \prebpp(\vect{T},\vect{f},Z,\vect{\mathscr{U}}) \leq \prebow(\vect{T},\vect{f},Z,\frac{\delta}{2}) + \alpha
      $$
      for all sufficiently small $\delta>0$. Therefore $\prebpp(\vect{T},\vect{f},Z,\vect{\mathscr{U}}) \leq \prebow(\vect{T},\vect{f},Z) + \alpha$,
      and       by the arbitrariness of $\alpha>0$, it follows that
      $$
      \prebpp(\vect{T},\vect{f},Z,\vect{\mathscr{U}}) \leq \prebow(\vect{T},\vect{f},Z).
      $$

      Finally, letting $\diam(\vect{\mathscr{U}})$ tend to $0$, we have that
      $$
      \varlimsup_{\diam(\vect{\mathscr{U}}) \to 0}\prebpp(\vect{T},\vect{f},Z,\vect{\mathscr{U}}) \leq \prebow(\vect{T},\vect{f},Z).
      $$
Combining this with \eqref{ineq_PB2EQB}, the conclusion holds.
    \end{proof}
\begin{rmk}
Proposition~\ref{prop:prebppcov} is still true if we  use $\underline{S}_{\len{\mathbf{U}}}^{\vect{T}}\vect{f}(\mathbf{U})$ in \eqref{eq:defmsrbpp}, and the argument is almost the same.
\end{rmk}

    \subsection{Bowen pressures via weighted measures}\label{ssect:equivdefbw}
    We introduce an alternative characterization of Bowen pressures,
    which is inspired by weighted Hausdorff measures from geometric measure theory (see \cite{Federer1969, Mattila1995}).

    Given $(\vect{X},\vect{T})$ and $Z \subset X_{0}$, we say that a family $\{(B_{n_{i}}^{\vect{T}}(x_{i},\varepsilon),c_{i})\}_{i \in \mathcal{I}}$ of paired
    Bowen balls $B_{n_{i}}^{\vect{T}}(x_{i},\varepsilon)$ and reals $0<c_{i}<+\infty$ is a \emph{weighted $(N,\varepsilon)$-cover of $Z$}
    if $n_{i} \geq N$ for all $i \in \mathcal{I}$ and
    \begin{equation}
      \sum_{i \in \mathcal{I}}c_{i}\chi_{B_{n_{i}}^{\vect{T}}(x_{i},\varepsilon)} \geq \chi_{Z},
    \end{equation}
    where $\chi_{A}$ denotes the characteristic function of a set $A \subset X_{0}$, i.e.,
    \begin{equation*}
      \chi_{A}(x)=\left\{
        \begin{array}{ll}
          1, & \text{if}\ x \in A, \\
          0, & \text{if}\ x \in X_{0} \setminus A.
        \end{array}
        \right.
    \end{equation*}

    Given $\vect{f} \in \vect{C}(\vect{X},\mathbb{R})$, for all $s \in \mathbb{R}$, $N>0$ and $\varepsilon>0$, we write
    \begin{equation}\label{def_WBPP}
      \mathscr{W}_{N,\varepsilon}^{s}(\vect{T},\vect{f},Z)=\inf\Big\{\sum_{i=1}^{\infty}c_{i}\exp{\Big(-n_{i}s+{S_{n_{i}}^{\vect{T}}\vect{f}(x_{i})}\Big)}\Big\},
    \end{equation}
    where the infimum is taken over all weighted $(N,\varepsilon)$-covers $\{(B_{n_{i}}^{\vect{T}}(x_{i},\varepsilon),c_{i})\}_{i=1}^{\infty}$ of $Z$.

    Next, we construct another type of topological pressure by using $\mathscr{W}_{N,\varepsilon}^{s}$. Setting
    $$
    \mathscr{W}_{\varepsilon}^{s}(\vect{T},\vect{f},Z)=\lim_{N \to \infty} \mathscr{W}_{N,\varepsilon}^{s}(\vect{T},\vect{f},{Z}),
    $$
    we write
    \begin{equation}\label{def_PBW}
      \prebw(\vect{T},\vect{f},Z,\varepsilon) = \inf\{s: \mathscr{W}_{\varepsilon}^{s}(\vect{T},\vect{f},Z)=0\}                                              =  \sup\{s: \mathscr{W}_{\varepsilon}^{s}(\vect{T},\vect{f},Z)=+\infty\}.
    \end{equation}

    We provide the following equivalent definition of the Bowen pressure for equicontinuous potentials $\vect{f} \in \vect{C}(\vect{X},\R)$.
    \begin{prop}\label{prop:prebw}
      Given   $Z \subset X_{0}$, for all equicontinuous $\vect{f} \in \vect{C}(\vect{X},\mathbb{R})$,
      $$
      \prebow(\vect{T},\vect{f},Z)=\lim_{\varepsilon \to 0}\prebw(\vect{T},\vect{f},Z,\varepsilon).
      $$
    \end{prop}

We need the $5r$-covering lemma to prove Proposition \ref{prop:prebw}; see  \cite{Mattila1995} for the proof.
    \begin{lem}\label{coveringlem}
      Let $(X,d)$ be a compact metric space and $\mathcal{B}=\{B(x_{i},r_{i})\}_{i \in \mathcal{I}}$
      a family of closed (or open) balls in $X$.
      Then there exists a finite or countable subfamily $\mathcal{B}^{\prime}=\{B(x_{i},r_{i})\}_{i \in \mathcal{I}^{\prime}}$
      of pairwise disjoint balls in $\mathcal{B}$ such that
      $$\bigcup_{B \in \mathcal{B}}B \subset \bigcup_{i \in \mathcal{I}^{\prime}}B(x_{i},5r_{i}).$$
    \end{lem}
    Note that the author of \cite{Mattila1995} assumed that $(X,d)$ is boundedly compact
    (all closed bounded subsets are compact), and the conclusion holds for $\mathcal{B}$ of closed balls with $\sup\{\diam(B): B \in \mathcal{B}\}<+\infty$. In this paper, we always assume that $X$ is a compact metric space. Since the compactness of $X$ implies that $X$ is boundedly compact, the conclusion  also holds for open balls by the same argument as in \cite[Thm.2.1]{Mattila1995}.

    \begin{proof}[Proof of Proposition~\ref{prop:prebw}]
      By Definition~\ref{def:prebpp}, it is sufficient to prove that for every $\alpha>0$, the inequality
      \begin{equation}\label{ineqn_QPQ}
        \prebow(\vect{T},\vect{f},Z,6\varepsilon)-\alpha \leq \prebw(\vect{T},\vect{f},Z,\varepsilon) \leq \prebow(\vect{T},\vect{f},Z,\varepsilon)
      \end{equation}
      holds for all $\varepsilon>0$.

      We first claim that for every $s \in \mathbb{R}$ and $\alpha>0$, the inequality
      \begin{equation}\label{ineq_RWM}
        \msrbow_{N,6\varepsilon}^{s+\alpha}(\vect{T},\vect{f},Z) \leq \mathscr{W}_{N,\varepsilon}^{s}(\vect{T},\vect{f},Z) \leq \msrbow_{N,\varepsilon}^{s}(\vect{T},\vect{f},Z)
      \end{equation}
      holds for all sufficiently small $\varepsilon>0$ and all sufficiently large $N$.
      It is clear that the inequality~\eqref{ineqn_QPQ} follows if the claim \eqref{ineq_RWM} were true.

      Given $N \geq 1$ and $\varepsilon>0$, since for all countable $(N,\varepsilon)$-covers $\{B_{n_{i}}^{\vect{T}}(x_{i},\varepsilon)\}_{i=1}^{\infty}$ of $Z$,
      by taking $c_{i}=1$, $\sum_{i=1}^{\infty}c_{i}\chi_{B_{n_{i}}^{\vect{T}}(x_{i},\varepsilon)} > \chi_{Z}$,
      and it follows that
      $$
      \mathscr{W}_{N,\varepsilon}^{s}(\vect{T},\vect{f},Z) \leq \msrbow_{N,\varepsilon}^{s}(\vect{T},\vect{f},Z)
      $$
      for all $N \geq 1$.

      To show that
      $$
      \msrbow_{N,6\varepsilon}^{s+\alpha}(\vect{T},\vect{f},Z) \leq \mathscr{W}_{N,\varepsilon}^{s}(\vect{T},\vect{f},Z)
      $$
      for sufficiently small $\varepsilon>0$ and sufficiently large $N >0$, it suffices to show that
      \begin{equation}\label{eq:msrbow6eleqmsrbppw}
        \msrbow_{N,6\varepsilon}^{s+\alpha}(\vect{T},\vect{f},Z) \leq \sum_{i \in \mathcal{I}}c_{i}\exp{\bigl(-n_{i}s + S_{n_{i}}^{\vect{T}}\vect{f}(x_{i})\bigr)}
      \end{equation}
      for every countable weighted $(N,\varepsilon)$-cover $\{(B_{n_{i}}^{\vect{T}}(x_{i},\varepsilon),c_{i})\}_{i \in \mathcal{I}}$.

      For all sufficiently large $N$, it is clear that $n \leq \e^{n\alpha}$ whenever $n \geq N$.
      Let $\{(B_{n_{i}}^{\vect{T}}(x_{i},\varepsilon),c_{i})\}_{i \in \mathcal{I}}$ be an arbitrary weighted $(N,\varepsilon)$-cover with $\mathcal{I} \subset \{1,2,3,\ldots\}$.
      By definition, it satisfies the inequality
      \begin{equation}\label{eq:msrbwcov}
        \sum_{i \in \mathcal{I}}c_{i}\chi_{B_{n_{i}}^{\vect{T}}(x_{i},\varepsilon)} \geq \chi_{Z}.
      \end{equation}
      For simplicity, we write $B_{i}=B_{n_{i}}^{\vect{T}}(x_{i},\varepsilon)$ and $5B_{i}=B_{n_{i}}^{\vect{T}}(x_{i},5\varepsilon)$ for every $i \in \mathcal{I}$. Without loss of generality, we assume $B_{i} \neq B_{j}$ for all $i \neq j \in \mathcal{I}$.

      For every $n \geq N$ and every integer $k>0$, we write
      $$
      \mathcal{I}_{n}=\{i \in \mathcal{I}: n_{i}=n\}\quad \text{and}\quad \mathcal{I}_{n,k}=\{i \in \mathcal{I}_{n}: i \leq k\},
      $$
      and it is clear that each $\mathcal{I}_{n,k}$ is finite, that $\mathcal{I}_{n,k}\subset \mathcal{I}_{n,k+1}$  with  $\lim_{k\to\infty} \mathcal{I}_{n,k} =\mathcal{I}_{n}$ and that
      \begin{equation}\label{eq:Indisju}
        \mathcal{I}=\bigcup_{n=1}^{\infty}\mathcal{I}_{n} \quad \text{where the}\ \mathcal{I}_{n}\text{'s are pairwise disjoint}.
      \end{equation}

      For each $r>0$, we write
      $$
      Z_{n}(r)=\Big\{x \in Z: \sum_{i \in \mathcal{I}_{n}}c_{i}\chi_{B_{i}}(x) > r\Big\} \ \  \text{and} \  \ Z_{n,k}(r)=\Big\{x \in Z: \sum_{i \in \mathcal{I}_{n,k}}c_{i}\chi_{B_{i}}(x) > r\Big\},
      $$
      and it is obvious that $Z_{n,k}(r)\subset Z_{n,k+1}(r)$ and
      \begin{equation}\label{eq_limznkr}
        \lim_{k\to\infty } Z_{n,k}(r)=Z_{n}(r).
      \end{equation}

      The proof for the inequality \eqref{eq:msrbow6eleqmsrbppw} is divided into the following 2 claims.
      \begin{claim}\label{clm:1}
        For each integer $n \geq N$, $k>0$ and real $r>0$, there exists a finite index set $\mathcal{J}_{n,k}(r) \subset \mathcal{I}_{n,k}$
        such that $B_{i} \cap B_{j} =\emptyset $ for all $i\neq j \in \mathcal{J}_{n,k}(r)$,
        $$
        Z_{n,k}(r) \subset \bigcup_{i \in \mathcal{J}_{n,k}(t)}5B_{i}
        $$
        and
        \begin{equation}\label{eq:clm1}
          \sum_{i \in \mathcal{J}_{n,k}(r)}\exp{\bigl(-ns+S_{n}^{\vect{T}}\vect{f}(x_{i})\bigr)} \leq \frac{1}{r}\sum_{i \in \mathcal{I}_{n,k}}c_{i}\exp{\bigl(-ns+S_{n}^{\vect{T}}\vect{f}(x_{i})\bigr)}.
        \end{equation}
      \end{claim}
      \begin{proof}
        Since $\mathcal{I}_{n,k}$ is finite, it is sufficient to prove that inequality \eqref{eq:clm1} holds for $\{c_i: i\in \mathcal{I}_{n,k}\}$ of rational $c_{i}$'s.
        Moreover, by multiplying with a common denominator on both sides of inequality \eqref{eq:clm1} when $\{c_i: i\in \mathcal{I}_{n,k}\}$ is of rationales,
        it is equivalent to prove \eqref{eq:clm1} for $\{c_i: i\in \mathcal{I}_{n,k}\}$ of positive integers.

        Fix $n \geq N$, $k > 0$ and $r>0$. Let $m$ be the smallest integer with $m \geq r$.
        Let
        $$
        \mathcal{B}=\{B_{i}: i \in \mathcal{I}_{n,k}\},
        $$
        and it is clear that
        $$
        Z_{n,k}(r)\subset\bigcup_{B\in\mathcal{B}}B.
        $$
        We define a mapping $u:  \mathcal{B} \to \mathbb{Z}$ by
        $$
        u(B_{i}) = c_{i}\qquad \text{for each}\ i \in \mathcal{I}_{n,k}.
        $$
        It is obvious that $\mathcal{B}$ is finite since $\mathcal{I}_{n,k}$ is finite.
        We set $v_{0}=u$, and for $l=1,\cdots,m$, we inductively define disjoint subfamilies $\mathcal{B}_{l}$ of $\mathcal{B}$ and integer-valued functions $v_{l}:\mathcal{B} \to \mathbb{Z}$ by
        \begin{equation}\label{def_vl}
          v_{l}(B)=\left\{
            \begin{array}{ll}
              v_{l-1}(B)-1 &\text{for}\ B \in \mathcal{B}_{l},\\
              v_{l-1}(B)   &\text{for}\ B \in \mathcal{B}\setminus\mathcal{B}_{l},
            \end{array}
            \right.
        \end{equation}
        where $\mathcal{B}_{l}$'s and $v_{l}$'s satisfy that
        \begin{enumerate}[(a)]
          \item\label{clm1a}
            $\mathcal{B}_{l} \subset \{B \in \mathcal{B}: v_{l-1}(B) \geq 1\}$ and $Z_{n,k}(r) \subset \bigcup_{B_{i} \in \mathcal{B}_{l}}5B_{i}$.
          \item\label{clm1b} $v_{l}(B) \geq 0$ for all $B \in \mathcal{B}$, and for all $x \in Z_{n,k}(r)$,
            $$
            \sum_{\substack{B \in \mathcal{B}: \\ B \ni x}}v_{l}(B) \geq \sum_{\substack{B \in \mathcal{B}: \\ B \ni x}}u(B) - l.
            $$

        \end{enumerate}

        For $l=1$, we write $\mathcal{G}_{0}=\{B \in \mathcal{B}: v_{0}(B) \geq 1\}$.
        It is clear that $\mathcal{G}_{0}= \mathcal{B}$ since $v_{0}=u$.
        We apply Lemma \ref{coveringlem} on the compact metric space $(X_{0},d_{n}^{\vect{T}})$
        to obtain a disjoint subfamily $\mathcal{B}_{1} \subset \mathcal{G}_{0}$ such that $   \bigcup_{B\in\mathcal{B}}B \subset \bigcup_{B_{i} \in \mathcal{B}_{1}}5B_{i},$
        which implies that $   Z_{n,k}(r) \subset \bigcup_{B_{i} \in \mathcal{B}_{1}}5B_{i}.$
        Obviously, the family $\mathcal{B}_{1}$ and the mapping $v_{1}$ satisfy \eqref{clm1a} and \eqref{clm1b}.

        For $2 \leq l \leq m$, assume that $\mathcal{B}_{l-1}$ and $v_{l-1}$ have been defined with \eqref{clm1a} and \eqref{clm1b} satisfied.
        Let
        $$
        \mathcal{G}_{l-1}=\{B \in \mathcal{B}: v_{l-1}(B) \geq 1\} \subset \mathcal{B}.
        $$
        By Lemma \ref{coveringlem}, there exists a disjoint subfamily $\mathcal{B}_{l} \subset \mathcal{G}_{l-1}$ such that
        $$
        \bigcup_{B \in \mathcal{G}_{l-1}}B \subset \bigcup_{B_{i} \in \mathcal{B}_{l}}5B_{i}.
        $$

        Recall that $Z_{n,k}(r)=\Big\{x \in Z: \sum_{\substack{B \in \mathcal{B}: \\ B \ni x}}u(B) \geq r\Big\}$.
        Since $m$ is the least integer such that $m\geq r$, and $c_i(i\in \mathcal{I}_{n,k})$ are positive integers, we have
        \begin{equation}\label{Znkrsqm}
          \sum_{\substack{B \in \mathcal{B}: \\ B \ni x}}u(B)>m
        \end{equation}
        for every $x \in Z_{n,k}(r)$.
        Meanwhile, by \eqref{def_vl}, we have that for all $x \in Z_{n,k}(r)$,
        \begin{align*}
          \sum_{\substack{B \in \mathcal{B}: \\ B \ni x}}v_{l}(B) &= \sum_{\substack{B \in \mathcal{B}_{l}: \\ B \ni x}}(v_{l-1}(B)-1) + \sum_{\substack{B \in \mathcal{B}\setminus\mathcal{B}_{l}: \\ B \ni x}}v_{l-1}(B) \\
                  &= \sum_{\substack{B \in \mathcal{B}: \\ B \ni x}}v_{l-1}(B) - \#\{B \in \mathcal{B}_{l}: B \ni x\}.
        \end{align*}
        Since $\mathcal{B}_{l}$ is disjoint, for each $x \in Z_{n,k}(r)$, there is at most one $B \in \mathcal{B}_{l}$
        with $x \in B$, i.e., $\#\{B \in \mathcal{B}_{l}: B \ni x\} \leq 1$.
        Combined with the induction hypothesis \eqref{clm1b} and the inequality \eqref{Znkrsqm},
        this implies that
        $$
        \sum_{\substack{B \in \mathcal{B}: \\ B \ni x}}v_{l}(B)
        \geq \sum_{\substack{B \in \mathcal{B}: \\ B \ni x}}v_{l-1}(B) - 1 
        \geq \sum_{\substack{B \in \mathcal{B}: \\ B \ni x}}u(B) - l 
        \geq m-l
        $$
        for all $x \in Z_{n,k}(r)$.
        Hence, for $l<m$, we have
        $$
        Z_{n,k}(r) \subset \Big\{x \in Z: \sum_{\substack{B \in \mathcal{B}: \\ B \ni x}}v_{l}(B) \geq m-l\Big\},
        $$
        which implies that every $x \in Z_{n,k}(r)$ is contained in some ball $B \in \mathcal{B}$ with $v_{l}(B) \geq 1$,
        so the induction process above does proceed.

        Therefore the integer-values mappings $v_{1},\cdots,v_{m}:\mathcal{B} \to \mathbb{Z}$ and the disjoint subfamilies $\mathcal{B}_{1},\cdots,\mathcal{B}_{m}$ of $\mathcal{B}$ are defined
        and satisfy the properties (\ref{clm1a}) and (\ref{clm1b}).

        It immediately follows that
        \begin{align*}
          \sum_{l=1}^{m}\sum_{B \in \mathcal{B}_{l}}\exp{\big(-ns + S_{n}^{\vect{T}}\vect{f}(x_{B})\big)} &= \sum_{l=1}^{m}\sum_{B \in\mathcal{B}_{l}}(v_{l-1}(B)-v_{l}(B))\exp{\big(-ns + S_{n}^{\vect{T}}\vect{f}(x_{B})\big)} \\
          &\leq \sum_{B \in \mathcal{B}}\sum_{l=1}^{m}(v_{l-1}(B)-v_{l}(B))\exp{\big(-ns + S_{n}^{\vect{T}}\vect{f}(x_{B})\big)} \\
          &\leq \sum_{B \in \mathcal{B}}u(B)\exp{\big(-ns + S_{n}^{\vect{T}}\vect{f}(x_{B})\big)} \\
          &= \sum_{i \in \mathcal{I}_{n,k}}c_{i}\exp{\big(-ns + S_{n}^{\vect{T}}\vect{f}(x_{i})\big)},
        \end{align*}
        where $x_{B}$ is the center of the Bowen ball for each $B \in \mathcal{B}$.
        Since $m$ is the least integer such that $m\geq r$, by choosing $l_{0}=l_{0}(n,k) \in \{1,\ldots,m\}$ such that
        $$
        \sum_{B \in \mathcal{B}_{l_{0}}}\exp{\big(-ns + S_{n}^{\vect{T}}\vect{f}(x_{B})\big)}=\min_{l=1,\ldots,m}\sum_{B \in \mathcal{B}_{l}}\exp{\big(-ns + S_{n}^{\vect{T}}\vect{f}(x_{B})\big)},
        $$
        we obtain that
        \begin{align*}
          \sum_{B \in \mathcal{B}_{l_{0}}}\exp{\bigl(-ns+S_{n}^{\vect{T}}\vect{f}(x_{B})\bigr)}
          &\leq \frac{1}{m}\sum_{i \in \mathcal{I}_{n,k}}c_{i}\exp{\bigl(-ns + S_{n}^{\vect{T}}\vect{f}(x_{i})\bigr)} \\
          &\leq \frac{1}{r}\sum_{i \in \mathcal{I}_{n,k}}c_{i}\exp{\bigl(-ns + S_{n}^{\vect{T}}\vect{f}(x_{i})\bigr)}.
        \end{align*}

        Recall that $\mathcal{I}_{n,k}$ and $\mathcal{B}$ are finite, and we have that $\mathcal{B}_{l_{0}}$ is finite and disjoint.
        Letting
        $$
        \mathcal{J}_{n,k}(r)=\{i \in \mathcal{I}_{n,k}: B_{i} \in \mathcal{B}_{l_{0}}\},
        $$
        the inequality \eqref{eq:clm1} holds, and we finish the proof of Claim \ref{clm:1}.
      \end{proof}
      \begin{claim}\label{clm:2}
        For each $n \geq N$, we have
        $$
        \msrbow_{N,6\varepsilon}^{s+\alpha}(\vect{T},\vect{f},Z_{n}(r)) \leq \frac{1}{rn^{2}}\sum_{i \in \mathcal{I}_{n}}c_{i}\exp{\bigl(-ns + S_{n}^{\vect{T}}\vect{f}(x_{i})\bigr)}.
        $$
      \end{claim}
      \begin{proof}
        We assume $Z_{n}(r) \neq \emptyset$ since the conclusion is obvious if $Z_{n}(r) = \emptyset$.
        By \eqref{eq_limznkr}, we have that $Z_{n,k}(r) \neq \emptyset$ for all sufficiently large integer $k$,
        which implies that $\mathcal{I}_{n,k}\neq \emptyset$.
        By Claim \ref{clm:1}, we have that $\mathcal{J}_{n,k}(r) \neq \emptyset$ for sufficiently large $k$.
        We write
        $$
        K_{n,k}(r)=\{x_{i}: i \in \mathcal{J}_{n,k}(r)\}
        $$
        for each integer $k>0$.
        Since the space of all non-empty compact subsets of $X_{0}$ is compact with respect to the Hausdorff distance (see Federer \cite[2.10.21]{Federer1969}),
        there exists a subsequence $\{k_{l}\}_{l=1}^{\infty}$ and a non-empty compact set $K_{n}(r) \subset X_{0}$ such that
        $K_{n,k_{l}}(r)$ converges to $K_{n}(r)$ in the Hausdorff distance.

        Since $d_{n}^{\vect{T}}(x, y)\geq \varepsilon$ for all $x\neq y\in K_{n,k}(r)$, we have that
        $d_{n}^{\vect{T}}(x, y)\geq \varepsilon$ for all $x\neq y\in K_{n}(r)$.
        Hence $K_{n}(r)$ is a finite set and $\#(K_{n,k_{l}}(r))=\#(K_{n}(r))$ for sufficiently large $l$.
        This implies that
        $$
        \bigcup_{x \in K_{n}(r)}B_{n}^{\vect{T}}(x,10/2 \cdot\varepsilon) \supset \bigcup_{x \in K_{n,k_{l}}(r)}B_{n}^{\vect{T}}(x,5\varepsilon)=\bigcup_{i \in \mathcal{J}_{n,k_{l}}(r)}5B_{i} \supset Z_{n,k_{l}}(a)
        $$
        for all sufficiently large $l$.
        By \eqref{eq_limznkr}, we obtain that
        \begin{equation}\label{eq:clm2cov6e}
          \bigcup_{x \in K_{n}(r)}B_{n}^{\vect{T}}(x,6\varepsilon) \supset Z_{n}(r).
        \end{equation}

        Note that $\varepsilon>0$ is arbitrary in all the arguments above.
        Fix $\varepsilon_{0}>0$ by the equicontinuity of $\vect{f}$ so that for all $\varepsilon$ with $0<\varepsilon\leq\varepsilon_{0}$ and all $n \geq N$,
        $$
        \lvert S_{n}^{\vect{T}}\vect{f}(x)-S_{n}^{\vect{T}}\vect{f}(y) \rvert < n\alpha
        $$
        whenever $d_{X_{0}}(x,y)<6\varepsilon$.
        Since $\lim_{k\to\infty} \mathcal{I}_{n,k} =\mathcal{I}_{n}$, combining this with Claim \ref{clm:1},
        it follows that
        \begin{equation}\label{eq:clm2kna}
          \begin{aligned}
            \sum_{x \in K_{n}(r)}\exp{\left(-ns + S_{n}^{\vect{T}}\vect{f}(x)\right)} &\leq \sum_{x \in K_{n,k_{l}}(r)}\exp{\big(-n(s+\alpha) + S_{n}^{\vect{T}}\vect{f}(x)\big)} \\
            &\leq \frac{1}{r}\sum_{i \in \mathcal{I}_{n,k_{l}}}c_{i}\exp{\big(-n(s+\alpha) + S_{n}^{\vect{T}}\vect{f}(x_{i})\big)} \\
            &\leq \frac{1}{r}\sum_{i \in \mathcal{I}_{n}}c_{i}\exp{\big(-n(s+\alpha) + S_{n}^{\vect{T}}\vect{f}(x_{i})\big)}.
          \end{aligned}
        \end{equation}
        Combining \eqref{eq:clm2cov6e} and \eqref{eq:clm2kna} with $n^{2} \leq \e^{2n\alpha}$, we have that
        \begin{align*}
          \msrbow_{N,6\varepsilon}^{s+\alpha}(\vect{T},\vect{f},Z_{n}(r)) &\leq \sum_{x \in K_{n}(r)}\exp{\left(-n(s+\alpha) + S_{n}^{\vect{T}}\vect{f}(x)\right)} \\
                                                                                  &\leq \frac{1}{r\e^{2n\alpha}}\sum_{i \in \mathcal{I}_{n}}c_{i}\exp{\big(-ns + S_{n}^{\vect{T}}\vect{f}(x_{i})\big)} \\
                                                                                  &\leq \frac{1}{rn^{2}}\sum_{i \in \mathcal{I}_{n}}c_{i}\exp{\big(-ns + S_{n}^{\vect{T}}\vect{f}(x_{i})\big)},
        \end{align*}
        and it completes the proof of Claim \ref{clm:2}.
      \end{proof}

      Finally, we are ready to prove Proposition \ref{prop:prebw}. Fix $r \in (0,1)$. For each $x \in Z$, by \eqref{eq:msrbwcov}, we have
      $\sum_{i \in \mathcal{I}}c_{i}\chi_{B_{i}}(x) \geq 1$.
      Note that $\sum_{n=N}^{\infty}\frac{1}{n^{2}}<1$ for sufficiently large $N$. Combining these with \eqref{eq:Indisju}, it follows that
      $$
      \sum_{n=N}^{\infty}\sum_{i \in \mathcal{I}_{n}}c_{i}\chi_{B_{i}}(x) > \sum_{n=N}^{\infty}\frac{r}{n^{2}},
      $$
      and this implies that for every $x \in Z$, there exists an integer $n \geq N$ such that
      $$
      \sum_{i \in \mathcal{I}_{n}}c_{i}\chi_{B_{i}}(x)>\frac{r}{n^{2}}.
      $$
      Immediately, we have that  $Z \subset \bigcup_{n=N}^{\infty}Z_{n}(\frac{r}{n^{2}})$, and by Claim~\ref{clm:2}, it follows that
      \begin{align*}
        \msrbow_{N,6\varepsilon}^{s+\alpha}(\vect{T},\vect{f},Z) &\leq \sum_{n=N}^{\infty}\msrbow_{N,6\varepsilon}^{s+\alpha}\Big(\vect{T},\vect{f},Z_{n}\big(\frac{r}{n^2}\big)\Big) \\
                                                            &\leq \sum_{n=N}^{\infty}\frac{1}{r}\sum_{i \in \mathcal{I}_{n}}c_{i}\exp{\big(-ns + S_{n}^{\vect{T}}\vect{f}(x_{i})\big)} \\
                                                            &\leq \frac{1}{r}\sum_{i \in \mathcal{I}}c_{i}\exp{\big(-n_{i}s + S_{n_{i}}^{\vect{T}}\vect{f}(x_{i})\big)}.
      \end{align*}
      Since the above inequalities hold for all $r \in (0,1)$, the inequality \eqref{eq:msrbow6eleqmsrbppw} holds by letting $r\to 1$, which completes the proof.
    \end{proof}

  \section{Properties Analogous to Fractal Dimensions}\label{sect:PropFracDim}
 As we have seen, the topological pressures are closely related to fractal dimensions, and in this section, we study the properties of the topological pressures from the  dimension  point of view,
  especially the properties of their dependence on the subsets and the relationships between the various pressures.  We refer readers to \cite{Falconer2014} for properties of fractal dimensions.

    \subsection{Properties with respect to subsets}
      Given an NDS $(\vect{X},\vect{T})$ and a potential $\vect{f} \in \vect{C}(\vect{X},\R)$, we first study the pressures' behavior on   subsets.
      Since a series of properties analogous to fractal dimensions are obtained,  these pressures may be regarded as a type of dimensions for nonautonomous dynamical systems, demonstrating the dimension structures correlated to dynamics.

      It follows direct from the definitions that those pressures  are increasing in sets.
      \begin{prop}\label{prop:premonotone}
        Given $\vect{f} \in \vect{C}(\vect{X},\R)$, if $Z_{1} \subset Z_{2} \subset X_{0}$, then
        $$
        P(\vect{T},\vect{f},Z_{1}) \leq P(\vect{T},\vect{f},Z_{2}),
        $$
        where $P \in \{\prebow, \prepac, \prelow, \preup, \qrelow, \qreup\}$.
      \end{prop}

      The lower and upper topological pressures have stability under closure.
      This is at first appealing as it suffices to calculate these pressure on compact (closed) subsets by such result,
      but undesirable consequences conflicts with the wish of countable stability. We write   $\overline{A}$ for  the closure of the given set $A$.
\begin{prop}\label{prop:preclstab}
Given $Z\subset X_{0}$  and $\vect{f} \in \vect{C}(\vect{X},\R)$,  for every $P \in \{\qrelow, \qreup, \prelow, \preup\}$.
        $$
        P(\vect{T},\vect{f},\overline{Z})=P(\vect{T},\vect{f},Z).
        $$

      \end{prop}
      \begin{proof}
        Since $\overline{Z} \supset Z$, by Proposition~\ref{prop:premonotone}, it suffices to prove the following inequalities
\begin{eqnarray}
&&        \qrelow(\vect{T},\vect{f},\overline{Z}) \leq \qrelow(\vect{T},\vect{f},Z),
        \quad
        \qreup(\vect{T},\vect{f},\overline{Z}) \leq \qreup(\vect{T},\vect{f},Z);          \label{ineq_QLQU} \\
&&        \prelow(\vect{T},\vect{f},\overline{Z}) \leq \prelow(\vect{T},\vect{f},Z),
        \quad
        \preup(\vect{T},\vect{f},\overline{Z}) \leq \preup(\vect{T},\vect{f},Z).          \label{ineq_PLPU}
\end{eqnarray}

      First, we prove \eqref{ineq_QLQU}.  Let $F$ be $(n,\varepsilon)$ spanning for $Z$ with respect to $\vect{T}$, that is, for every $x \in Z$, there exists $y \in F$ with $d_{n}^{\vect{T}}(x,y) \leq \varepsilon$.
        For every $x \in \overline{Z}$, suppose that  $x \in \overline{Z} \setminus Z$. Then for every $\delta>0$, there exists $x^{\prime} \in Z$ such that $d_{n}^{\vect{T}}(x,x^{\prime}) \leq \delta$. Hence  there exists $y \in F$ such that
        $$
        d_{n}^{\vect{T}}(x,y) \leq d_{n}^{\vect{T}}(x,x^{\prime}) + d_{n}^{\vect{T}}(x^{\prime},y) \leq \varepsilon+\delta.
        $$
        It follows that $F$ is a $(n,\varepsilon+\delta)$ spanning set for $\overline{Z}$, and by \eqref{eq:defQn},
        \begin{equation}\label{ineq_QEDQE}
        Q_{n}(\vect{T},\vect{f},\overline{Z},\varepsilon+\delta) \leq Q_{n}(\vect{T},\vect{f},Z,\varepsilon)
        \end{equation}
        for every $\delta>0$.

        For each $\varepsilon>0$, let $\delta=\varepsilon$. Combining \eqref{ineq_QEDQE} with \eqref{eq:Qe},  we have that
        $$
        \qrelow(\vect{T},\vect{f},\overline{Z},2\varepsilon) \leq \qrelow(\vect{T},\vect{f},Z,\varepsilon),
        \quad\text{and}\quad
        \qreup(\vect{T},\vect{f},\overline{Z},2\varepsilon) \leq \qreup(\vect{T},\vect{f},Z,\varepsilon).
        $$
        By \eqref{eq:defqre}, it follows that
        $$
        \qrelow(\vect{T},\vect{f},\overline{Z}) \leq \qrelow(\vect{T},\vect{f},Z)
        \quad\text{and}\quad
        \qreup(\vect{T},\vect{f},\overline{Z}) \leq \qreup(\vect{T},\vect{f},Z),
        $$
        which completes the proof for   \eqref{ineq_QLQU}.

        Next, we prove \eqref{ineq_PLPU}.  Let $E$ be a $(n,\varepsilon)$-separated set in $\overline{Z}$ with respect to $\vect{T}$.
Arbitrarily choose $r>1$  and  $0<\delta_{0}<\frac{\varepsilon}{4}$. For every real $\alpha$ such that $1<\e^{n\alpha} \leq r$, if $x \notin Z$,  there exists  $\delta_{x}$ with $0<\delta_{x}\leq\delta_{0}$  such that
        \begin{equation}\label{eq:Snfxpna}
          S_{n}^{\vect{T}}\vect{f}(x) \leq S_{n}^{\vect{T}}\vect{f}(x^{\prime})+n\alpha
        \end{equation}
        whenever $d_{n}^{\vect{T}}(x,x^{\prime})<\delta_{x}$, where real $\alpha$ satisfies $1<\e^{n\alpha} \leq r$.

        Since $E \subset \overline{Z}$, for every $x \in E$, if $x \in \overline{Z} \setminus Z$,
        then we choose  $x^{\prime} \in Z$ such that
        $$
        d_{n}^{\vect{T}}(x,x^{\prime}) \leq \delta_{x} \leq \delta_{0} <\frac{\varepsilon}{4}.
        $$
        Let $E^{\prime}=(E \cap Z) \cup \{x^{\prime}\}_{x \in \overline{Z} \setminus Z} \subset Z$.
        It is clear that  $d_{n}^{\vect{T}}(x,y)>\frac{\varepsilon}{2}$  for all $x\neq y \in E^{\prime}$, and this implies that $E^{\prime}$ is $(n,\frac{\varepsilon}{2})$ separated in $Z$. By \eqref{eq:defPn} and  \eqref{eq:Snfxpna}, we have
        \begin{align*}
          \sum_{x \in E}\e^{S_{n}^{\vect{T}}\vect{f}(x)} &= \sum_{x \in E \cap Z}\e^{S_{n}^{\vect{T}}\vect{f}(x)} + \sum_{x \in E \setminus Z}\e^{S_{n}^{\vect{T}}\vect{f}(x)} \\
                                                         &\leq \sum_{x \in E \cap Z}\e^{S_{n}^{\vect{T}}\vect{f}(x)} + \sum_{x \in E \setminus Z}\e^{S_{n}^{\vect{T}}\vect{f}(x^{\prime})+n\alpha} \\
                                                         &\leq \sum_{x \in E \cap Z}\e^{S_{n}^{\vect{T}}\vect{f}(x)} + r\sum_{x \in E \setminus Z}\e^{S_{n}^{\vect{T}}\vect{f}(x^{\prime})} \\
                                                         &\leq r\sum_{x \in E^{\prime}}\e^{S_{n}^{\vect{T}}\vect{f}(x)} \\
                                                         & \leq r P_{n}(\vect{T},\vect{f},Z,\frac{\varepsilon}{2}).
        \end{align*}
         Since it holds for all $(n,\varepsilon)$-separated sets and all $r>1$, it follows from  \eqref{eq:defPn} that
        $$
        P_{n}(\vect{T},\vect{f},\overline{Z},\varepsilon) \leq P_{n}(\vect{T},\vect{f},Z,\frac{\varepsilon}{2}).
        $$
        Therefore, by \eqref{eq:Pe} and \eqref{eq:defpre}, we obtain that
        $$
        \prelow(\vect{T},\vect{f},\overline{Z}) \leq \prelow(\vect{T},\vect{f},Z),
        \quad\text{and}\quad
        \preup(\vect{T},\vect{f},\overline{Z}) \leq \preup(\vect{T},\vect{f},Z),
        $$
        and it completes the proof for $\prelow$ and $\preup$.
      \end{proof}

      Since Bowen and packing pressures are defined by measures, it is clear that the Bowen and packing pressures are countably stable.
      \begin{prop}\label{prop:precountstab}
        Given $\vect{f} \in \vect{C}(\vect{X},\mathbb{R})$, if $Z=\bigcup_{i=1}^{\infty}Z_{i}$, then
        $$
        \prebow(\vect{T},\vect{f},Z)=\sup_{i =1,2,\ldots }\prebow(\vect{T},\vect{f},Z_{i}), \quad \prepac(\vect{T},\vect{f},Z)=\sup_{i  =1,2,\ldots}\prepac(\vect{T},\vect{f},Z_{i}).
        $$
      \end{prop}

In general, we do not have countable stability for the lower and upper topological pressures.   The upper topological pressure is finitely stable, but finite stability does not hold for the lower topological pressure.
      \begin{prop}\label{prop:presupstab}
       Given $\vect{f} \in \vect{C}(\vect{X},\mathbb{R})$,   if $Z=\bigcup_{i=1}^{\infty}Z_{i} \subset X_{0}$, then
        $$
        \prelow(\vect{T},\vect{f},Z) \geq \sup_{i = 1,2,\ldots}{\prelow(\vect{T},\vect{f},Z_{i})}
        \quad\text{and}\quad
        \preup(\vect{T},\vect{f},Z) \geq \sup_{i = 1,2,\ldots}{\preup(\vect{T},\vect{f},Z_{i})}
        $$
      \end{prop}

      \begin{prop}
   Given $\vect{f} \in \vect{C}(\vect{X},\mathbb{R})$, if $Z_{1}$ and $Z_{2} \subset X_{0}$, then
        $$
        \preup(\vect{T},\vect{f},Z_{1} \cup Z_{2}) = \max\{\preup(\vect{T},\vect{f},Z_{1}),\preup(\vect{T},\vect{f},Z_{2})\}
        $$
      \end{prop}

      The following result is analogous to \cite[Prop.2.8]{Falconer2014}, and it shows that the lower and upper topological pressures
      measure the `exact global maximum' of the local complexities of an NDS on a class of compact subsets, which might be described as `pressurely homogeneous'.
      \begin{thm}\label{thm:preeqpremod}
        Given   a non-empty compact subset $K \subset X_{0}$, let $\vect{f} \in \vect{C}(\vect{X},\R)$.
        \begin{enumerate}[(1)]
          \item Suppose $\prelow(\vect{T},\vect{f},K \cap V) = \prelow(\vect{T},\vect{f},K)$
                for all open sets $V$ such that   $V\cap K\neq \emptyset $. Then
                $$
                \prelow(\vect{T},\vect{f},K) = \inf \Big\{ \sup_{i=1,2,\ldots}{\prelow(\vect{T},\vect{f},F_{i})} : \bigcup_{i=1}^{\infty}F_{i} \supset K \Big\}.
                $$
  \item Suppose $ \preup(\vect{T},\vect{f},K \cap V) = \preup(\vect{T},\vect{f},K)$     for all open sets $V$ such that   $V\cap K\neq \emptyset $. Then
                \begin{equation}\label{eq:preupeqpremod}
                  \preup(\vect{T},\vect{f},K) = \inf \Big\{ \sup_{i=1,2,\ldots}{\preup(\vect{T},\vect{f},F_{i})} : \bigcup_{i=1}^{\infty}F_{i} \supset K \Big\}.
                \end{equation}
        \end{enumerate}
      \end{thm}
      \begin{proof}
        We only give  the proof of (2), and the argument for (1) is identical.   By Proposition~\ref{prop:presupstab}, it is sufficient to prove that
        $$
        \preup(\vect{T},\vect{f},K) \leq \inf \Big\{ \sup_{i=1,2,\ldots}{\preup(\vect{T},\vect{f},F_{i})} : \bigcup_{i=1}^{\infty}F_{i} \supset K \Big\}.
        $$

        By Proposition \ref{prop:preclstab}, it suffices to take the infimum in \eqref{eq:preupeqpremod} over decompositions where $F_{i}$ are all closed sets.
        Let $K \subset \bigcup_{i=1}^{\infty}F_{i}$ be a decomposition of $K$ with each $F_{i}$ closed.
        Considering the subspace topology on $K$, the compact set $K$ is of the second category.
        By Baire's category theorem (see  Munkres \cite[\S48]{Munkres2000} or Rudin \cite[2.2]{Rudin1973}),
        there exists $i_0$ and a non-empty set $V_{K}$ open in $K$ such that $V_{K} \subset F_{i_0}$.
       Hence there exists an open  set $V\subset X_{0}$ such that $\emptyset \neq V \cap K \subset F_{i_0}$
        Since  $ \preup(\vect{T},\vect{f},K \cap V) = \preup(\vect{T},\vect{f},K)$,     by   Proposition~\ref{prop:premonotone}, it follows that
        $$
        \preup(\vect{T},\vect{f},F_{i_0}) \geq \preup(\vect{T},\vect{f},F_{i_0} \cap V) = \preup(\vect{T},\vect{f},K \cap V) = \preup(\vect{T},\vect{f},K).
        $$
        Hence
        \begin{align*}
          \preup(\vect{T},\vect{f},K) &\leq \inf \Big\{ \sup_{i=1,2,\ldots}{\preup(\vect{T},\vect{f},F_{i})} : \bigcup_{i=1}^{\infty}F_{i} \supset K\ \text{where the}\ F_{i}\ \text{are closed sets}\Big\} \\
                                      &= \inf \Big\{ \sup_{i=1,2,\ldots}{\preup(\vect{T},\vect{f},F_{i})} : \bigcup_{i=1}^{\infty}F_{i} \supset K \Big\},
        \end{align*}
        which completes the proof.
      \end{proof}

    \subsection{Pressure inequalities}
Dimension inequalities \eqref{ineq_dim} are important in fractal geometry, and we obtain similar facts for   the   topological pressures that we have defined in this paper.
      \begin{prop}\label{prop:preneq}
        Given  a subset $Z \subset X_{0}$, let $\vect{f} \in \vect{C}(\vect{X},\R)$.
        Then
        \begin{enumerate}[(1)]
          \item $\prebow(\vect{T},\vect{f},Z) \leq \qrelow(\vect{T},\vect{f},Z) \leq \qreup(\vect{T},\vect{f},Z)$;
          \item $\prebow(\vect{T},\vect{f},Z) \leq \prelow(\vect{T},\vect{f},Z) \leq \preup(\vect{T},\vect{f},Z)$;
          \item $\prebow(\vect{T},\vect{f},Z) \leq \prepac(\vect{T},\vect{f},Z) \leq \preup(\vect{T},\vect{f},Z)$.
        \end{enumerate}
      \end{prop}
      \begin{proof}
        (1) For every integral $n>0$ and real $\varepsilon>0$, if $\{\overline{B}_{n}^{\vect{T}}(x_{i},\varepsilon)\}_{i=1}^{\infty}$ is a cover of $Z$,
        then  $\{B_{n}^{\vect{T}}(x_{i},2\varepsilon)\}_{i=1}^{\infty}$is also a cover of $Z$;
        moreover it is a $(n,2\varepsilon)$-cover of $Z$. By \eqref{eq:defcarpes} and \eqref{eq:defmsrbow}, for all $s \in \R$, we have that $ \carpes_{n,\varepsilon}^{s}(\vect{T},\vect{f},Z) \geq \msrbow_{n,2\varepsilon}^{s}(\vect{T},\vect{f},Z),$ and it implies that
        $$
        \carpeslow_{\varepsilon}^{s}(\vect{T},\vect{f},Z) \geq \msrbow_{2\varepsilon}^{s}(\vect{T},\vect{f},Z).
        $$
        By Proposition~\ref{prop:prepesincap}, \eqref{def_PBep} and Definition~\ref{def:prebpp}, this implies that
        $$
        \qrelow(\vect{T},\vect{f},Z) \geq \prebow(\vect{T},\vect{f},Z).
        $$

        (2) It is a direct consequence of   (1) and  Proposition~\ref{prop:PeqQ}.

        (3)             For every $N >0$ and $\varepsilon>0$, we choose a $(N, \varepsilon)$-packing $\{\overline{B}_{N}^{\vect{T}}(x_{i},\varepsilon)\}_{i\in \mathcal{I}}$ of $Z$ with maximal cardinality, that is, for every $x\in Z$, there exists $i\in \mathcal{I}$ such that
            $$
            \overline{B}_{N}^{\vect{T}}(x_{i},\varepsilon)\cap \overline{B}_{N}^{\vect{T}}(x,\varepsilon)\neq \emptyset.
            $$
            Hence, for every $y \in Z$, there exists  $i\in \mathcal{I}$ such that $d_{N}^{\vect{T}}(x_{i},y) \leq 2\varepsilon$, which implies that the collection $\{B_{N}^{\vect{T}}(x_{i},3\varepsilon)\}_{i \in \mathcal{I}}$ is a cover of $Z$. Since $X_0$ is compact, it is clear that $\mathcal{I}$ is finite.

            Given $s \in \mathbb{R}$, for sufficiently small $\varepsilon>0$, by \eqref{eq:defmsrbow} and \eqref{eq:defmsrPack},
            we have
            $$
            \msrbow_{N,3\varepsilon}^{s}(\vect{T},\vect{f},Z) \leq \sum_{i\in \mathcal{I}}\exp{\left(-Ns+S_{N}^{\vect{T}}\vect{f}(x_{i})\right)} \leq \msrpac_{N,\varepsilon}^{s}(\vect{T},\vect{f},Z).
            $$
            Letting  $N\to \infty$, it implies that
            \begin{equation}\label{eq:msrbowleqmsrpac}
                \msrbow_{3\varepsilon}^{s}(\vect{T},\vect{f},Z) \leq \msrpac_{\infty,\varepsilon}^{s}(\vect{T},\vect{f},Z).
            \end{equation}
            Since $\msrbow_{3\varepsilon}^{s}$ is an outer measure,  by \eqref{eq:msrbowleqmsrpac}, we obtain that
            \begin{align*}
                \msrbow_{3\varepsilon}^{s}(\vect{T},\vect{f},Z) &\leq \inf\Big\{\sum_{i=1}^{\infty}\msrbow_{3\varepsilon}^{s}(\vect{T},\vect{f},Z_{i}): \bigcup_{i=1}^{\infty}Z_{i} \supset Z\Big\} \\
                &\leq \inf\Big\{\sum_{i=1}^{\infty}\msrpac_{\infty,\varepsilon}^{s}(\vect{T},\vect{f},Z_{i}): \bigcup_{i=1}^{\infty}Z_{i} \supset Z\Big\} \\
                &= \msrpac_{\varepsilon}^{s}(\vect{T},\vect{f},Z).
            \end{align*}

For every $s<\prebow(\vect{T},\vect{f},Z)$, by Proposition~\ref{prop:msrbowepsilon} and \eqref{def_PBep},
           for  all sufficiently small $\varepsilon>0$, we have that $\msrpac_{\varepsilon}^{s}(\vect{T},\vect{f},Z) \geq \msrbow_{3\varepsilon}^{s}(\vect{T},\vect{f},Z)=+\infty$. By \eqref{def_PP} and Proposition \ref{prop:msrpacepsilon},  for all  $s<\prebow(\vect{T},\vect{f},Z)$, we have that
            $$
            \prepac(\vect{T},\vect{f},Z) \geq \prepac(\vect{T},\vect{f},Z,\varepsilon) \geq s,
            $$
and it follows that
            $\prebow(\vect{T},\vect{f},Z) \leq \prepac(\vect{T},\vect{f},Z).$

            To show $\prepac(\vect{T},\vect{f},Z) \leq \preup(\vect{T},\vect{f},Z)$, note that by \eqref{def_pes}, for all $s \in \R$,
        $$
        \msrpac_{\infty,\varepsilon}^{s}(\vect{T},\vect{f},Z) \geq \msrpac_{\varepsilon}^{s}(\vect{T},\vect{f},Z),
        $$
        and it immediate follows by  Proposition~\ref{prop:preupeqpac} and Definition~\ref{def:prepac}.
      \end{proof}

      A natural question raised is, under what conditions do the equalities in Proposition \ref{prop:preneq} hold?
      The question is of great importance from the perspective of both dynamical systems and dimension theory.
      For the autonomous dynamical system $(X,T)$,  given potential $f \in C(X,\R)$,   if $Z \subset X$ is $T$-invariant, then $\prelow(T,f,Z)=\preup(T,f,Z)$. Moreover, if $Z$ is compact and $T$-invariant\footnote{In this case, the pair $(Z,T\vert_{Z})$ becomes an autonomous subsystem.},
      then $\prebow(T,f,Z)=\preup(T,f,Z)$, and hence all the pressures coincide, that is,
      $$
      \prebow(T,f,Z)=\prelow(T,f,Z)=\prepac(T,f,Z)=\preup(T,f,Z).
      $$
See \cite{Pesin1997} for details.      In view of dimension theory, the coincidence of these topological pressures indicates a certain type of `dynamical regularity' of the subsets, such as invariance under the dynamics.

      We follow the spirits of the dimension theory of fractals and explore a particular general case of packing pressures coinciding with upper topological pressures.   We start with a result resembling \cite[Prop.3.9]{Falconer2014} which provides an exact characterization of  packing pressures with upper topological pressures.
      The theorem essentially states that the extra step of decomposing the given subset $Z$ in \eqref{def_pes} and the step of giving a dimensional structure for pressure are exchangeable,
      and this slightly relieves the difficulty of   calculating packing pressures.
      \begin{thm}\label{thm:prepaceqpremod}
        Given  a subset $Z \subset X_{0}$ and  $\vect{f} \in \vect{C}(\vect{X},\R)$,
        $$
        \prepac(\vect{T},\vect{f},Z) = \inf \Big\{ \sup_{i=1,2,\ldots}{\preup(\vect{T},\vect{f},Z_{i})} : \bigcup_{i=1}^{\infty}Z_{i} \supset Z \Big\}.
        $$
      \end{thm}
      \begin{proof}
        If $Z \subset \bigcup_{i=1}^{\infty}Z_{i}$ with $Z_{i}\neq \emptyset$, then by Proposition~\ref{prop:precountstab} and Proposition~\ref{prop:preneq}(2),
        $$
        \prepac(\vect{T},\vect{f},Z) \leq \sup_{i=1,2,\ldots}{\prepac(\vect{T},\vect{f},Z_{i})} \leq \sup_{i=1,2,\ldots}{\preup(\vect{T},\vect{f},Z_{i})}.
        $$
        Since this holds for all countable decompositions $\{Z_{i}\}_{i=1}^{\infty}$ of $Z$, we have
        $$
        \prepac(\vect{T},\vect{f},Z) \leq \inf \Big\{ \sup_{i=1,2,\ldots}{\preup(\vect{T},\vect{f},Z_{i})} : \bigcup_{i=1}^{\infty}Z_{i} \supset Z \Big\}.
        $$

        Conversely, it suffices to find a decomposition $\{Z_{i}\}_{i=1}^{\infty}$ of $Z$ such that
        $$
        \sup_{i=1,2,\ldots}{\preup(\vect{T},\vect{f},Z_{i})} \leq \prepac(\vect{T},\vect{f},Z).
        $$
 Arbitrarily choose $s>\prepac(\vect{T},\vect{f},Z)$, and by Proposition~\ref{prop:msrpacepsilon}, it is clear that $s>\prepac(\vect{T},\vect{f},Z,\varepsilon)$ for all $\varepsilon>0$. Hence by \eqref{def_PP}, $\msrpac_{\varepsilon}^{s}(\vect{T},\vect{f},Z)=0$. By \eqref{def_pes}, there exists $\{Z_{i}\}_{i=1}^\infty$  such that $Z \subset \bigcup_{i=1}^{\infty}Z_{i}$ satisfying  $\msrpac_{\infty,\varepsilon}^{s}(\vect{T},\vect{f},Z_{i})<+\infty$.
        By Proposition~\ref{prop:preupeqpac}, it follows  that $\preup(\vect{T},\vect{f},Z_{i}) \leq s$ for all $i$, which implies
        $$
        \sup_{i=1,2,\ldots}{\preup(\vect{T},\vect{f},Z_{i})} \leq s.
        $$
Hence  by the arbitrariness of $s$, we obtain  $\sup_{i=1,2,\ldots}{\preup(\vect{T},\vect{f},Z_{i})} \leq \prepac(\vect{T},\vect{f},Z).$
      \end{proof}

The following conclusions are direct consequences of   Theorem~\ref{thm:prepaceqpremod} and   Theorem~\ref{thm:preeqpremod}.
      \begin{cor}
        Given   a subset $Z \subset X_{0}$ and $\vect{f} \in \vect{C}(\vect{X},\R)$,
        $$
        \prepac(\vect{T},\vect{f},Z) \geq \inf \Big\{ \sup_{i=1,2,\ldots}{\prelow(\vect{T},\vect{f},Z_{i})} : \bigcup_{i=1}^{\infty}Z_{i} \supset Z \Big\}.
        $$
      \end{cor}

      \begin{cor}\label{cor:PupeqPP}
        Given   a non-empty compact subset $K \subset X_{0}$,
        if      $\preup(\vect{T},\vect{f},K \cap V) = \preup(\vect{T},\vect{f},K)$ for all open sets $V$ such that  $V\cap K\neq \emptyset$, then
        $$
        \preup(\vect{T},\vect{f},K)=\prepac(\vect{T},\vect{f},K).
        $$
      \end{cor}

  \section{Dynamical Properties of Pressures}\label{sect:Prop}
  In this section, we give some useful properties of the topological pressures   from the  dynamical point of view.

    \subsection{Properties with respect to potentials}
      Since the topological pressures are generalizations of the ones in topological dynamical systems,
      they still keep some similar properties; see \cite{Walters1982, Walters1975} for details on $(X, T)$.

For   a fixed  subset $Z \subset X_{0}$, we regard   the pressures on $Z$  as mappings $\vect{C}(\vect{X},\R) \to \R \cup \{\pm\infty\}$ in this subsection.
      The following property  is a direct consequence of the definitions, and it shows certain additivity of pressures.
      \begin{prop}
        Given $Z \subset X_{0}$, for all $\vect{f} \in \vect{C}(\vect{X},\R)$ and $a \in \R$,
        $$
        P(\vect{T},\vect{f}+a\vect{1},Z) = P(\vect{T},\vect{f},Z)+a,
        $$
       for every pressure $P \in \{\prebow, \prepac, \prelow, \preup, \qrelow, \qreup\}$.
      \end{prop}

Note that the potential $a\vect{1}$ is equicontinuous, and so $\preL(\vect{T},a\vect{1},Z)$ and  $ \preU(\vect{T},a\vect{1},Z)$ exist.
      The following conclusion  reveals the  relation of pressures and  their  entropies.
     \begin{cor}\label{prop:Peqhplusa}
        Given $Z \subset X_{0}$, for every pressure $P \in \{\prebow, \prepac, \preL,\preU \}$, let $h$ be   the corresponding entropy of $P$.
        Then for every $a \in \mathbb{R}$,
        $$
        P(\vect{T},a\vect{1},Z) = h(\vect{T},Z)+a.
        $$
      \end{cor}

      The following conclusion shows the positivity  of pressures.
      \begin{prop}\label{prop:positivity}
        Given $Z \subset X_{0}$, if $\vect{f} \preceq \vect{g}$, then
        $$
        P(\vect{T},\vect{f},Z) \leq P(\vect{T},\vect{g},Z),
        $$
      for every $P \in \{\prebow, \prepac, \prelow, \preup, \qrelow, \qreup\}$;
        in particular, for $\vect{f} \succeq \vect{0}$,
        $$
        0 \leq h(\vect{T},Z) \leq P(\vect{T},\vect{f},Z),
        $$
        where $h$ is the corresponding entropy of $P$.
      \end{prop}

       The positivity property of  pressures  provides lower and upper bounds for the pressures of a given potential.
      \begin{prop}
        Given $Z \subset X_{0}$, for all $\vect{f} \in \vect{C}(\vect{X},\R)$,
        $$
        \lvert P(\vect{T},\vect{f},Z) \rvert \leq P(\vect{T},\lvert\vect{f}\rvert,Z),
        $$
       for  $P \in \{\prebow, \prepac, \prelow, \preup, \qrelow, \qreup\}$, where $\lvert\vect{f}\rvert$ is defined by \eqref{def_absf}.
      \end{prop}

      The value of pressures may vary dramatically in $\mathbb{R} \cup \{\pm\infty\}$, but the next conclusion shows  if we restrict potentials to $\vect{C}_{b}(\vect{X},\mathbb{R})$,  pressures are bounded by their entropies and the bounds of potentials.      For $\vect{f} \in \vect{C}_{b}(\vect{X},\mathbb{R})$, we write
      $$
      \inf{\vect{f}} =\inf_{k \in \mathbb{N}}{(\inf{f_{k}})}\qquad \text{and}\qquad \sup{\vect{f}}=\sup_{k \in \mathbb{N}}{(\sup{f_{k}})}.
      $$
      \begin{prop}
For each    $P \in \{\prebow, \prepac, \prelow, \preup, \qrelow, \qreup\}$,  let  $h$ be the   corresponding entropy of $P$.
        Given $Z \subset X_{0}$, if $\vect{f} \in \vect{C}_{b}(\vect{X},\mathbb{R})$, then
        $$
        h(\vect{T},Z) + \inf{\vect{f}} \leq P(\vect{T},\vect{f},Z) \leq h(\vect{T},Z) + \sup{\vect{f}},
        $$
        Moreover, $P(\vect{T},\cdot,Z)$ is either finite or constantly $\infty$ on $\vect{C}_{b}(\vect{X},\mathbb{R})$.
      \end{prop}
      \begin{proof}
        The inequalities are immediate from Proposition~\ref{prop:positivity} and Proposition~\ref{prop:Peqhplusa}, and for all $\vect{f} \in \vect{C}_{b}(\vect{X},\R)$, $P(\vect{T},\vect{f},Z)=\infty$ if and only if $h(\vect{T},Z)=\infty$.
      \end{proof}

      The continuity property of the pressures in potentials is essential to many fields of study, and  one of the main objects in the thermodynamic formalism is to study the differentiability and analyticity with respect to the potentials of  pressures. We refer readers to   \cite{Ruelle2004} for the background reading.
      \begin{thm} \label{thm_plfg}
        For all $\vect{f}$ and $\vect{g} \in \vect{C}(\vect{X},\R)$, we have
        \begin{equation}\label{eq:prects}
          \lvert P(\vect{T},\vect{f},Z)-P(\vect{T},\vect{g},Z) \rvert \leq \|\vect{f}-\vect{g}\|,
        \end{equation}
        where $P \in \{\prebow, \prepac, \prelow, \preup, \qrelow, \qreup\}$.
      \end{thm}
      \begin{proof}
        The inequality \eqref{eq:prects} obviously holds if $\|\vect{f}-\vect{g}\|=+\infty$.

        If  $\|\vect{f}-\vect{g}\|<+\infty$,  for every integer $n>0$ and all $x \in X_{0}$, it is clear that
        \begin{equation}\label{eq:Snfmg}
          \lvert S_{n}^{\vect{T}}\vect{f}(x)-S_{n}^{\vect{T}}\vect{g}(x) \rvert \leq n\|\vect{f}-\vect{g}\|.
        \end{equation}

        We prove \eqref{eq:prects} for $\prelow$ and $\preup$ first.        By \eqref{eq:defPn} and  \eqref{eq:Snfmg}, for all integral $n>0$, we have that
        \begin{align*}
          \frac{P_{n}(\vect{T},\vect{f},Z,\varepsilon)}{P_{n}(\vect{T},\vect{g},Z,\varepsilon)} &\leq \sup\left\{\frac{\sum_{x \in E}\e^{S_{n}^{\vect{T}}\vect{f}(x)}}{\sum_{x \in E}\e^{S_{n}^{\vect{T}}\vect{g}(x)}} : E \ \text{is a}\ (n,\varepsilon)\text{-separated set for}\ Z\  \right\} \\
                                                                                            &\leq \sup\left\{\max_{x \in E}\frac{\e^{S_{n}^{\vect{T}}\vect{f}(x)}}{\e^{S_{n}^{\vect{T}}\vect{g}(x)}} : E \ \text{is a}\ (n,\varepsilon)\text{-separated set for}\ Z\  \right\}\\
 &\leq \e^{n\|\vect{f}-\vect{g}\|}.
        \end{align*}
Since   $\vect{f}$ and $\vect{g}$  are symmetric in the argument,  it follows that
        $$
        \big|\frac{1}{n}\log{P_{n}(\vect{T},\vect{f},Z,\varepsilon)} - \frac{1}{n}\log{P_{n}(\vect{T},\vect{g},Z,\varepsilon)} \big|\leq \|\vect{f}-\vect{g}\|.
        $$
For $P \in \{  \prelow, \preup\}$, by \eqref{eq:Pe}, \eqref{eq:defpre}, it follows that
        $$
     \lvert P(\vect{T},\vect{f},Z)-P(\vect{T},\vect{g},Z) \rvert \leq \|\vect{f}-\vect{g}\|.
        $$

        Next, we prove \eqref{eq:prects} for $\prebow$ and $\prepac$.
        By \eqref{eq:Snfmg}, it follows from \eqref{eq:defmsrbow} and \eqref{eq:defmsrPack}
        that for all $s \in \R$, $N>0$, and $\varepsilon>0$,
        $$
        \msrbow_{N,\varepsilon}^{s+\|\vect{f}-\vect{g}\|}(\vect{T},\vect{g},Z) \leq \msrbow_{N,\varepsilon}^{s}(\vect{T},\vect{f},Z) \leq \msrbow_{N,\varepsilon}^{s-\|\vect{f}-\vect{g}\|}(\vect{T},\vect{g},Z),
        $$
        and
        $$
        \msrpac_{N,\varepsilon}^{s+\|\vect{f}-\vect{g}\|}(\vect{T},\vect{g},Z) \leq \msrpac_{N,\varepsilon}^{s}(\vect{T},\vect{f},Z) \leq \msrpac_{N,\varepsilon}^{s-\|\vect{f}-\vect{g}\|}(\vect{T},\vect{g},Z).
        $$
By \eqref{def_PBep} and Definition~\ref{def:prebpp},  this implies that the inequality \eqref{eq:prects} holds for $\prebow$,
        and by \eqref{def_pes},\eqref{def_PP} and Definition~\ref{def:prepac},  the inequality \eqref{eq:prects} holds for $\prepac$.

Recall that  $\qrelow$ and $\qreup$ have equivalent definitions by using  $\carpes_{n,\varepsilon}^{s}$ given by \eqref{eq:defcarpes}.
We apply the same arguments of  $\prebow$ and $\prepac$  to   $\carpes_{n,\varepsilon}^{s}$, and the conclusion holds.
      \end{proof}
    Note that the argument for $\prelow$ and $\preup$ in Theorem \ref{thm_plfg} does not work for $\qrelow$ and $\qreup$; see \cite[Thm.2.1(\romannumeral5)]{Walters1975} and \cite[Thm.9.7(\romannumeral4)]{Walters1982}.

A more precise estimate on the difference between $\vect{f}$ and $\vect{g}$ in \eqref{eq:Snfmg} provides sharper results.
      In particular,   the following theorem indicates that the tail of the potential $\vect{f}$ plays the dominating role in pressures.
      \begin{thm}
        Given   $Z \subset X_{0}$. If $\vect{f}$ and $\vect{g} \in \vect{C}(\vect{X},\R)$ coincide except for a finite many of $f_{k} \neq g_{k}$,
        i.e., there exists an integer $N \in \N$ such that $f_{k}=g_{k}$ for all $k \geq N$, then
        $$
        P(\vect{T},\vect{f},Z) = P(\vect{T},\vect{g},Z),
        $$
        where $P \in \{\prebow, \prepac, \prelow, \preup, \qrelow, \qreup\}$.
      \end{thm}
      \begin{proof}
     Since $f_{k}=g_{k}$ for all $k \geq N$, we improve the estimate of \eqref{eq:Snfmg}  to
        $$
        \lvert S_{n}^{\vect{T}}\vect{f}(x)-S_{n}^{\vect{T}}\vect{g}(x) \rvert = \lvert S_{N}^{\vect{T}}(\vect{f}-\vect{g})(x) \rvert \leq \sum_{j=0}^{N-1}\|f_{j}-g_{j}\|_{\infty},
        $$
and we have that
    \begin{equation}\label{ineq_SnSnLsum}
        \frac{\e^{S_{n}^{\vect{T}}\vect{f}(x)}}{\e^{S_{n}^{\vect{T}}\vect{g}(x)}} = \e^{S_{N}^{\vect{T}}(\vect{f}-\vect{g})(x)}\leq \e^{\sum_{j=0}^{N-1}\|f_{j}-g_{j}\|_{\infty}}.
        \end{equation}
Then, the conclusion follows by the same argument as Theorem \ref{thm_plfg} where we only replace \eqref{eq:Snfmg} by \eqref{ineq_SnSnLsum}.
      \end{proof}
Next, we need the following simple fact in our proofs.
\begin{lem}\label{eq:simpneq}
Let $\{a_{i}\}_{i \in \mathcal{I}}$ be a  countable collection of positive numbers such that $\sum_{i \in \mathcal{I}}a_{i}\leq1$. Then for all $c>0$,
        \begin{equation}
          \sum_{i \in \mathcal{I}}a_{i}^{c}\left\{
          \begin{array}{ll}
             \leq 1  & \text{if}\ c \geq 1, \\
             \geq 1 & \text{if}\ c \leq 1.
        \end{array}            \right.
        \end{equation}
\end{lem}

       \begin{prop}
        Given   $Z \subset X_{0}$   and $\vect{f} \in \vect{C}(\vect{X},\R)$, for all $c>0$, we have that
        \begin{equation}\label{eq:Pcf}
          \begin{split}
            P(\vect{T},c\vect{f},Z) \leq c P(\vect{T},\vect{f},Z) &\quad \text{if}\ c \geq 1, \\
           P(\vect{T},c\vect{f},Z) \geq c P(\vect{T},\vect{f},Z) &\quad \text{if}\ c \leq 1,
          \end{split}
        \end{equation}
where $P \in \{\prebow, \prepac, \prelow, \preup, \qrelow, \qreup\}$.
      \end{prop}
\begin{proof}
        We give the proof of  \eqref{eq:Pcf} for $\prelow$ and $\preup$, and  the  argument  for $\qrelow$ and $\qreup$ is similar.

        Let $E$ be a maximal $(n,\varepsilon)$-separated set for $Z$.
        It follows by Lemma \ref{eq:simpneq} that
        $$
        \sum_{x \in E}\e^{S_{n}^{\vect{T}}(c\vect{f})(x)} \leq \left(\sum_{x \in E}\e^{S_{n}^{\vect{T}}\vect{f}(x)}\right)^{c} \quad \text{if}\ c \geq 1,
        $$
        and
        $$
        \sum_{x \in E}\e^{S_{n}^{\vect{T}}(c\vect{f})(x)} \geq \left(\sum_{x \in E}\e^{S_{n}^{\vect{T}}\vect{f}(x)}\right)^{c} \quad \text{if}\ c \leq 1.
        $$
        Combined with \eqref{eq:defPn}, this implies that
        $$
        P_{n}(\vect{T},c\vect{f},Z,\varepsilon) \leq (P_{n}(\vect{T},\vect{f},Z,\varepsilon))^{c} \quad \text{if}\ c \geq 1,
        $$
        and
        $$
        P_{n}(\vect{T},c\vect{f},Z,\varepsilon) \geq (P_{n}(\vect{T},\vect{f},Z,\varepsilon))^{c} \quad \text{if}\ c \leq 1.
        $$
        Therefore \eqref{eq:Pcf} holds for $\prelow$ and $\preup$ by \eqref{eq:Pe} and \eqref{eq:defpre}.

Since the proof for $P^B$ is similar but simpler to $P^P$, we    only  give the proof \eqref{eq:Pcf} for $\prepac$.

        If $c>1$, then for every given $s>c\prepac(\vect{T},\vect{f},Z)$, by Proposition~\ref{prop:msrpacepsilon}, we have
        $$
        \prepac(\vect{T},\vect{f},Z,\varepsilon) \leq \prepac(\vect{T},\vect{f},Z) <\frac{s}{c}
        $$
        for all $\varepsilon>0$. It follows by \eqref{def_PP} that $\msrpac_{\varepsilon}^{\frac{s}{c}}(\vect{T},\vect{f},Z)=0$,
        and hence for every $\eta$ with $0<\eta\leq1$, there exists a decomposition $\{Z_{i}\}_{i=1}^{\infty}$ of $Z$ such that $\bigcup_{i=1}^{\infty}Z_{i} \supset Z$
        and
        \begin{equation}\label{eq:pacscZi}
          \sum_{i=1}^{\infty}\msrpac_{\infty,\varepsilon}^{\frac{s}{c}}(\vect{T},\vect{f},Z_{i}) < \eta,
        \end{equation}
        which implies that $\msrpac_{\infty,\varepsilon}^{\frac{s}{c}}(\vect{T},\vect{f},Z_{i}) < \eta$ for all $i$.
        Fix $i$. By the definition of $\msrpac_{\infty,\varepsilon}^{\frac{s}{c}}$ and \eqref{eq:defmsrPack}, there exists  $N_{i}>0$ such that
        for all $N \geq N_{i}$ and for all countable $(N,\varepsilon)$-packings $\{\overline{B}_{n_{l}}^{\vect{T}}(x_{l},\varepsilon)\}_{l=1}^{\infty}$ of $Z_{i}$,
        $$
        0 < \sum_{l=1}^{\infty}\exp{\Big(-n_{l}\Big(\frac{s}{c}\Big) + S_{n_{l}}^{\vect{T}}\vect{f}(x_{l})\Big)} < \eta \leq 1.
        $$
        By Lemma~\ref{eq:simpneq}, it follows that
        \begin{align*}
        \sum_{l=1}^{\infty}\exp{\left(-n_{l}s+S_{n_{l}}^{\vect{T}}(c\vect{f})(x_{l})\right)} &= \sum_{l=1}^{\infty}\Big[ \exp{\Big(-n_{l}\Big(\frac{s}{c}\Big) + S_{n_{l}}^{\vect{T}}\vect{f}(x_{l})\Big)} \Big]^{c} \\
                                                                               &\leq \sum_{l=1}^{\infty}\exp{\Big(-n_{l}\Big(\frac{s}{c}\Big) + S_{n_{l}}^{\vect{T}}\vect{f}(x_{l})\Big)}.
        \end{align*}
        Since the $(N,\varepsilon)$-packing of $Z_{i}$ is arbitrarily chosen, this implies by \eqref{eq:defmsrPack} that $$\msrpac_{N,\varepsilon}^{s}(\vect{T},c\vect{f},Z_{i}) \leq \msrpac_{N,\varepsilon}^{\frac{s}{c}}(\vect{T},\vect{f},Z_{i})$$ for all $N \geq N_{i}$,
        and hence
$$
\msrpac_{\infty,\varepsilon}^{s}(\vect{T},c\vect{f},Z_{i}) \leq \msrpac_{\infty,\varepsilon}^{\frac{s}{c}}(\vect{T},\vect{f},Z_{i}).
$$
       This is true for every $i=1,2,\ldots$, and combining it with \eqref{eq:pacscZi} and \eqref{def_pes} implies that
        \begin{align*}
          \msrpac_{\varepsilon}^{s}(\vect{T},c\vect{f},Z) &\leq \sum_{i=1}^{\infty}\msrpac_{\infty,\varepsilon}^{s}(\vect{T},c\vect{f},Z_{i})\leq \sum_{i=1}^{\infty}\msrpac_{\infty,\varepsilon}^{\frac{s}{c}}(\vect{T},\vect{f},Z_{i}) < \eta.
        \end{align*}
        By the arbitrariness of $\eta$, we have that $\msrpac_{\varepsilon}^{s}(\vect{T},c\vect{f},Z)=0$, whence $\prepac(\vect{T},c\vect{f},Z,\varepsilon)<s$ for all $\varepsilon>0$.
        It follows by Definition~\ref{def:prepac} that $\prepac(\vect{T},c\vect{f},Z) \leq s$,
        and thus
        $$
        \prepac(\vect{T},c\vect{f},Z) \leq c\prepac(\vect{T},\vect{f},Z)
        $$
        by the arbitrariness of $s$.

        If $c<1$, then for every given $s>\frac{1}{c}\prepac(\vect{T},c\vect{f},Z)$, we have by Proposition~\ref{prop:msrpacepsilon} that
        $$
        \prepac(\vect{T},c\vect{f},Z,\varepsilon) \leq \prepac(\vect{T},c\vect{f},Z) < cs,
        $$
        and hence by \eqref{def_PP} that $\msrpac_{\varepsilon}^{cs}(\vect{T},c\vect{f},Z)=0$ for all $\varepsilon>0$.
        This implies by \eqref{def_pes} that for every $\eta$ with $0<\eta\leq1$, there exists a decomposition $\{Z_{i}\}_{i=1}^{\infty}$ of $Z$ such that
        $\bigcup_{i=1}^{\infty}Z_{i} \supset Z$ and
        \begin{equation}\label{eq:paccsZi}
          \sum_{i=1}^{\infty}\msrpac_{\infty,\varepsilon}^{cs}(\vect{T},c\vect{f},Z_{i}) < \eta.
        \end{equation}
        In particular, $\msrpac_{\infty,\varepsilon}^{cs}(\vect{T},c\vect{f},Z_{i}) < \eta$ for all $i$.
        Fix $i$. By the definition of $\msrpac_{\infty,\varepsilon}^{cs}$ and \eqref{eq:defmsrPack},
        there exists $N_{i}>0$ such that for all $N \geq N_{i}$ and all $(N,\varepsilon)$-packings $\{\overline{B}_{n_{l}}^{\vect{T}}(x_{l},\varepsilon)\}_{l=1}^{\infty}$ of $Z_{i}$,
        $$
        0 < \sum_{l=1}^{\infty}\exp{\left(-n_{l}(cs)+S_{n_{l}}^{\vect{T}}(c\vect{f})(x_{l})\right)} < \eta \leq 1,
        $$
        which implies by Lemma~\ref{eq:simpneq} that
        \begin{align*}
          \sum_{l=1}^{\infty}\exp{\left(-n_{l}s+S_{n_{l}}^{\vect{T}}\vect{f}(x_{l})\right)} &\leq \sum_{l=1}^{\infty}\left[\exp{\left(-n_{l}s+S_{n_{l}}^{\vect{T}}\vect{f}(x_{l})\right)}\right]^{c} \\
                                                                                            &= \sum_{l=1}^{\infty}\exp{\left(-n_{l}(cs) + S_{n_{l}}^{\vect{T}}(c\vect{f})(x_{l})\right)}.
        \end{align*}
        It follows by the arbitrariness of the $(N,\varepsilon)$-packing $\{\overline{B}_{n_{l}}^{\vect{T}}(x_{l},\varepsilon)\}_{l=1}^{\infty}$ that
        $$
        \msrpac_{N,\varepsilon}^{s}(\vect{T},\vect{f},Z_{i}) \leq \msrpac_{N,\varepsilon}^{cs}(\vect{T},c\vect{f},Z_{i})
        $$
        for all $N \geq N_{i}$, and therefore
        $$
        \msrpac_{\infty,\varepsilon}^{s}(\vect{T},\vect{f},Z_{i}) \leq \msrpac_{\infty,\varepsilon}^{cs}(\vect{T},c\vect{f},Z_{i})
        $$
        for every $i=1,2,\cdots$. Combined with \eqref{eq:paccsZi} and \eqref{def_pes}, this implies that
        \begin{align*}
          \msrpac_{\varepsilon}^{s}(\vect{T},\vect{f},Z) &\leq \sum_{i=1}^{\infty}\msrpac_{\infty,\varepsilon}^{s}(\vect{T},\vect{f},Z_{i}) \leq \sum_{i=1}^{\infty}\msrpac_{\infty,\varepsilon}^{cs}(\vect{T},c\vect{f},Z_{i})<\eta.
        \end{align*}
        Since $\eta>0$ is arbitrary, we have $\msrpac_{\varepsilon}^{s}(\vect{T},\vect{f},Z) = 0,$
        and thus $\prepac(\vect{T},\vect{f},Z,\varepsilon)<s$ for all $\varepsilon>0$.
        By Definition~\ref{def:prepac} and the arbitrariness of $s$, we conclude that $c\prepac(\vect{T},\vect{f},Z) \leq \prepac(\vect{T},c\vect{f},Z)$.
      \end{proof}

      An interpretation for the pressure of dynamics on subsets as embedded subsystems is more direct in the context of nonautonomous dynamics.
      Given an NDS $(\vect{X},\vect{T})$ and a closed (compact) subset $K \subset X_{0}$, we write
      $$
      \vect{X}\vert_{K}=\{\vect{T}^{k}K\}_{k=0}^{\infty}
      \quad\text{and}\quad
      \vect{T}\vert_{K}=\{T_{k}\vert_{\vect{T}^{k}K}\}_{k=0}^{\infty}.
      $$
      It is clear that $(\vect{X}\vert_{K},\vect{T}\vert_{K})$ is an NDS. For all $\vect{f} \in \vect{C}(\vect{X},\R)$,
      we write
      $$
      \vect{f}\vert_{K}=\{f_{k}\vert_{\vect{T}^{k}K}\}_{k=0}^{\infty},
      $$
      which is obviously in $\vect{C}(\vect{X}\vert_{K},\R)$.
      \begin{prop}
 Let $K\subset X_0$ be  non-empty compact.
        If $Z \subset K$, then for each $P\in \{\prebow, \prepac, \prelow, \preup, \qrelow, \qreup\}$,
        we have that
        $$
        P(\vect{T},\vect{f},Z) = P(\vect{T}\vert_{K},\vect{f}\vert_{K},Z)
        $$
        for all $\vect{f} \in \vect{C}(\vect{X},\R)$.
      \end{prop}

    \subsection{Power rules}
      In this subsection, we discuss the power rules for pressures of NDSs, which     important in proving various variational principles.
      One of the current difficulty in generalizing the variational principles to the lower and upper topological pressures in NDSs
      is to formulate the suitable corresponding notions of measure-theoretic entropy (pressure) with power rules.
      We refer  readers to \cite{Cao&Feng&Huang2008,Falconer1988,Kawan2014, Kolyada&Snoha1996,Wang&Yang&Zhang2024, ZLXZ2012} for related studies on TDS $(X, T)$ and NDS $(X, \vect{T})$ with identical space.

      Recall that the $m$-th power system $(\vect{X}^{[m]},\vect{T}^{[m]})$ of a given NDS $(\vect{X},\vect{T})$ is given by
      $$
      X_{k}^{[m]} = X_{k \cdot m}
      \quad\text{and}\quad
      T_{k}^{[m]} = \vect{T}_{k \cdot m}^{m}=T_{km+(m-1)} \circ \cdots \circ T_{km}: X_{km} \to X_{(k+1)m}.
      $$
 Note that for all integers $k \geq 0$, $m \geq 1$ and $n \geq 1$,
      \begin{equation}\label{eq:dnm}
        d_{k,n}^{\vect{T}^{[m]}}(x,y) \leq d_{km,nm}^{\vect{T}}(x,y)
      \end{equation}
      for all $x,y \in X_{km}$.
      Given $\vect{f} \in \vect{C}(\vect{X},\R)$, by \eqref{eq:sumknTf}, we write $f_{k}^{[m]} =S_{km,m}^{\vect{T}}\vect{f}$ and  define the induced power of potentials by
      $$
      \vect{f}^{[m]}=\{f_{k}^{[m]}\}_{k=0}^{\infty}=\{S_{km,m}^{\vect{T}}\vect{f}\}_{k=0}^{\infty}.
      $$
It is clear that
      \begin{equation}\label{eq:Snm}
        \begin{aligned}
          S_{n}^{\vect{T}^{[m]}}\vect{f}^{[m]}(x) 
                                                  &= S_{nm}^{\vect{T}}\vect{f}(x)
        \end{aligned}
      \end{equation}
      for all $x \in X_{0}$.

      We first show a part of the power rules in general.
      \begin{lem}\label{lem:Pmpowneq}
   Given $\vect{f} \in \vect{C}(\vect{X},\R)$ and $Z\subset X_0$.    Let $P \in \{\prepac, \prelow, \preup, \qrelow, \qreup\}$. Then
        $$
        P(\vect{T}^{[m]},\vect{f}^{[m]},Z) \leq m P(\vect{T},\vect{f},Z);
        $$
        and if $\|\vect{f}\|<+\infty$, then
        $$
        \prebow(\vect{T}^{[m]},\vect{f}^{[m]},Z) \leq m \prebow(\vect{T},\vect{f},Z).
        $$
      \end{lem}
      \begin{proof}
        We first show the inequalities for  $\prelow$ and $\preup$, and the proofs for $\qrelow$ and $\qreup$ are similar.

 By \eqref{eq:dnm}, every $(n,\varepsilon)$-separated set $E$ for $Z$ with respect to $\vect{T}^{[m]}$ is $(nm,\varepsilon)$ separated for $Z$ with respect to $\vect{T}$.
        Thus, by \eqref{eq:defPn} and \eqref{eq:Snm},
        $$
        P_{n}(\vect{T}^{[m]},\vect{f}^{[m]},Z) \leq P_{nm}(\vect{T},\vect{f},Z),
        $$
        and it follows by \eqref{eq:Pe} that
          $$
        P(\vect{T}^{[m]},\vect{f}^{[m]},Z) \leq m P(\vect{T},\vect{f},Z),
        $$
for  $P \in \{  \prelow, \preup\}$.

       By Theorem~\ref{thm:prepaceqpremod},  the inequality for $\prepac$ is a consequence of the inequality  for $\preup$.

 It remains to show the inequality for $\prebow$ for $\|\vect{f}\|<+\infty$.    Similarly by \eqref{eq:dnm}, for all integers $m \geq 1$ and $n \geq 1$,
        $$
        B_{nm}^{\vect{T}}(x,\varepsilon) \subset B_{n}^{\vect{T}^{[m]}}(x,\varepsilon)
        $$
        for all $x \in X_{0}$ and $\varepsilon>0$, and hence for every countable $(N,\varepsilon)$-cover $\{B_{n_{i}}^{\vect{T}}(x_{i},\varepsilon)\}_{i=1}^{\infty}$ of $Z$,
        the family $\{B_{\lfloor\frac{n_{i}}{m}\rfloor}^{\vect{T}^{[m]}}(x_{i},\varepsilon)\}_{i=1}^{\infty}$ of concentric Bowen balls  is a $(\lfloor\frac{N}{m}\rfloor,\varepsilon)$-cover of $Z$ with respect to $\vect{T}^{[m]}$.

        For every $s>m\prebow(\vect{T},\vect{f},Z)$, by Proposition~\ref{prop:msrbowepsilon}, we have
        $$
        \prebow(\vect{T},\vect{f},Z,\varepsilon) \leq \prebow(\vect{T},\vect{f},Z) <\frac{s}{m}
        $$
        for all $\varepsilon>0$. Fix $\varepsilon>0$.
        By \eqref{def_PBep} and the definition of $\msrbow_{\varepsilon}^{\frac{s}{m}}(\vect{T}^{[m]},\vect{f}^{[m]},Z)$,  we have  that $\msrbow_{N,\varepsilon}^{\frac{s}{m}}(\vect{T},\vect{f},Z)=0$ for all large  $N$.
        This implies that for all $\eta>0$, there exists a $(N,\varepsilon)$-cover $\{B_{n_{i}}^{\vect{T}}(x_{i},\varepsilon)\}_{i=1}^{\infty}$ of $Z$ with respect to $\vect{T}$
        such that
        $$
          \sum_{i=1}^{\infty}\exp{\Big(-n_{i}\left(\frac{s}{m}\right)+S_{n_{i}}^{\vect{T}}\vect{f}(x_{i})\Big)} \leq \eta.
        $$
        Note that the family $\{B_{\lfloor\frac{n_{i}}{m}\rfloor}^{\vect{T}^{[m]}}(x_{i},\varepsilon)\}_{i=1}^{\infty}$ is a $(\lfloor\frac{N}{m}\rfloor,\varepsilon)$-cover of $Z$ with respect to $\vect{T}^{[m]}$.
        Combining the above inequality with \eqref{eq:defmsrbow}, we obtain that
        \begin{align*}
          \msrbow_{\left\lfloor\frac{N}{m}\right\rfloor,\varepsilon}^{s}(\vect{T}^{[m]},\vect{f}^{[m]},Z)
          & \leq \sum_{i=1}^{\infty}\exp{\Big(-\left\lfloor\frac{n_{i}}{m}\right\rfloor s + S_{\lfloor\frac{n_{i}}{m}\rfloor}^{\vect{T}^{[m]}}\vect{f}^{[m]}(x_{i})\Big)} \\
          &= \sum_{i=1}^{\infty}\exp{\Big(-\left\lfloor\frac{n_{i}}{m}\right\rfloor s + S_{\lfloor\frac{n_{i}}{m}\rfloor m}^{\vect{T}}\vect{f}(x_{i})\Big)} \\
          &\leq \exp{\Bigl(C_{s}+\sum_{j=\left\lfloor\frac{n_{i}}{m}\right\rfloor m}^{n_{i}}\|f_{j}\|_{\infty}\Bigr)}\sum_{i=1}^{\infty}\exp{\Big(\left\lfloor -\frac{n_{i}s}{m} \right\rfloor + S_{n_{i}}^{\vect{T}}\vect{f}(x_{i})\Big)} \\
          &\leq \exp{(C_{s}+m\|\vect{f}\|)}\sum_{i=1}^{\infty}\exp{\Big(-n_{i}\left(\frac{s}{m}\right)+S_{n_{i}}^{\vect{T}}\vect{f}(x_{i})\Big)}\\
          &\leq \exp{(C_{s}+m\|\vect{f}\|)}\eta,
        \end{align*}
        where $C_{s}=\max\{0,2s\}$. Therefore $\msrbow_{\lfloor\frac{N}{m}\rfloor,\varepsilon}^{s}(\vect{T}^{[m]},\vect{f}^{[m]},Z)=0$ for all $N>0$ and $\varepsilon>0$,
        and hence $\prebow(\vect{T}^{[m]},\vect{f}^{[m]},Z)<s$.
      \end{proof}

      Full power rules are obtained with the assumption of the equicontinuity of $\vect{T}$.
      \begin{thm}\label{thm:Pmpoweq}
 Given $\vect{f} \in \vect{C}(\vect{X},\R)$ and $Z\subset X_0$.    Let $P \in \{\prepac, \prelow, \preup, \qrelow, \qreup\}$.
        If $\vect{T}$ is equicontinuous, then
        $$
        P(\vect{T}^{[m]},\vect{f}^{[m]},Z) = m P(\vect{T},\vect{f},Z);
        $$
        if additionally $\|\vect{f}\|<+\infty$, then
        $$
        \prebow(\vect{T}^{[m]},\vect{f}^{[m]},Z) = m \prebow(\vect{T},\vect{f},Z).
        $$
      \end{thm}
      \begin{proof}
By Lemma~\ref{lem:Pmpowneq}. for every $P \in \{\prepac, \prelow, \preup, \qrelow, \qreup\}$, it is sufficient to show that
\begin{equation}\label{ineq_PMMP}
  P(\vect{T}^{[m]},\vect{f}^{[m]},Z) \geq  m P(\vect{T},\vect{f},Z);
\end{equation}

Since  $\vect{T}$ is equicontinuous, for each $\varepsilon>0$, we  choose $\delta>0$ so that for all $k \in \N$ and $x,y \in X_{km}$,
        $$
        d_{km,m}^{\vect{T}}(x,y)<\varepsilon
        \quad
        \text{whenever}\ d_{X_{km}}(x,y)<\delta.
        $$
        It follows that for all $k \in \N$ and $x,y \in X_{km}$,
        \begin{equation}\label{eq:dnmc}
          d_{km,nm}^{\vect{T}}(x,y)<\varepsilon
          \quad
          \text{whenever}\ d_{k,n}^{\vect{T}^{[m]}}(x,y)<\delta.
        \end{equation}

        First, we prove  Inequality \eqref{ineq_PMMP} for $\prelow$ and $\preup$, and the argument for $\qrelow$ and $\qreup$ is similar.
       By \eqref{eq:dnmc}, every $(nm,\varepsilon)$-separated set for $Z$ with respect to $\vect{T}$ is $(n,\delta)$ separated for $Z$ with respect to $\vect{T}^{[m]}$.
        By \eqref{eq:defPn} and \eqref{eq:Snm}, it follows that
        $$
        P_{n}(\vect{T}^{[m]},\vect{f}^{[m]},Z,\delta) \geq P_{nm}(\vect{T},\vect{f},Z,\varepsilon),
        $$
        and hence
        $$
        \prelow(\vect{T}^{[m]},\vect{f}^{[m]},Z) \geq m \prelow(\vect{T},\vect{f},Z),
        \quad
        \preup(\vect{T}^{[m]},\vect{f}^{[m]},Z) \geq m \preup(\vect{T},\vect{f},Z).
        $$

       Inequality \eqref{ineq_PMMP} for $\prepac$ is a  consequence of the  inequality for $\preup$  and  Theorem~\ref{thm:prepaceqpremod}.

        Next, we show the opposite inequality for $\prebow$.    By \eqref{eq:dnmc}, for all integers $m \geq 1$ and $n \geq 1$, we have
        $$
        B_{n}^{\vect{T}^{[m]}}(x,\delta) \subset B_{nm}^{\vect{T}}(x,\varepsilon)
        $$
        for all $x \in X_{0}$. It follows that for every countable $(N,\delta)$-cover $\{B_{n_{i}}^{\vect{T}^{[m]}}(x_{i},\delta)\}_{i=1}^{\infty}$ of $Z$ with respect to $\vect{T}^{[m]}$,
        the family $\{B_{n_{i}m}^{\vect{T}}(x_{i},\varepsilon)\}_{i=1}^{\infty}$ of concentric Bowen balls has to be a $(Nm,\varepsilon)$-cover of $Z$ with respect to $\vect{T}$.

        For every $s>\frac{1}{m}\prebow(\vect{T}^{[m]},\vect{f}^{[m]},Z)$, by Proposition~\ref{prop:msrbowepsilon}, we have
        $$
        \prebow(\vect{T}^{[m]},\vect{f}^{[m]},Z,\delta) \leq \prebow(\vect{T}^{[m]},\vect{f}^{[m]},Z) < ms
        $$
        for all $\delta>0$, and by \eqref{def_PBep} and \eqref{eq:defmsrbow}, for every $\eta>0$, there exists a $(N,\delta)$-cover $\{B_{n_{i}}^{\vect{T}^{[m]}}(x_{i},\delta)\}_{i=1}^{\infty}$ of $Z$ with respect to $\vect{T}^{[m]}$ such that
        \begin{equation}
          \sum_{i=1}^{\infty}\exp{\left(-n_{i}(ms) + S_{n_{i}}^{\vect{T}^{[m]}}\vect{f}^{[m]}(x_{i})\right)} \leq \eta.
        \end{equation}
        It follows that
        \begin{align*}
          \sum_{i=1}^{\infty}\exp{\left(-(n_{i}m)s + S_{n_{i}m}^{\vect{T}}\vect{f}(x_{i})\right)} &= \sum_{i=1}^{\infty}\exp{\left(-n_{i}(ms) + S_{n_{i}}^{\vect{T}^{[m]}}\vect{f}^{[m]}(x_{i})\right)} \\
                                                                                                &\leq \eta,
        \end{align*}
        where $\{B_{n_{i}m}^{\vect{T}}(x_{i},\varepsilon)\}_{i=1}^{\infty}$ is a $(Nm,\varepsilon)$-cover of $Z$ with respect to $\vect{T}$, so
        $\msrbow_{Nm,\varepsilon}^{s}(\vect{T},\vect{f},Z) = 0$ for each $\varepsilon>0$.
        This implies by \eqref{def_PBep} and Definition~\ref{def:prebpp} that $ \prebow(\vect{T},\vect{f},Z) \leq s,$
        and  by the arbitrariness of $s$, we conclude that  $\prebow(\vect{T}^{[m]},\vect{f}^{[m]},Z)\geq        m\prebow(\vect{T},\vect{f},Z) .$
    \end{proof}

    \section{Product of Nonautonomous Dynamical Systems}\label{sect:prod}
      In this section, we study the pressures of the direct product of two NDSs $(\vect{X},\vect{T})$ and $(\vect{Y},\vect{R})$.
      We write $\vect{X} \times \vect{Y}=\{X_{k} \times Y_{k}\}_{k=0}^{\infty}$ and endow each $X_{k} \times Y_{k}$ with the metric $d_{X_{k}} \times d_{Y_{k}}$ given by
      $$
      (d_{X_{k}} \times d_{Y_{k}})((x_{1},y_{1}), (x_{2},y_{2})) = \max\{d_{X_{k}}(x_{1},x_{2}), d_{Y_{k}}(y_{1},y_{2})\}
      $$
      for every $(x_{1},y_{1})$ and $(x_{2},y_{2}) \in X_{k} \times Y_{k}$.
      Clearly, every $X_{k} \times Y_{k}$ is compact, and for all $(x,y) \in X_{k} \times Y_{k}$,
      $$
      B_{X_{k} \times Y_{k}}((x,y),\varepsilon) = B_{X_{k}}(x,\varepsilon) \times B_{Y_{k}}(y,\varepsilon).
      $$

      We write $\vect{T} \times \vect{R}=\{T_{k} \times R_{k}\}_{k=0}^{\infty}$ where each $T_{k} \times R_{k}:X_{k} \times Y_{k} \to X_{k+1} \times Y_{k+1}$ is given by
      $$
      (T_{k} \times R_{k})(x,y) = (T_{k}x,R_{k}y) \quad \text{for every}\ (x,y) \in X_{k} \times Y_{k}.
      $$
      It is straightforward to verify that every $T_{k} \times R_{k}$ is continuous.
      Moreover, for each $k,n \in \N$,
      $$
      (\vect{T} \times \vect{R})_{k}^{n}=\vect{T}_{k}^{n} \times \vect{R}_{k}^{n},
      $$
      and for all $(x,y) \in X_{k} \times Y_{k}$,
      \begin{equation}\label{eq:prodBowball}
      B_{k,n}^{\vect{T}\times\vect{R}}((x,y),\varepsilon) = B_{k,n}^{\vect{T}}(x,\varepsilon) \times B_{k,n}^{\vect{R}}(y,\varepsilon).
      \end{equation}
      If $F_{1} \subset X_{0}$ and $F_{2} \subset Y_{0}$ are $(n,\varepsilon)$-spanning sets for $Z \subset X_{0}$ and $W \subset Y_{0}$, respectively,
      then $F_{1} \times F_{2}$ is $(n,\varepsilon)$ spanning for $Z \times W$;
      if $E_{1} \subset X_{0}$ and $E_{2} \subset Y_{0}$ are $(n,\varepsilon)$-separated sets for $Z \subset X_{0}$ and $W \subset Y_{0}$, respectively,
      then $E_{1} \times E_{2}$ is $(n,\varepsilon)$ separated for $Z \times W$.

      Given $\vect{f} \in \vect{C}(\vect{X},\R)$ and $\vect{g} \in \vect{C}(\vect{Y},\R)$,
      we write $\vect{f} \dotplus \vect{g} = \{f_{k} \dotplus g_{k}\}_{k=0}^{\infty}$ where each $f_{k} \dotplus g_{k}$ is given by
      $$
      (f_{k} \dotplus g_{k})(x,y) = f_{k}(x) + g_{k}(y) \quad \text{for all}\ (x,y) \in X_{k} \times Y_{k}.
      $$
      It is obvious that $\vect{f} \dotplus \vect{g} \in \vect{C}(\vect{X} \times \vect{Y},\R)$.
      Furthermore, we have that for all $k \in \N$, $n \geq 1$ and $(x,y) \in X_{k} \times Y_{k}$,
      \begin{equation}\label{eq:prodSkn}
        S_{k,n}^{\vect{T}\times\vect{R}}(\vect{f}\dotplus\vect{g})(x,y) = S_{k,n}^{\vect{T}}\vect{f}(x) + S_{k,n}^{\vect{R}}\vect{g}(y).
      \end{equation}

    The following two conclusions are parallel to the dimension inequalities  \eqref{ineq_fdHP} for product fractals. From now on, we always assume the subtraction of two pressures is well defined, i.e., $P(\vect{T},\vect{f},Z) \pm P'(\vect{R},\vect{g},W)$ is not the form of $\infty-\infty$.
      \begin{thm}\label{thm:prodBBP}
        Let $Z \subset X_{0}$ and $W \subset Y_{0}$. Then for all equicontinuous $\vect{f} \in \vect{C}(\vect{X},\R)$ and equicontinuous $\vect{g} \in \vect{C}(\vect{Y},\R)$,
        $$
        \prebow(\vect{T},\vect{f},Z) + \prebow(\vect{R},\vect{g},W)\leq  \prebow(\vect{T}\times\vect{R},\vect{f}\dotplus\vect{g},Z \times W) \leq     \prebow(\vect{T},\vect{f},Z) + \prepac(\vect{R},\vect{g},W).
        $$
      \end{thm}

      \begin{thm}\label{thm:prodPPPB}
        Let $Z \subset X_{0}$ and $W \subset Y_{0}$. Then for all equicontinuous $\vect{f} \in \vect{C}(\vect{X},\R)$ and equicontinuous $\vect{g} \in \vect{C}(\vect{Y},\R)$,
        $$
        \prebow(\vect{T},\vect{f},Z) + \prepac(\vect{R},\vect{g},W)\leq          \prepac(\vect{T}\times\vect{R},\vect{f}\dotplus\vect{g},Z \times W) \leq  \prepac(\vect{T},\vect{f},Z) + \prepac(\vect{R},\vect{g},W).
        $$
      \end{thm}

      Similar results hold for the lower and upper pressures of the products.

      \begin{thm}\label{thm:prodLLU}
        Let $Z \subset X_{0}$ and $W \subset Y_{0}$. Then for all equicontinuous $\vect{f} \in \vect{C}(\vect{X},\R)$ and equicontinuous $\vect{g} \in \vect{C}(\vect{Y},\R)$,
        $$
        \preL(\vect{T},\vect{f},Z) + \preL(\vect{R},\vect{g},W)\leq  \preL(\vect{T}\times\vect{R},\vect{f}\dotplus\vect{g},Z \times W) \leq     \preL(\vect{T},\vect{f},Z) + \preU(\vect{R},\vect{g},W).
        $$
      \end{thm}

      \begin{thm}\label{thm:prodLUU}
        Let $Z \subset X_{0}$ and $W \subset Y_{0}$. Then for all equicontinuous $\vect{f} \in \vect{C}(\vect{X},\R)$ and equicontinuous $\vect{g} \in \vect{C}(\vect{Y},\R)$,
        $$
        \preL(\vect{T},\vect{f},Z) + \preU(\vect{R},\vect{g},W) \leq   \preU(\vect{T}\times\vect{R},\vect{f}\dotplus\vect{g},Z \times W) \leq  \preU(\vect{T},\vect{f},Z) + \preU(\vect{R},\vect{g},W).
        $$
      \end{thm}

        For certain `dynamically regular' sets (see Tricot \cite[\S4, Def.6]{Tricot1982} for a type of geometric regularity) where the pressures coincide,
        we obtain equalities in the product formulae.

        \begin{cor}\label{cor:prodeq}   
        Let $Z \subset X_{0}$ and $W \subset Y_{0}$.
          Suppose that $\vect{f} \in \vect{C}(\vect{X},\R)$ and $\vect{g} \in \vect{C}(\vect{Y},\R)$ are both equicontinuous.
          \begin{enumerate}[(1)]
            \item If $\prebow(\vect{T},\vect{f},Z)=\prepac(\vect{T},\vect{f},Z)$ or $\prebow(\vect{R},\vect{g},W)=\prepac(\vect{R},\vect{g},W)$,
                  then
\[ \begin{split}
                 & \prebow(\vect{T}\times\vect{R},\vect{f}\dotplus\vect{g},Z \times W) = \prebow(\vect{T},\vect{f},Z) + \prebow(\vect{R},\vect{g},W); \\
                  &\prepac(\vect{T}\times\vect{R},\vect{f}\dotplus\vect{g},Z \times W) = \prepac(\vect{T},\vect{f},Z) + \prepac(\vect{R},\vect{g},W).
 \end{split} \]
          \item If $\preL(\vect{T},\vect{f},Z)=\preU(\vect{T},\vect{f},Z)$ or $\preL(\vect{R},\vect{g},W)=\preU(\vect{R},\vect{g},W)$,
                  then
\[ \begin{split}
                 & \preL(\vect{T}\times\vect{R},\vect{f}\dotplus\vect{g},Z \times W) = \preL(\vect{T},\vect{f},Z) + \preL(\vect{R},\vect{g},W);\\
                  &\preU(\vect{T}\times\vect{R},\vect{f}\dotplus\vect{g},Z \times W) = \preU(\vect{T},\vect{f},Z) + \preU(\vect{R},\vect{g},W).
 \end{split} \]
          \end{enumerate}
        \end{cor}
 
Kawan and Latushkin  studied  equalities for the product formula of a type of measure-theoretic entropy  in \cite[Prop.3.1(3) \& Prop.3.3]{Kawan&Latushkin2015}.
        \begin{rmk}
The equalities of pressures in Corollary \ref{cor:prodeq} are satisfied in  the autonomous systems  or compact invariant subsets of autonomous systems $(X,T)$. For example, if  one of the systems is autonomous, say $(X,T)$, and  $Z \subset X$ is $T$-invariant,   letting $\vect{f}=\{f\}_{k=0}^{\infty}$, then $\prebow(T,f,Z)=\preup(T,f,Z)$, and all the equalities in Corollary~\ref{cor:prodeq} hold.
        \end{rmk}

The proof of Theorem \ref{thm:prodBBP} is given in Subsection \ref{ssec_PBR}; the proofs of Theorem~\ref{thm:prodLLU} and Theorem~\ref{thm:prodLUU} are given in Subsection~\ref{ssect:prodLU}; and the proofs of Theorem \ref{thm:prodPPPB} and Corollary \ref{cor:prodeq} are given in Subsection \ref{ssec_PPP}.

\subsection{Product for lower and upper pressures}\label{ssect:prodLU}
In this subsection, we study the product   formulae for the lower and upper pressures,

   \begin{thm}\label{thm:prodPL}
        Given two NDSs $(\vect{X},\vect{T})$ and $(\vect{Y},\vect{R})$, let $Z \subset X_{0}$ and $W \subset Y_{0}$.
        Then for all $\vect{f} \in \vect{C}(\vect{X},\R)$ and $\vect{g} \in \vect{C}(\vect{Y},\R)$,
        $$
        \prelow(\vect{T} \times \vect{R},\vect{f} \dotplus \vect{g},Z \times W) \geq \prelow(\vect{T},\vect{f},Z) + \prelow(\vect{R},\vect{g},W).
        $$
      \end{thm}
      \begin{proof}
        Recall that the Cartesian product $E_{1} \times E_{2}$ of two $(n,\varepsilon)$-separated sets $E_{1}$ for $Z$ and $E_{2}$ for $W$ is $(n,\varepsilon)$ separated for $Z \times W$.
        By \eqref{eq:prodSkn},
        \begin{align*}
          \sum_{(x,y) \in E_{1} \times E_{2}}\e^{S_{n}^{\vect{T}\times\vect{R}}(\vect{f}\dotplus\vect{g})(x,y)}
          &=\sum_{(x,y) \in E_{1} \times E_{2}}\e^{S_{n}^{\vect{T}}\vect{f}(x)}\e^{S_{n}^{\vect{R}}\vect{g}(y)} \\
          &=\Big(\sum_{x \in E_{1}}\e^{S_{n}^{\vect{T}}\vect{f}(x)}\Big)\Big(\sum_{y \in E_{2}}\e^{S_{n}^{\vect{R}}\vect{g}(y)}\Big).
        \end{align*}
        It follows by \eqref{eq:defPn} that
        $$
        P_{n}(\vect{T}\times\vect{R},\vect{f}\dotplus\vect{g},Z \times W) \geq P_{n}(\vect{T},\vect{f},Z) \cdot P_{n}(\vect{R},\vect{g},W)
        $$
        and by \eqref{eq:Pe} and \eqref{eq:defpre} that
        $$
        \prelow(\vect{T}\times\vect{R},\vect{f}\dotplus\vect{g},Z \times W) \geq \prelow(\vect{T},\vect{f},Z) + \prelow(\vect{R},\vect{g},W).
        $$
      \end{proof}

      An opposite inequality holds for the upper topological pressure $\qreup$.
      \begin{thm}\label{thm:prodQU}
        Given two NDSs $(\vect{X},\vect{T})$ and $(\vect{Y},\vect{R})$, let $Z \subset X_{0}$ and $W \subset Y_{0}$.
        Then for all $\vect{f} \in \vect{C}(\vect{X},\R)$ and $\vect{g} \in \vect{C}(\vect{Y},\R)$,
        $$
        \qreup(\vect{T} \times \vect{R},\vect{f} \dotplus \vect{g},Z \times W) \leq \qreup(\vect{T},\vect{f},Z) + \qreup(\vect{R},\vect{g},W).
        $$
      \end{thm}
\begin{proof}
The proof is similar to Theorem \ref{thm:prodPL}, and we omit it.
\end{proof}

      We provide the following result which relates the lower topological pressure of the product system
      to the lower and upper topological pressures of the original systems.
      \begin{thm}\label{thm:prodQLPLPU}
         Given two NDSs $(\vect{X},\vect{T})$ and $(\vect{Y},\vect{R})$, let $Z \subset X_{0}$ and $W \subset Y_{0}$.
        Then for all $\vect{f} \in \vect{C}(\vect{X},\R)$ and $\vect{g} \in \vect{C}(\vect{Y},\R)$.
       $$
        \qrelow(\vect{T}\times\vect{R},\vect{f}\dotplus\vect{g},Z \times W) \leq
          \prelow(\vect{T},\vect{f},Z) + \preup(\vect{R},\vect{g},W) 
        $$
      \end{thm}
      \begin{proof}
    Let $s>\prelow(\vect{T},\vect{f},Z)$ and $t>\preup(\vect{R},\vect{g},W)$.
        By \eqref{eq:defpre}, for all $\varepsilon>0$, we have that $\prelow(\vect{T},\vect{f},Z,\varepsilon)<s$ and $\preup(\vect{R},\vect{g},W,\varepsilon)<t$. Hence  by \eqref{eq:Pe},  there exists an increasing integral sequence $\{n_l\}_{l=1}^\infty$  such that for all $l \geq 1$, we have
        $$
        P_{n_l}(\vect{T},\vect{f},Z) \leq \e^{n_ls} \qquad \textit{and }\qquad    P_{n_{l}}(\vect{R},\vect{g},W) \leq \e^{n_{l}t},
        $$
        and it follows from \eqref{eq:defPn} that for all $(n_{l},\varepsilon)$-separated sets $E_{1}$ for $Z$ and all $(n_{l},\varepsilon)$-separated sets $E_{2}$ for $W$,
        \begin{equation}\label{eq:sumleqenls}
          \sum_{x \in E_{1}}\e^{S_{n_{l}}^{\vect{T}}\vect{f}(x)} \leq \e^{n_{l}s} \qquad \textit{and }\qquad
          \sum_{y \in E_{2}}\e^{S_{n_{l}}^{\vect{R}}\vect{g}(y)} \leq \e^{n_{l}t}.
        \end{equation}

        Fix $n_{l}$. Suppose that $E_{1}$ and $E_{2}$ are maximal $(n_{l},\varepsilon)$-separated sets for $Z$ and $W$, respectively.
        Since maximal $(n_{l},\varepsilon)$-separated sets are also $(n_{l},\varepsilon)$ spanning sets,
        the Cartesian product set $E_{1} \times E_{2}$ $(n_{l},\varepsilon)$ spans $Z \times W$.
        By \eqref{eq:defQn}, \eqref{eq:prodSkn} and \eqref{eq:sumleqenls}, we obtain that
        \begin{align*}
          Q_{n_{l}}(\vect{T} \times \vect{R},\vect{f}\dotplus\vect{g},Z \times W,\varepsilon) &\leq \sum_{(x,y) \in E_{1} \times E_{2}}\e^{S_{n_{l}}^{\vect{T}\times\vect{R}}(\vect{f}\dotplus\vect{g})(x,y)} \\
                                                                                  &= \Bigl(\sum_{x \in E_{1}}\e^{S_{n_{l}}^{\vect{T}}\vect{f}(x)}\Bigr) \Bigl(\sum_{y \in E_{2}}\e^{S_{n_{l}}^{\vect{R}}\vect{g}(y)}\Bigr) \\
                                                                                  &\leq \e^{n_{l}(s+t)}.
        \end{align*}
This implies that  for all $\varepsilon>0$ and each $l \geq 1$, there exists $n_{l} \geq l$ such that
$$
 \varliminf_{n \to \infty}\frac{1}{n_l} \log Q_{n_{l}}(\vect{T} \times \vect{R},\vect{f}\dotplus\vect{g},Z \times W,\varepsilon) \leq (s+t).
$$
        It follows by \eqref{eq:Qe} and  Proposition \ref{prop:QPe} that
        $$
        \qrelow(\vect{T}\times\vect{R},\vect{f}\dotplus\vect{g},Z \times W) \leq s+t,
        $$
        and the conclusion follows by  the arbitrariness of $s$ and $t$.
      \end{proof}

       The following product formulae relate the upper   pressure of the product system to the lower  (Bowen) pressure and the upper   pressure of the original systems.
        \begin{thm}\label{thm:prodPUQLQU}
          Given two NDSs $(\vect{X},\vect{T})$ and $(\vect{Y},\vect{R})$, let $Z \subset X_{0}$ and $W \subset Y_{0}$.
          Then for all $\vect{f} \in \vect{C}(\vect{X},\R)$ and $\vect{g} \in \vect{C}(\vect{Y},\R)$,
         $$
          \preup(\vect{T}\times\vect{R},\vect{f}\dotplus\vect{g},Z \times W) \geq
            \qrelow(\vect{T},\vect{f},Z) + \qreup(\vect{R},\vect{g},W) \geq \prebow(\vect{T},\vect{f},Z) + \qreup(\vect{R},\vect{g},W) .
          $$
        \end{thm}
\begin{proof}
The argument for the first inequality is similar to Theorem \ref{thm:prodQLPLPU}, and the second inequality is direct from Proposition \ref{prop:preneq}.
\end{proof}

\begin{proof}[Proofs of Theorem~\ref{thm:prodLLU} and Theorem~\ref{thm:prodLUU}]
  Since $\vect{f}$ and $\vect{g}$ are both equicontinuous, so is $\vect{f} \dotplus \vect{g}$. Hence, by Proposition \ref{prop:PeqQ}, the pressures $\preL$ and $\preU$ exist.
  The inequalities in Theorem~\ref{thm:prodLLU} are from Theorem~\ref{thm:prodPL} and Theorem~\ref{thm:prodQLPLPU}, and the ones in Theorem~\ref{thm:prodLUU} are consequences of Theorem~\ref{thm:prodPUQLQU} and Theorem~\ref{thm:prodQU}.
\end{proof}

\subsection{Product of Bowen pressures}\label{ssec_PBR}
In this subsection, we study the product   formula for the Bowen pressures.

By Proposition \ref{prop:prebw}, the Bowen pressures may be given by the weighted Hausdorff measure \eqref{def_WBPP}, and  it   provides a simple method to prove the lower bound.
      \begin{lem}\label{lemPBWWW}
Given $Z \subset X_{0}$, $W \subset Y_{0}$,  $\vect{f} \in \vect{C}(\vect{X},\R)$ and $\vect{g} \in \vect{C}(\vect{Y},\R)$.    For all $s,t \in \R$, $N \geq 1$, and $\varepsilon>0$,
        \begin{equation}\label{eq:Wprodneq}
          \mathscr{W}_{N,\varepsilon}^{s+t}(\vect{T}\times\vect{R},\vect{f}\dotplus\vect{g},Z \times W) \geq \mathscr{W}_{N,\varepsilon}^{s}(\vect{T},\vect{f},Z) \mathscr{W}_{N,\varepsilon}^{t}(\vect{R},\vect{g},W).
        \end{equation}
      \end{lem}
      \begin{proof}
        We  assume $\mathscr{W}_{N,\varepsilon}^{s+t}(\vect{T}\times\vect{R},\vect{f}\dotplus\vect{g},Z \times W)$ finite and $\mathscr{W}_{N,\varepsilon}^{t}(\vect{R},\vect{g},W)$ positive.

 Let $c$ be any number such that  $0<c<\mathscr{W}_{N,\varepsilon}^{t}(\vect{R},\vect{g},W)$
        and $\{(B_{n_{i}}^{\vect{T}\times\vect{R}}((x_{i},y_{i}),\varepsilon),c_{i})\}_{i=1}^{\infty}$ an arbitrary weighted $(N,\varepsilon)$-cover of $Z \times W$.         For simplicity, we write $B_{i}=B_{n_{i}}^{\vect{T}}(x_{i},\varepsilon)$ and $C_{i}=B_{n_{i}}^{\vect{R}}(y_{i},\varepsilon)$,
        and by \eqref{eq:prodBowball}, it is clear  that
        $$
        B_{n_{i}}^{\vect{T}\times\vect{R}}((x_{i},y_{i}),\varepsilon) = B_{i} \times C_{i}.
        $$

         For each $x \in Z$, let $\mathcal{I}(x)=\{i: x \in B_{i}\}$. Since $\chi_{B_{i} \times C_{i}}(x,y) = \chi_{B_{i}}(x)\chi_{C_{i}}(y)$, we have that
        $$
        \sum_{i \in \mathcal{I}(x)}c_{i}\chi_{C_{i}} = \sum_{i=1}^{\infty}c_{i}\chi_{B_{i}}(x)\chi_{C_{i}} = \sum_{i=1}^{\infty}c_{i}\chi_{B_{i} \times C_{i}}(x,\cdot) \geq \chi_{W}
        $$
Hence  $\{(C_{i},c_{i})\}_{i \in \mathcal{I}(x)}$ is a weighted $(N,\varepsilon)$-cover of $W$, and  by \eqref{def_WBPP}, we have
        $$
        \sum_{i \in \mathcal{I}(x)}c_{i}\exp{\left(-n_{i}t + S_{n_{i}}^{\vect{R}}\vect{g}(y_{i})\right)} \geq c.
        $$
Letting  $b_{i}=\frac{c_{i}}{c}\exp{(-n_{i}t + S_{n_{i}}^{\vect{R}}\vect{g}(y_{i}))}$,  it  is clear that  $\sum_{i \in \mathcal{I}(x)}b_{i} \geq 1$ for each $x \in Z$ and $ \sum_{i=1}^{\infty}b_{i}\chi_{B_{i}} \geq \chi_{Z},$ which implies that  $\{(B_{i},b_{i})\}_{i=1}^{\infty}$ is a weighted $(N,\varepsilon)$-cover of $Z$.
        By \eqref{eq:prodSkn}, we obtain that
        \begin{align*}
          &\sum_{i=1}^{\infty}c_{i}\exp{\left(-n_{i}(s+t) + S_{n_{i}}^{\vect{T}\times\vect{R}}(\vect{f}\dotplus\vect{g})(x_{i},y_{i})\right)} \\
            &=\sum_{i=1}^{\infty}c_{i}\exp{\left(-n_{i}s + S_{n_{i}}^{\vect{T}}\vect{f}(x_{i})\right)}\exp{\left(-n_{i}t + S_{n_{i}}^{\vect{R}}\vect{g}(y_{i})\right)} \\
            &=\sum_{i=1}^{\infty}cb_{i}\exp{\left(-n_{i}s + S_{n_{i}}^{\vect{T}}\vect{f}(x_{i})\right)} \\
            &\geq c\mathscr{W}_{N,\varepsilon}^{s}(\vect{T},\vect{f},Z),
        \end{align*}
        and  the conclusion  follows from the arbitrariness of $c$ and $\{(B_{i} \times C_{i},c_{i})\}_{i=1}^{\infty}$.
      \end{proof}

\begin{thm}\label{thm:prodPB}
        Given two NDSs $(\vect{X},\vect{T})$ and $(\vect{Y},\vect{R})$, let $Z \subset X_{0}$ and $W \subset Y_{0}$.
        Then for all equicontinuous $\vect{f} \in \vect{C}(\vect{X},\R)$ and equicontinuous $\vect{g} \in \vect{C}(\vect{Y},\R)$,
        $$
        \prebow(\vect{T} \times \vect{R},\vect{f} \dotplus \vect{g},Z \times W) \geq \prebow(\vect{T},\vect{f},Z) + \prebow(\vect{R},\vect{g},W).
        $$
      \end{thm}
      \begin{proof} 
        We assume that $\prebow(\vect{T},\vect{f},Z) + \prebow(\vect{R},\vect{g},W)>-\infty$ since the conclusion is trivial otherwise.
        For every $s<\prebow(\vect{T},\vect{f},Z)$ and $t<\prebow(\vect{R},\vect{g},W)$, by Proposition~\ref{prop:prebw} and \eqref{def_PBW},
        there exists $\varepsilon_{0}>0$ such that for all $0<\varepsilon\leq\varepsilon_{0}$, we have that
$$
\mathscr{W}_{\varepsilon}^{s}(\vect{T},\vect{f},Z)=+\infty
\qquad \textit{ and } \qquad  \mathscr{W}_{\varepsilon}^{t}(\vect{R},\vect{g},W)=+\infty.
$$
For every $0<\varepsilon\leq\varepsilon_{0}$, we choose  $M>0$ and  a large $N_{1} \geq 1$  such that  for all $N>N_1$,
        $$
        \min\{\mathscr{W}_{N,\varepsilon}^{s}(\vect{T},\vect{f},Z), \mathscr{W}_{N,\varepsilon}^{t}(\vect{R},\vect{g},W)\} \geq M.
        $$
By Lemma \ref{lemPBWWW}, it implies that
        $$
        \mathscr{W}_{N,\varepsilon}^{s+t}(\vect{T}\times\vect{R},\vect{f}\dotplus\vect{g},Z \times W) \geq \mathscr{W}_{N,\varepsilon}^{s}(\vect{T},\vect{f},Z)\mathscr{W}_{N,\varepsilon}^{t}(\vect{R},\vect{g},W) \geq M^{2},
        $$
and for every $0<\varepsilon\leq\varepsilon_{0}$, we have that
$$
\mathscr{W}_{\varepsilon}^{s+t}(\vect{T}\times\vect{R},\vect{f}\dotplus\vect{g},Z \times W)\geq M^2.
$$
By \eqref{def_PBW} and Proposition~\ref{prop:prebw}, it follows that
$$
\prebow(\vect{T}\times\vect{R},\vect{f}\dotplus\vect{g},Z \times W) \geq s+t,
$$
        and the conclusion follows by the arbitrariness of $s$ and $t$.
      \end{proof}

    To find the upper bound, we need the following property on product measures.
      \begin{lem}\label{lem_2RRP}
 Given $Z \subset X_{0}$, $W \subset Y_{0}$,  $\vect{f} \in \vect{C}(\vect{X},\R)$ and $\vect{g} \in \vect{C}(\vect{Y},\R)$.         For all $s,t \in \R$, and $\varepsilon>0$,
        \begin{equation}\label{eq:RPprodneqlim}
          \msrbow_{2\varepsilon}^{s+t}(\vect{T}\times\vect{R},\vect{f}\dotplus\vect{g},Z \times W) \leq \msrbow_{2\varepsilon}^{s}(\vect{T},\vect{f},Z) \msrpac_{\varepsilon}^{t}(\vect{R},\vect{g},W),
        \end{equation}
 provided that the product on the RHS is not of the form $0 \cdot \infty$ or $\infty \cdot 0$.
      \end{lem}
  \begin{proof}
        We assume that $\msrbow_{2\varepsilon}^{s}(\vect{T},\vect{f},Z)<\infty$ and $\msrpac_{\varepsilon}^{t}(\vect{R},\vect{g},W)<\infty$.
        For every $\eta>0$, by \eqref{def_pes}, there exists a decomposition $\{W_{j}\}_{j=1}^{\infty}$ of $W$
        such that $\bigcup_{j=1}^{\infty}W_{j} \supset W$ and
        $$
        \sum_{j=1}^{\infty}\msrpac_{\infty,\varepsilon}^{t}(\vect{R},\vect{g},W_{j}) \leq \msrpac_{\varepsilon}^{t}(\vect{R},\vect{g},W) + \eta.
        $$
        Since $Z \times W \subset \bigcup_{j=1}^{\infty}(Z \times W_{j})$,   by the subadditivity of $\msrbow_{2\varepsilon}^{s+t}(\vect{T}\times\vect{R},\vect{f}\dotplus\vect{g},\cdot)$, we have that
        $$
        \msrbow_{2\varepsilon}^{s+t}(\vect{T}\times\vect{R},\vect{f}\dotplus\vect{g},Z \times W) \leq \sum_{j=1}^{\infty}\msrbow_{2\varepsilon}^{s+t}(\vect{T}\times\vect{R},\vect{f}\dotplus\vect{g},Z \times W_{j}).
        $$
        Combining the two inequalities above, to show the inequality \eqref{eq:RPprodneqlim}, it suffices by the arbitrariness of $\eta$ to show that for all integers $j \geq 1$,
        \begin{equation}\label{eq:RPprodneq}
          \msrbow_{2\varepsilon}^{s+t}(\vect{T}\times\vect{R},\vect{f}\dotplus\vect{g},Z \times W_{j}) \leq \msrbow_{2\varepsilon}^{s}(\vect{T},\vect{f},Z)\msrpac_{\infty,\varepsilon}^{t}(\vect{R},\vect{g},W_{j}).
        \end{equation}

        Fix $j \geq 1$. For every $N \geq 1$, suppose that $\{B_{n_{i}}^{\vect{T}}(x_{i},2\varepsilon)\}_{i=1}^{\infty}$ is an arbitrarily chosen $(N,2\varepsilon)$-cover of $Z$.
        For simplicity, we write $B_{i}=B_{n_{i}}^{\vect{T}}(x_{i},2\varepsilon)$ for each $i$.
        Fix $i \geq 1$. Since $Y_{0}$ is compact, $W_{j}$ is bounded, and  there exists a maximal $(N,\varepsilon)$-packing $\{\overline{B}_{n_{i}}^{\vect{R}}(y_{l},\varepsilon)\}_{l=1}^{\infty}$ of $W_{j}$
        such that $\{B_{n_{i}}^{\vect{R}}(y_{l},2\varepsilon)\}_{l=1}^{\infty}$ covers $W_{j}$.
        Writing $C_{l}=B_{n_{i}}^{\vect{R}}(y_{l},2\varepsilon)$ for each $l$, it follows that $\{B_{i} \times C_{l}\}_{l=1}^{\infty}$ covers $B_{i} \times W_{j}$,
        and thus $\{B_{n_{i}}^{\vect{T}\times\vect{R}}((x_{i},y_{l}),2\varepsilon)\}_{l=1}^{\infty}$ is a countable $(N,2\varepsilon)$-cover of $B_{i} \times W_{j}$.
        Combining these with \eqref{eq:defmsrbow}, \eqref{eq:prodSkn} and \eqref{eq:defmsrPack}, we have that
        \begin{align*}
          \msrbow_{N,2\varepsilon}^{s+t}(\vect{T}\times\vect{R},\vect{f}\dotplus\vect{g},B_{i} \times W_{j}) &\leq \sum_{l=1}^{\infty}\exp{\left(-n_{i}(s+t)+S_{n_{i}}^{\vect{T}\times\vect{R}}(\vect{f}\dotplus\vect{g})(x_{i},y_{l})\right)} \\
                                                                                                            &\leq \sum_{l=1}^{\infty}\exp{\left(-n_{i}s+S_{n_{i}}^{\vect{T}}\vect{f}(x_{i})\right)}\exp{\left(-n_{i}t+S_{n_{i}}^{\vect{R}}\vect{g}(y_{l})\right)} \\
                                                                                                            &\leq \exp{\left(-n_{i}s+S_{n_{i}}^{\vect{T}}\vect{f}(x_{i})\right)}\msrpac_{N,\varepsilon}^{t}(\vect{R},\vect{g},W_{j}).
        \end{align*}
        By the subadditivity of $\msrbow_{N,2\varepsilon}^{s+t}(\vect{T}\times\vect{R},\vect{f}\dotplus\vect{g},\cdot)$, summing over $i$ gives
        \begin{align*}
          \msrbow_{N,2\varepsilon}^{s+t}(\vect{T}\times\vect{R},\vect{f}\dotplus\vect{g},Z \times W_{j}) &\leq \sum_{i=1}^{\infty}\msrbow_{N,2\varepsilon}^{s+t}(\vect{T}\times\vect{R},\vect{f}\dotplus\vect{g},B_{i} \times W_{j}) \\
                                                                                                      &\leq \msrpac_{N,\varepsilon}^{t}(\vect{R},\vect{g},W_{j})\sum_{i=1}^{\infty}\exp{(-n_{i}s+S_{n_{i}}^{\vect{T}}\vect{f}(x_{i}))},
        \end{align*}
        and by \eqref{eq:defmsrbow} and the arbitrariness of $\{B_{i}\}_{i=1}^{\infty}$, it implies  that
        $$
        \msrbow_{N,2\varepsilon}^{s+t}(\vect{T}\times\vect{R},\vect{f}\dotplus\vect{g},Z \times W_{j}) \leq \msrbow_{N,2\varepsilon}^{s}(\vect{T},\vect{f},Z)\msrpac_{N,\varepsilon}^{t}(\vect{R},\vect{g},W_{j}).
        $$
        Letting $N$ tend to infinity, we obtain \eqref{eq:RPprodneq}, and the conclusion holds.
      \end{proof}
   \begin{thm}\label{thm:prodPBPPPU}
        Given two NDSs $(\vect{X},\vect{T})$ and $(\vect{Y},\vect{R})$, let $Z \subset X_{0}$ and $W \subset Y_{0}$.
        Then for all equicontinuous $\vect{f} \in \vect{C}(\vect{X},\R)$ and equicontinuous $\vect{g} \in \vect{C}(\vect{Y},\R)$,
        $$
        \prebow(\vect{T}\times\vect{R},\vect{f}\dotplus\vect{g},Z \times W) \leq
            \prebow(\vect{T},\vect{f},Z) + \prepac(\vect{R},\vect{g},W) \leq \prebow(\vect{T},\vect{f},Z) + \preup(\vect{R},\vect{g},W). 
        $$
      \end{thm}
  \begin{proof}  
        By Proposition~\ref{prop:preneq},
        it clearly suffices to show that
        $$
        \prebow(\vect{T}\times\vect{R},\vect{f}\dotplus\vect{g},Z \times W) \leq \prebow(\vect{T},\vect{f},Z) + \prepac(\vect{R},\vect{g},W),
        $$
        and we assume that the sum on the RHS of the inequality is finite.

        For every $r > \prebow(\vect{T},\vect{f},Z)+\prepac(\vect{R},\vect{g},W)$, let
        $$
        \alpha = r - (\prebow(\vect{T},\vect{f},Z)+\prepac(\vect{R},\vect{g},W)) > 0
        $$
        and let $s=\prebow(\vect{T},\vect{f},Z)+\frac{\alpha}{2}$ and $t=\prepac(\vect{R},\vect{g},W)+\frac{\alpha}{2}$ so that $s+t=r$.

        Since $s>\prebow(\vect{T},\vect{f},Z)$, by Definition~\ref{def:prebpp}, \eqref{def_PBep} and Proposition~\ref{prop:msrbowepsilon}, we have that
        $$\msrbow_{\varepsilon}^{s}(\vect{T},\vect{f},Z)=0$$ for all $\varepsilon>0$.
        Meanwhile, since $t>\prepac(\vect{R},\vect{g},W)$, by Definition~\ref{def:prepac}, \eqref{def_PP} and Proposition~\ref{prop:msrpacepsilon}, it is clear that
        $$\msrpac_{\varepsilon}^{t}(\vect{R},\vect{g},W)=0$$ for all $\varepsilon>0$.
        By Lemma \ref{lem_2RRP},  it follows  that $$\msrbow_{2\varepsilon}^{r}(\vect{T}\times\vect{R},\vect{f}\dotplus\vect{g},Z \times W)=0$$ for all $\varepsilon>0$,
        and  by \eqref{def_PBep} and Definition~\ref{def:prebpp}, we obtain that $$\prebow(\vect{T}\times\vect{R},\vect{f}\dotplus\vect{g},Z \times W) \leq r. $$
        Hence  the conclusion follows by   the arbitrariness of $r$.
      \end{proof}

\begin{proof}[Proof of Theorem \ref{thm:prodBBP}]
It is  a direct consequence of Theorems \ref{thm:prodPB} and \ref{thm:prodPBPPPU}.
\end{proof}

\subsection{Product of packing pressures}\label{ssec_PPP}
      To obtain the lower bounds of  product formulae for the packing pressures, we need the following property on measures.
        \begin{lem} \label{eq:PRprodneq}
Given $Z \subset X_{0}$, $W \subset Y_{0}$,  $\vect{f} \in \vect{C}(\vect{X},\R)$ and $\vect{g} \in \vect{C}(\vect{Y},\R)$.            For all $s,t \in \R$, and $\varepsilon>0$,
          \begin{equation*}
            \msrpac_{\varepsilon}^{s+t}(\vect{T}\times\vect{R},\vect{f}\dotplus\vect{g},Z \times W) \geq \msrpac_{\varepsilon}^{s}(\vect{T},\vect{f},Z)\msrbow_{2\varepsilon}^{t}(\vect{R},\vect{g},W),
          \end{equation*}
provided that the product on the RHS is not of the form $0 \cdot \infty$ or $\infty \cdot 0$.
        \end{lem}
         \begin{proof}
            We assume that $\msrpac_{\varepsilon}^{s+t}(\vect{T}\times\vect{R},\vect{f}\dotplus\vect{g},Z \times W)<\infty$ and $\msrbow_{2\varepsilon}^{t}(\vect{R},\vect{g},W)>0$.
            By \eqref{def_pes}, for each $\eta>0$, there exists a decomposition $\{G_{i}\}_{i=1}^{\infty}$ of $Z \times W$ by subsets of $X_{0} \times Y_{0}$
            such that
            \begin{equation}\label{eq:pacGieta}
              \sum_{i=1}^{\infty}\msrpac_{\infty,\varepsilon}^{s+t}(\vect{T}\times\vect{R},\vect{f}\dotplus\vect{g},G_{i}) \leq \msrpac_{\varepsilon}^{s+t}(\vect{T}\times\vect{R},\vect{f}\dotplus\vect{g},Z \times W) + \eta,
            \end{equation}
and for every $l$ with $0<l<\msrbow_{2\varepsilon}^{t}(\vect{R},\vect{g},W)$,  there exists an integer $N_{l} \geq 1$ such that
            $$
\msrbow_{N,2\varepsilon}^{t}(\vect{R},\vect{g},W)>l
$$
for all $N \geq N_{l}.
$

          For every given $x \in Z$, let
            $$
            G_{i}^{W}(x)=\{y \in W: (x,y) \in G_{i}\}.
            $$
            It follows that $\bigcup_{i=1}^{\infty}G_{i}^{W}(x) \supset W$, and by the subadditivity of $\msrbow_{N,2\varepsilon}^{t}(\vect{R},\vect{g},\cdot)$, this implies that
            \begin{equation}\label{ineq_R2Egl}
            \sum_{i=1}^{\infty}\msrbow_{N,2\varepsilon}^{t}(\vect{R},\vect{g},G_{i}^{W}(x)) \geq \msrbow_{N,2\varepsilon}^{t}(\vect{R},\vect{g},W) > l
            \end{equation}
            for all $N \geq N_{l}$ and all $x \in Z$.
            Fix $N \geq N_{l}$. For every integer $m \geq 1$, write
            \begin{equation}\label{eq:Znm}
            Z_{N,m}=\left\{x \in Z: \sum_{i=1}^{m}\msrbow_{N,2\varepsilon}^{t}(\vect{R},\vect{g},G_{i}^{W}(x)) \geq l\right\}.
            \end{equation}
           It is clear that $Z_{N,m}$ increases to $Z$ as $m$ tends to $\infty$,
            i.e., $Z_{N,m} \subset Z_{N,m+1}$ for all $m \geq 1$ and $Z=\bigcup_{m=1}^{\infty}Z_{N,m}$.
            By \eqref{def_pes}, $\msrpac_{\varepsilon}^{s}(\vect{T},\vect{f},Z_{N,m}) \leq \msrpac_{\infty,\varepsilon}^{s}(\vect{T},\vect{f},Z_{N,m})$ for all $m \geq 1$.
            Since $\msrpac_{\varepsilon}^{s}(\vect{T},\vect{f},\cdot)$ is a measure, it follows that
            \begin{equation}\label{eq:paclimZNm}
              \msrpac_{\varepsilon}^{s}(\vect{T},\vect{f},Z) = \lim_{m \to \infty}\msrpac_{\varepsilon}^{s}(\vect{T},\vect{f},Z_{N,m}) \leq \lim_{m \to \infty}\msrpac_{\infty,\varepsilon}^{s}(\vect{T},\vect{f},Z_{N,m}).
            \end{equation}

            It remains to show that
            \begin{equation}\label{eq:sumPGigeqlPZm}
              \sum_{i=1}^{m}\msrpac_{\infty,\varepsilon}^{s+t}(\vect{T}\times\vect{R},\vect{f}\dotplus\vect{g},G_{i}) \geq l \msrpac_{\infty,\varepsilon}^{s}(\vect{T},\vect{f},Z_{N,m})
            \end{equation}
            for all positive real $l<\msrbow_{\varepsilon}^{t}(\vect{R},\vect{g},W)$ and integral $m \geq 1$. By \eqref{eq:pacGieta} and \eqref{eq:paclimZNm}, this implies that
            $$
            \msrpac_{\varepsilon}^{s+t}(\vect{T}\times\vect{R},\vect{f}\dotplus\vect{g},Z \times W) + \eta \geq l \msrpac_{\varepsilon}^{s}(\vect{T},\vect{f},Z),
            $$
            and  the conclusion \eqref{eq:PRprodneq} follows by the arbitrariness of $\eta$ and $l$.

            For every integral $m \geq 1$ and $n \geq N$, let $\{\overline{B}_{n_{j}}^{\vect{T}}(x_{j},\varepsilon)\}_{j=1}^{\infty}$ be an arbitrarily given countable $(n,\varepsilon)$-packing of $Z_{N,m}$.
            For each $i$ with $1 \leq i \leq m$, let
            $$
            \mathcal{J}(i)=\{j: G_{i}^{W}(x_{j}) \neq \emptyset\}.
            $$
            Obviously, for each $j \in \mathcal{J}(i)$, we have $x_{j} \in \proj_{X_{0}}(G_{i})$,
            and $\{\overline{B}_{n_{j}}^{\vect{T}}(x_{j},\varepsilon)\}_{j \in \mathcal{J}(i)}$ is a countable $(n,\varepsilon)$-packing of $\proj_{X_{0}}(G_{i})$ for every $i$.
            Fix $i$. For each $j \in \mathcal{J}(i)$, by the boundedness of $G_{i}^{W}(x_{j})$, there exists an at most countable set $H(i,j) \subset G_{i}^{W}(x_{j})$
            such that the family $\{\overline{B}_{n_{j}}^{\vect{R}}(y,\varepsilon)\}_{y \in H(i,j)}$ is a maximal $(n,\varepsilon)$-packing of $G_{i}^{W}(x_{j})$
            with $\{B_{n_{j}}^{\vect{R}}(y,2\varepsilon)\}_{y \in H(i,j)}$ covering $G_{i}^{W}(x_{j})$.
            By \eqref{eq:defmsrbow}, it follows that
            \begin{equation}\label{eq:bowGixjsumHij}
              \msrbow_{N,2\varepsilon}^{t}(\vect{R},\vect{g},G_{i}^{W}(x_{j})) \leq \msrbow_{n,2\varepsilon}^{t}(\vect{R},\vect{g},G_{i}^{W}(x_{j})) \leq \sum_{y \in H(i,j)}\exp{\bigl(-n_{j}t+S_{n_{j}}^{\vect{R}}\vect{g}(y)\bigr)}
            \end{equation}
            since $n \geq N$.
            Meanwhile, by \eqref{eq:prodBowball}, it is clear that $\{\overline{B}_{n_{j}}^{\vect{T}\times\vect{R}}((x_{j},y),\varepsilon)\}_{y \in H(i,j)}$ is an at most countable and maximal $(n,\varepsilon)$-packing of the `cross section' $\overline{B}_{n_{j}}^{\vect{T}}(x_{j},\varepsilon) \times G_{i}^{W}(x_{j})$ for each $j \in \mathcal{J}(i)$.
            Hence, reindexing $\{\overline{B}_{n_{j}}^{\vect{T}\times\vect{R}}((x_{j},y),\varepsilon)\}_{j \in \mathcal{J}(i), y\in H(i,j)}$ gives a countable $(n,\varepsilon)$-packing of $G_{i}$,
            which implies by \eqref{eq:defmsrPack} that
            \begin{equation}\label{eq:pacneGi}
            \msrpac_{n,\varepsilon}^{s+t}(\vect{T}\times\vect{R},\vect{f}\dotplus\vect{g},G_{i}) \geq \sum_{j \in \mathcal{J}(i)}\sum_{y \in H(i,j)}\exp{\bigl(-n_{j}(s+t) + S_{n_{j}}^{\vect{T}\times\vect{R}}(\vect{f}\dotplus\vect{g})(x_{j},y)\bigr)}
            \end{equation}
            for all $i=1,\ldots,m$.

           Given $m \geq 1$, for every $j \geq 1$, we write
            $$
            \mathcal{I}_{\leq m}(j)=\{i: j \in \mathcal{J}(i)\ \text{and}\ i \leq m\},
            $$
and by \eqref{eq:prodSkn} and  \eqref{eq:bowGixjsumHij}, it follows that
            \begin{align*}
            S  &:=\sum_{i=1}^{m}\sum_{j \in \mathcal{J}(i)}\sum_{y \in H(i,j)}\exp{\bigl(-n_{j}(s+t) + S_{n_{j}}^{\vect{T}\times\vect{R}}(\vect{f}\dotplus\vect{g})(x_{j},y)\bigr)} \\
              &= \sum_{j=1}^{\infty}\sum_{i \in \mathcal{I}_{\leq m}(j)}\sum_{y \in H(i,j)} \exp{\bigl(-n_{j}(s+t) + S_{n_{j}}^{\vect{T}\times\vect{R}}(\vect{f}\dotplus\vect{g})(x_{j},y)\bigr)} \\
              &= \sum_{j=1}^{\infty}\Bigl(\exp{\bigl(-n_{j}s+S_{n_{j}}^{\vect{T}}\vect{f}(x_{j})\bigr)}\sum_{i \in \mathcal{I}_{\leq m}(j)}\sum_{y \in H(i,j)}\exp{\bigl(-n_{j}t+S_{n_{j}}^{\vect{R}}\vect{g}(y)\bigr)}\Bigr) \\
              &\geq \sum_{j=1}^{\infty}\Bigl(\exp{\bigl(-n_{j}s+S_{n_{j}}^{\vect{T}}\vect{f}(x_{j})\bigr)}\sum_{i \in \mathcal{I}_{\leq m}(j)}\msrbow_{N,2\varepsilon}^{t}(\vect{R},\vect{g},G_{i}^{W}(x_{j}))\Bigr).
            \end{align*}
Since $G_{i}^{W}(x_{j})=\emptyset$ for all $i \notin \mathcal{I}_{\leq m}(j)$, the summation over $i \in \mathcal{I}_{\leq m}(j)$ is essentially the same as summing over $i \in \{1,\ldots,m\}$. Since $x_{j} \in Z_{N,m}$, by \eqref{eq:Znm}, we obtain that
            \begin{align*}
             S\geq l \sum_{j=1}^{\infty}\exp{\bigl(-n_{j}s+S_{n_{j}}^{\vect{T}}\vect{f}(x_{j})\bigr)}.
            \end{align*}

   By \eqref{eq:pacneGi}, it follows that
            $$
            \sum_{i=1}^{m}\msrpac_{n,\varepsilon}^{s+t}(\vect{T}\times\vect{R},\vect{f}\dotplus\vect{g},G_{i}) \geq S \geq l \sum_{j=1}^{\infty}\exp{\bigl(-n_{j}s+S_{n_{j}}^{\vect{T}}\vect{f}(x_{j})\bigr)},
            $$
            and by \eqref{eq:defmsrPack}, we have
            $$
         \sum_{i=1}^{m}\msrpac_{n,\varepsilon}^{s+t}(\vect{T}\times\vect{R},\vect{f}\dotplus\vect{g},G_{i}) \geq l \msrpac_{n,\varepsilon}^{s}(\vect{T},\vect{f},Z_{N,m}) \geq l \msrpac_{\infty,\varepsilon}^{s}(\vect{T},\vect{f},Z_{N,m}).
            $$
Letting $n$ tend to $\infty$, we obtain   the inequality \eqref{eq:sumPGigeqlPZm}, and the conclusion holds.
          \end{proof}

          \begin{thm}\label{thm:prodPPPBPP}
          Given two NDSs $(\vect{X},\vect{T})$ and $(\vect{Y},\vect{R})$, let $Z \subset X_{0}$ and $W \subset Y_{0}$.
          Then for all $\vect{f} \in \vect{C}(\vect{X},\R)$ and $\vect{g} \in \vect{C}(\vect{Y},\R)$,
          $$
          \prepac(\vect{T}\times\vect{R},\vect{f}\dotplus\vect{g},Z \times W) \geq
            \prebow(\vect{T},\vect{f},Z) + \prepac(\vect{R},\vect{g},W).  
          $$
        \end{thm}
 \begin{proof}
          Arbitrarily choose $s<\prebow(\vect{T},\vect{f},Z)$ and $t<\prepac(\vect{R},\vect{g},W)$.
          By Proposition~\ref{prop:msrbowepsilon}, Proposition~\ref{prop:msrpacepsilon},  \eqref{def_PBep} and \eqref{def_PP}, there exists $\varepsilon_{0}>0$ such that for all $0<\varepsilon\leq\varepsilon_{0}$,
          we have that
          $$
          \msrbow_{2\varepsilon}^{s}(\vect{T},\vect{f},Z)=+\infty \qquad \textit{and }\qquad
           \msrpac_{\varepsilon}^{t}(\vect{R},\vect{g},W)=+\infty.
          $$
By Lemma \ref{eq:PRprodneq}, it immediately follows that
          $$
          \msrpac_{\varepsilon}^{s+t}(\vect{T}\times\vect{R},\vect{f}\dotplus\vect{g},Z \times W)=+\infty,
          $$
          and by \eqref{def_PP},  this implies that $\prepac(\vect{T}\times\vect{R},\vect{f}\dotplus\vect{g},Z \times W) \geq s+t$.
          The conclusion is clear by the arbitrariness of $s$ and $t$.
        \end{proof}

 An inequality similar to Theorem~\ref{thm:prodQU} holds for the packing pressure by Theorem~\ref{thm:prepaceqpremod}.
      \begin{thm}\label{thm:prodPP}
        Given two NDSs $(\vect{X},\vect{T})$ and $(\vect{Y},\vect{R})$, let $Z \subset X_{0}$ and $W \subset Y_{0}$.
        Then for all equicontinuous $\vect{f} \in \vect{C}(\vect{X},\R)$ and equicontinuous $\vect{g} \in \vect{C}(\vect{Y},\R)$,
        $$
        \prepac(\vect{T} \times \vect{R},\vect{f} \dotplus \vect{g},Z \times W) \leq \prepac(\vect{T},\vect{f},Z) + \prepac(\vect{R},\vect{g},W).
        $$
      \end{thm}
      \begin{proof}
        Suppose that $\{Z_{i}\}_{i=1}^{\infty}$ and $\{W_{j}\}_{j=1}^{\infty}$ are decompositions of $Z$ and $W$, respectively,
        i.e.,
        $$
        \bigcup_{i=1}^{\infty}Z_{i} \supset Z, \quad \bigcup_{j=1}^{\infty}W_{j} \supset W.
        $$
        Then it is clear that $\{Z_{i} \times W_{j}\}_{i,j}$ is a decomposition of $Z \times W$.
        Since $\vect{f}$ and $\vect{g}$ are both equicontinuous, it follows by Proposition~\ref{prop:PeqQ} and Theorem~\ref{thm:prodQU} that for each $i$ and $j$,
        $$
        \preup(\vect{T}\times\vect{R},\vect{f}\dotplus\vect{g},Z_{i} \times W_{j}) \leq \preup(\vect{T},\vect{f},Z_{i}) + \preup(\vect{R},\vect{g},W_{j}),
        $$
        and so taking the suprema on both sides gives
        $$
        \sup_{i,j}\preup(\vect{T}\times\vect{R},\vect{f}\dotplus\vect{g},Z_{i} \times W_{j}) \leq \sup_{i}\preup(\vect{T},\vect{f},Z_{i}) + \sup_{j}\preup(\vect{R},\vect{g},W_{j}).
        $$
        This implies by Theorem~\ref{thm:prepaceqpremod} that
        $$
        \prepac(\vect{T}\times\vect{R},\vect{f}\dotplus\vect{g},Z \times W) \leq \sup_{i}\preup(\vect{T},\vect{f},Z_{i}) + \sup_{j}\preup(\vect{R},\vect{g},W_{j}),
        $$
  and by the arbitrariness of $\{Z_{i}\}_{i=1}^{\infty}$ and $\{W_{j}\}_{j=1}^{\infty}$, we have that
        $$
        \prepac(\vect{T} \times \vect{R},\vect{f} \dotplus \vect{g},Z \times W) \leq \prepac(\vect{T},\vect{f},Z) + \prepac(\vect{R},\vect{g},W).
        $$
        \end{proof}

\begin{proof}[ Proof of Theorem \ref{thm:prodPPPB}]
It is  a direct consequence of Theorems \ref{thm:prodPPPBPP} and \ref{thm:prodPP}.
\end{proof}

       \begin{proof}[ Proof of Corollary \ref{cor:prodeq} ]
(1) Since  $\prebow(\vect{T},\vect{f},Z)=\prepac(\vect{T},\vect{f},Z)$ or $\prebow(\vect{R},\vect{g},W)=\prepac(\vect{R},\vect{g},W)$, the equalities follow directly from Theorem \ref{thm:prodBBP} and Theorem \ref{thm:prodPPPB}.

(2) Similarly, the equalities are consequences of Theorem \ref{thm:prodLLU} and Theorem \ref{thm:prodLUU}.
\end{proof}

    \section{Invariance under Equiconjugacies}\label{sect:invar}
      Let $(\vect{X},\vect{T})$ and $(\vect{Y},\vect{R})$ be two NDSs, where $X_{k}$ and $Y_{k}$ are endowed with
      the metrics $d_{X_{k}}$ and $d_{Y_{k}}$ respectively throughout this section.

      We call a sequence $\vect{\pi}=\{\pi_{k}\}_{k=0}^{\infty}$ of surjective continuous mappings $\pi_{k}:X_{k} \to Y_{k}$ a \emph{semiconjugacy} or a \emph{factor mapping sequence} from $(\vect{X},\vect{T})$ to $(\vect{Y},\vect{R})$
      if $R_{k} \circ \pi_{k}=\pi_{k+1} \circ T_{k}$  for every $k \in \mathbb{N}$,
      i.e., the diagram
      \begin{equation*}
        \xymatrix{
          (\vect{X},\vect{T}): \ar[d]^{\vect{\pi}:} & X_{0} \ar[r]^{T_{0}} \ar[d]_{\pi_{0}} & X_{1} \ar[r]^{T_{1}} \ar[d]_{\pi_{1}} & X_{2} \ar[r]^{T_{2}} \ar[d]_{\pi_{2}} & \cdots \ar[r]^{T_{k-1}} \ar @{} [d] |{\cdots} & X_{k} \ar[r]^{T_{k}} \ar[d]_{\pi_{k}} & \cdots \ar @{} [d] |{\cdots}\\
          (\vect{Y},\vect{R}):                      & Y_{0} \ar[r]^{R_{0}}                  & Y_{1} \ar[r]^{R_{1}}                  & Y_{2} \ar[r]^{R_{2}}                  & \cdots \ar[r]^{R_{k-1}}                       & Y_{k} \ar[r]^{R_{k}}                  & \cdots
        }
      \end{equation*}
      commutes. If there is a semiconjugacy from $(\vect{X},\vect{T})$ to $(\vect{Y},\vect{R})$,
      we say that $(\vect{Y},\vect{R})$ is a \emph{factor} of $(\vect{X},\vect{T})$.
      Moreover, we call $\vect{\pi}$ a \emph{conjugacy} from $(\vect{X},\vect{T})$ to $(\vect{Y},\vect{R})$ if the continuous mappings $\pi_{k}$ are homeomorphisms for all $k\geq 0$,
      in which case the sequence $\vect{\pi}^{-1}$ of the inverses $\pi_{k}^{-1}$ is a semiconjugacy from $(\vect{Y},\vect{R})$ to $(\vect{X},\vect{T})$.

      \begin{defn}
        Let $\vect{\pi}$ be a semiconjugacy from $(\vect{X},\vect{T})$ to $(\vect{Y},\vect{R})$.
        We say $\vect{\pi}$ is an \emph{equisemiconjugacy} if it is equicontinuous, i.e.,
        for every $\varepsilon>0$, there exists $\delta>0$ such that for all $k \in \mathbb{N}$
        and all $x^{\prime},x^{\prime\prime} \in X_{k}$ satisfying  $d_{X_{k}}(x^{\prime},x^{\prime\prime})<\delta $, we have that
        $$
        d_{Y_{k}}(\pi_{k}x^{\prime},\pi_{k}x^{\prime\prime})<\varepsilon.
        $$
       Given  a conjugacy $\vect{\pi}$, we say $\vect{\pi}$ is an \emph{equiconjugacy}    if    $\vect{\pi}$ and $\vect{\pi}^{-1}$ are   equicontinuous.
      \end{defn}
      Let $\vect{\pi}$ be a semiconjugacy from $(\vect{X},\vect{T})$ to $(\vect{Y},\vect{R})$.
      For every $\vect{g} \in \vect{C}(\vect{Y},\mathbb{R})$, we write $\vect{\pi}^{*}\vect{g}=\{g_{k} \circ \pi_{k}\}_{k=0}^{\infty}$.
      It is clear that $\vect{\pi}^{*}\vect{g} \in \vect{C}(\vect{X},\mathbb{R})$.
      Moreover, $\vect{\pi}^{*}\vect{g}$ is equicontinuous if $\vect{g} \in \vect{C}(\vect{Y},\mathbb{R})$ is equicontinuous and $\vect{\pi}$ is an equisemiconjugacy.

      We show the invariance of pressures under equiconjugacies.
      \begin{thm}\label{thm:invar}
        Let $\vect{g} \in \vect{C}(\vect{X},\mathbb{R})$ be equicontinuous and $Z \subset X_{0}$.
        Suppose that $P \in \{\qrelow, \qreup, \prelow, \preup, \prebow, \prepac\}$.
        \begin{enumerate}[(1)]
          \item If $\vect{\pi}$ is an equisemiconjugacy from $(\vect{X},\vect{T})$ to $(\vect{Y},\vect{R})$,
            then
            $$
            P(\vect{T},\vect{\pi}^{*}\vect{g},Z) \geq P(\vect{R},\vect{g},\pi_{0}(Z)).
            $$
          \item If $\vect{\pi}$ is an equiconjugacy, then
            $$
            P(\vect{T},\vect{\pi}^{*}\vect{g},Z) = P(\vect{R},\vect{g},\pi_{0}(Z)).
            $$
        \end{enumerate}
      \end{thm}

      \begin{proof}
        (1) First, we show the cases for $\qrelow$ and $\qreup$ using the classic argument of Walters' (see \cite[Thm.2.2]{Walters1975} or \cite[Thm.9.8]{Walters1982}).
        The proofs  for $\prelow$ and $\preup$ are similar, and  we omit them.

        For each $\varepsilon>0$, since $\vect{\pi}$ is an equisemiconjugacy, we are able to choose $\delta$ with $0<\delta<\varepsilon$ such that
        for all integers $k \geq 0$, $n >0$ and all $x^{\prime},x^{\prime\prime} \in X_{k}$ satisfying $d_{k,n}^{\vect{T}}(x^{\prime},x^{\prime\prime})<\delta $, we have that
        $$
        d_{k,n}^{\vect{R}}(\pi_{k}x^{\prime},\pi_{k}x^{\prime\prime})<\varepsilon.
        $$
        In particular, for all integers $n >0$, this implies that if $F$ is $(n,\delta)$ spanning for $Z$ with respect to $\vect{T}$, then $\pi_{0}(F)$ is $(n,\varepsilon)$ spanning for $\pi_{0}(Z)$ with respect to $\vect{R}$.
        It follows that
        $$
        \sum_{x \in F}\e^{S_{n}^{\vect{T}}\vect{\pi}^{*}\vect{g}(x)} \geq \sum_{y \in \pi_{0}(F)}\e^{S_{n}^{\vect{R}}\vect{g}(y)} \geq Q_{n}(\vect{R},\vect{g},\pi_{0}(Z),\varepsilon)
        $$
        for every $(n,\delta)$-spanning set $F$ for $Z$ with respect to $\vect{T}$ and hence
        $$
        Q_{n}(\vect{T},\vect{\pi}^{*}\vect{g},Z,\delta) \geq Q_{n}(\vect{R},\vect{g},\pi_{0}(Z),\varepsilon).
        $$
        Therefore, by \eqref{eq:Qe},
        $$
        \qrelow(\vect{T},\vect{\pi}^{*}\vect{g},Z,\delta) \geq \qrelow(\vect{R},\vect{g},\pi_{0}(Z),\varepsilon)
        \quad \text{and} \quad
        \qreup(\vect{T},\vect{\pi}^{*}\vect{g},Z,\delta) \geq \qreup(\vect{R},\vect{g},\pi_{0}(Z),\varepsilon).
        $$
        Since $\delta$ tends to $0$ as $\varepsilon$ goes to $0$, letting $\varepsilon \to 0$ gives 
        $$
        \qrelow(\vect{T},\vect{\pi}^{*}\vect{g},Z) \geq \qrelow(\vect{R},\vect{g},\pi_{0}(Z))
        \quad \text{and} \quad
        \qreup(\vect{T},\vect{\pi}^{*}\vect{g},Z) \geq \qreup(\vect{R},\vect{g},\pi_{0}(Z)).
        $$

        Next, we show the cases for $\prebow$ and $\prepac$.           We prove $\prebow(\vect{T},\vect{\pi}^{*}\vect{g},Z) \geq \prebow(\vect{R},\vect{g},\pi_{0}(Z))$ first. Recall that since $\vect{\pi}$ is an equisemiconjugacy, for every $\varepsilon>0$, there exists $\delta>0$
            such that for all integers $k \geq 0$, $n>0$ and all $x^{\prime},x^{\prime\prime} \in X_{k}$ satisfying $d_{k,n}^{\vect{T}}(x^{\prime},x^{\prime\prime})<\delta $,
            we have that
            $$
            d_{k,n}^{\vect{R}}(\pi_{k}x^{\prime},\pi_{k}x^{\prime\prime})<\varepsilon.
            $$
            This implies that
            $$
            \pi_{k}(B_{k,n}^{\vect{T}}(x,\delta)) \subset B_{k,n}^{\vect{R}}(\pi_{k}x,\varepsilon)
            $$
            for every  $x \in X_{k}$.     In particular, for all integers $n >0$ and all  $x \in X_{0}$,
            \begin{equation}\label{eq:notebowballs}
                \pi_{0}(B_{n}^{\vect{T}}(x,\delta)) \subset B_{n}^{\vect{R}}(\pi_{0}x,\varepsilon).
            \end{equation}

            Fix $s \in \mathbb{R}$, $N>0$ and $\varepsilon>0$. We choose $\delta>0$ as above such that $\delta \to 0$ as $\varepsilon \to 0$.
            Given a countable $(N,\delta)$-cover $\{B_{n_{i}}^{\vect{T}}(x_{i},\delta)\}_{i=1}^{\infty}$ of $Z$, it is clear that $\{B_{n_{i}}^{\vect{R}}(\pi_{0}x_{i},\varepsilon)\}_{i=1}^{\infty}$  is a $(N,\varepsilon)$-cover of $\pi_{0}(Z)$.
            Hence, by \eqref{eq:defmsrbow}, we have that
            \begin{align*}
            \msrbow_{N,\varepsilon}^{s}(\vect{R},\vect{g},\pi_{0}(Z)) &\leq \sum_{i=1}^{\infty}\exp{\left(-n_{i}s + S_{n_{i}}^{\vect{R}}\vect{g}(\pi_{0}x_{i})\right)} \\
                                                                                &= \sum_{i=1}^{\infty}\exp{\left(-n_{i}s + S_{n_{i}}^{\vect{T}}\vect{\pi}^{*}\vect{g}(x_{i})\right)}.
            \end{align*}
            By \eqref{eq:notebowballs}, it is clear that
            $$
            \msrbow_{N,\varepsilon}^{s}(\vect{R},\vect{g},\pi_{0}(Z)) \leq \msrbow_{N,\delta}^{s}(\vect{T},\vect{\pi}^{*}\vect{g},Z),
            $$
            and we obtain that
            $$
            \msrbow_{\varepsilon}^{s}(\vect{R},\vect{g},\pi_{0}(Z)) \leq \msrbow_{\delta}^{s}(\vect{T},\vect{\pi}^{*}\vect{g},Z).
            $$
            By \eqref{def_PBep},    it follows that
            $$
            \prebow(\vect{R},\vect{g},\pi_{0}(Z),\varepsilon) \leq \prebow(\vect{T},\vect{\pi}^{*}\vect{g},Z,\delta).
            $$
            Since $\delta$ tends to $0$ as $\varepsilon$ goes to $0$,
            we have that $\prebow(\vect{R},\vect{g},\pi_{0}(Z)) \leq \prebow(\vect{T},\vect{\pi}^{*}\vect{g},Z)$.

            Next, we prove $\prepac(\vect{T},\vect{\pi}^{*}\vect{g},Z) \geq \prepac(\vect{R},\vect{g},\pi_{0}(Z))$.

            For every $s<\prepac(\vect{R},\vect{g},\pi_{0}(Z))$, by Proposition \ref{prop:msrpacepsilon}, there exists $\varepsilon_{0}>0$
            such that for all $\varepsilon$ satisfying $0<\varepsilon<\varepsilon_{0}$, we have that
            $$
            \prepac(\vect{R},\vect{g},\pi_{0}(Z),\varepsilon)>s.
            $$
            By \eqref{def_PP}, we have that
            $$
            \msrpac_{\varepsilon}^{s}(\vect{R},\vect{g},\pi_{0}(Z))=+\infty.
            $$

            Given a countable cover $\{Z_{i}\}_{i=1}^{\infty}$ of $Z$, it is clear that $\{\pi_{0}(Z_{i})\}_{i=1}^{\infty}$ is a cover of $\pi_{0}(Z)$.
            By the definition of $\msrpac_{\varepsilon}^{s}$, we have
            $$
            \sum_{i=1}^{\infty}\msrpac_{\infty,\varepsilon}^{s}(\vect{R},\vect{g}, \pi_{0}(Z_{i}))=+\infty.
            $$
            Since $\vect{\pi}$ is an equisemiconjugacy, by the same argument as before, there exists $\delta>0$
            such that for all $n \in \mathbb{N}$ and all $x \in X_{0}$,
            $$
            \pi_{0}(\overline{B}_{n}^{\vect{T}}(x,\delta)) \subset \overline{B}_{n}^{\vect{R}}(\pi_{0}x,\varepsilon).
            $$

            Fix $N >0$. For each $i$, let $\{\overline{B}_{n_{i,l}}^{\vect{R}}(y_{i,l},\varepsilon)\}_{l=1}^{\infty}$ be a $(N,\varepsilon)$-packing of $\pi_0(Z_i)$,
            and choose $x_{i,l} \in \pi_{0}^{-1}(y_{i,l}) \cap Z_{i}$.
            Clearly the collection $\{\overline{B}_{n_{i,l}}^{\vect{T}}(x_{i,l},\delta)\}_{l=1}^{\infty}$ is a $(N,\delta)$-packing of $Z_i$.
            Combining this with \eqref{eq:defmsrPack}, we have that
            \begin{align*}
                \msrpac_{N,\delta}^{s}(\vect{T},\vect{\pi}^{*}\vect{g},Z_{i}) &\geq \sum_{l=1}^{\infty}\exp{\left(-n_{i,l}s + S_{n_{i,l}}^{\vect{T}}\vect{\pi}^{*}\vect{g}(x_{i,l})\right)} \\
                                                                                &= \sum_{l=1}^{\infty}\exp{\left(-n_{i,l}s + S_{n_{i,l}}^{\vect{R}}\vect{g}(\pi_{0}x_{i,l})\right)} \\
                                                                                &= \sum_{l=1}^{\infty}\exp{\left(-n_{i,l}s + S_{n_{i,l}}^{\vect{R}}\vect{g}(y_{i,l})\right)}.
            \end{align*}
            It immediately follows that
            $$
            \msrpac_{N,\delta}^{s}(\vect{T},\vect{\pi}^{*}\vect{g},Z_{i}) \geq \msrpac_{N,\varepsilon}^{s}(\vect{R},\vect{g},\pi_0(Z_{i}))
            $$
            for all $i$.
            Letting $N $ tend to $\infty$ and summing the inequality over $i$, we obtain that
            $$
            \sum_{i=1}^{\infty}\msrpac_{\infty,\delta}^{s}(\vect{T},\vect{\pi}^{*}\vect{g},Z_{i}) \geq \sum_{i=1}^{\infty}\msrpac_{\infty,\varepsilon}^{s}(\vect{R},\vect{g},\pi_{0}(Z_{i})) = +\infty.
            $$
            By the arbitrariness of $\{Z_{i}\}_{i=1}^{\infty}$, we have $\msrpac_{\delta}^{s}(\vect{T},\vect{\pi}^{*}\vect{g},Z)=+\infty$,
            and by \eqref{def_PP}, this implies that $\prepac(\vect{T},\vect{\pi}^{*}\vect{g},Z,\delta)>s$ for all $s<\prepac(\vect{R},\vect{g},\pi_{0}(Z))$.
            It follows that
            $$
            \prepac(\vect{R},\vect{g},\pi_{0}(Z),\varepsilon) \leq \prepac(\vect{T},\vect{\pi}^{*}\vect{g},Z,\delta).
            $$
            Letting $\varepsilon \to 0$, we obtain that $\prepac(\vect{R},\vect{g},\pi_{0}(Z)) \leq \prepac(\vect{T},\vect{\pi}^{*}\vect{g},Z)$.

        (2) Since $\vect{\pi}$ is an equiconjugacy from $(\vect{X},\vect{T})$ to $(\vect{Y},\vect{R})$,
        it is clear that  $\vect{\pi}^{-1}=\{\pi_{k}^{-1}\}_{k=0}^{\infty}$ is an equisemiconjugacy from $(\vect{Y},\vect{R})$ to $(\vect{X},\vect{T})$.
        The conclusion follows by applying (1) on both $\vect{\pi}$ and $\vect{\pi}^{-1}$.
      \end{proof}
Immediately, we have the invariance of entropies under equiconjugacies as a direct  consequence of Theorem \ref{thm:invar}. Kawan in \cite{Kawan2014} obtained the same result for $\entup$.
      \begin{cor}
        Let $Z \subset X_{0}$.
        Suppose that $h \in \{\enttoplow, \enttopup, \enttopbow, \enttoppac\}$.
        \begin{enumerate}[(1)]
          \item If $\vect{\pi}$ is an equisemiconjugacy from $(\vect{X},\vect{T})$ to $(\vect{Y},\vect{R})$,
            then
            $$
            h(\vect{T},Z) \geq h(\vect{R},\pi_{0}(Z)).
            $$
          \item If $\vect{\pi}$ is an equiconjugacy,
            then
            $$
            h(\vect{T},Z) = h(\vect{R},\pi_{0}(Z)).
            $$
        \end{enumerate}
      \end{cor}

The next special conclusion shows the relation between $(\vect{X},\vect{T})$ and the shifted $(\vect{X}_{k},\vect{T}_{k})$.
Correspondingly, given $\vect{f} \in \vect{C}(\vect{X},\R)$, we write $\vect{f}_{k}=\{f_{j+k}\}_{j=0}^{\infty}$ for all $k \in \N$.
Clearly, $\vect{f}_{k} \in \vect{C}(\vect{X}_{k},\R)$.
      \begin{cor}
        Given an NDS $(\vect{X},\vect{T})$ and  $P \in \{\qrelow, \qreup, \prelow, \preup, \prebow, \prepac\}$, if $\vect{T}$ is an equicontinuous sequence of surjective mappings $T_{k}$, and $\vect{f} \in \vect{C}(\vect{X},\R)$ is $\vect{T}$-invariant in the sense that $\vect{T}^{k,*}\vect{f}_{k}=\vect{f}$ ($f_{k} \circ \vect{T}_{j}^{k}=f_{j}$ for all $j \in \N$) for all $k \in \N$, then for all $k \in \N$,
        $$
        P(\vect{T},\vect{f},Z) \geq P(\vect{T}_{k},\vect{f}_{k},\vect{T}^{k}Z).
        $$
        Moreover, if $\vect{T}$ is bi-equicontinuous (all $T_{k}$ are homeomorphisms and $\vect{T}^{-1}=\{T_{k}^{-1}\}_{k=0}^{\infty}$ is equicontinuous),
        then
        $$
        P(\vect{T},\vect{f},Z) = P(\vect{T}_{k},\vect{f}_{k},\vect{T}^{k}Z).
        $$
      \end{cor}
      \begin{proof}
        It is easy to verify that for every $k \in \N$, $\{\vect{T}_{j}^{k}\}_{j=0}^{\infty}$ is the equisemiconjugacy (equiconjugacy) of our desire:
        \begin{equation*}
          \xymatrix{
            (\vect{X},\vect{T}): \ar[d]    & X_{0} \ar[r]^{T_{0}} \ar[d]_{\vect{T}_{0}^{k}} & X_{1} \ar[r]^{T_{1}} \ar[d]_{\vect{T}_{1}^{k}} & X_{2} \ar[r]^{T_{2}} \ar[d]_{\vect{T}_{2}^{k}} & \cdots \ar[r]^{T_{j-1}} \ar @{} [d] |{\cdots} & X_{j} \ar[r]^{T_{j}} \ar[d]_{\vect{T}_{j}^{k}} & \cdots \ar @{} [d] |{\cdots}\\
            (\vect{X}_{k},\vect{T}_{k}):   & X_{k} \ar[r]^{T_{k}}                           & X_{k+1} \ar[r]^{T_{k+1}}                       & X_{k+2} \ar[r]^{T_{k+2}}                  & \cdots \ar[r]^{T_{k+j-1}}                       & X_{k+j} \ar[r]^{T_{k+j}}                  & \cdots
          }
        \end{equation*}
        The conclusions follow from Theorem~\ref{thm:invar}.
      \end{proof}
      \begin{rmk}
        A particular case of $\preup$ for $\vect{f}=\vect{0}$, that is, the case of topological entropy, has been considered in \cite[Prop.2.1]{Kawan&Latushkin2015},
        where
        $$
        \enttop(\vect{T},Z) \leq \enttop(\vect{T}_{k},\vect{T}^{k}Z)
        $$
        holds trivially without the assumption of equicontinuity on $\vect{T}$ (see \cite[Lem.4.5]{Kolyada&Snoha1996} for its proof),
        and hence equality is achieved assuming only the equicontinuity of $\vect{T}$.
      \end{rmk}

      Another immediate consequence of Theorem \ref{thm:invar} is that the topological pressures and entropies
      are indeed `topological', or in other words, metric irrelevant in the following sense. For each integer $k\geq 0$, let  $d_{k}$  and  $d_{k}^{\prime}$ be two   metrics on   $X_{k}$. We say $\vect{d}=\{d_{k}\}_{k=0}^{\infty}$ and $\vect{d}^{\prime}=\{d_{k}^{\prime}\}_{k=0}^{\infty}$
      on $\vect{X}$ are  \emph{uniformly equivalent} if the identity mappings
      $$
      \vect{\id}=\{\id_{X_{k}}:(X_{k},d_{k})\to(X_{k},d_{k}^{\prime})\}_{k=0}^{\infty}\quad \text{and}\quad \vect{\id}^{\prime}=\{\id_{X_{k}}^{\prime}:(X_{k},d_{k}^{\prime})\to(X_{k},d_{k})\}_{k=0}^{\infty}
      $$
      are both equicontinuous.

      Consider the NDSs $(\vect{X},\vect{d},\vect{T})$ and $(\vect{X},\vect{d}^{\prime},\vect{T})$.
      We write $h^{\mathrm{B}}$, $\prebow$, $h^{\mathrm{P}}$, $\prepac$, $\entlow$, $\qrelow$, $\prelow$, $\entup$, $\qreup$, and $\preup$ with subscripts $\vect{d}$ and $\vect{d}^{\prime}$
      to emphasize the dependence on metrics.
      \begin{thm}\label{thm:pretopuniequiv}
        Given an NDS $(\vect{X},\vect{T})$, let $\vect{d}$ and $\vect{d}^{\prime}$ be two sequences of metrics
        $d_{k}$ and $d_{k}^{\prime}$ inducing the same topologies on $X_{k}$ for each $k \in \mathbb{N}$. If $\vect{d}$ and $\vect{d}^{\prime}$ are uniformly equivalent, then for all equicontinuous $\vect{f} \in \vect{C}(\vect{X},\mathbb{R})$, we have that
        $$
        P_{\vect{d}}(\vect{T},\vect{f},Z) = P_{\vect{d}^{\prime}}(\vect{T},\vect{f},Z)
        $$
        for every $P \in \{\prebow,\prepac,\preL, \preU\}$.
        In particular,
        $$
        h_{\vect{d}}(\vect{T},Z) = h_{\vect{d}^{\prime}}(\vect{T},Z)
        $$
        for $h \in \{\entbow,\entpac,\entlow,\entup\}$.
      \end{thm}
      \begin{proof}
        Obviously, compositions of $\id_{X_{k}}$ and $T_{k}$ are commutative.
        \begin{equation*}
          \xymatrix{
            (\vect{X},\vect{d},\vect{T}): \ar[d]^{\vect{\id}:}          & (X_{0},d_{0}) \ar[r]^{T_{0}} \ar[d]_{\id_{X_{0}}}     & (X_{1},d_{1}) \ar[r]^{T_{1}} \ar[d]_{\id_{X_{1}}}          & \cdots \ar[r]^{T_{k-1}} \ar @{} [d] |{\cdots} & (X_{k},d_{k}) \ar[r]^{T_{k}} \ar[d]_{\id_{X_{k}}}          & \cdots \ar @{} [d] |{\cdots}\\
            (\vect{X},\vect{d}^{\prime},\vect{T}):                      & (X_{0},d_{0}^{\prime}) \ar[r]^{T_{0}}                  & (X_{1},d_{1}^{\prime}) \ar[r]^{T_{1}}                  & \cdots \ar[r]^{T_{k-1}}                       & (X_{k},d_{k}^{\prime}) \ar[r]^{T_{k}}                  & \cdots
          }
        \end{equation*}
        Since $\vect{d}$ and $\vect{d}^{\prime}$
        are uniformly equivalent, $\vect{\id}$ is an equiconjugacy from $(\vect{X},\vect{d},\vect{T})$ to $(\vect{X},\vect{d}^{\prime},\vect{T})$,    and the conclusion follows by Theorem \ref{thm:invar}.
      \end{proof}

      A particular useful case of uniformly equivalent metric sequence is the following uniformly bounded sequence of metrics.

      Given a sequence $(\vect{X},\vect{d})=\{(X_{k},d_{k})\}_{k=0}^{\infty}$ of metric spaces $X_{k}$ each endowed with $d_{k}$,
      let $\vect{d}_{b}=\{d_{b,k}\}_{k=0}^{\infty}$ be given, for each $k \in \N$,  by
      \begin{equation}\label{eq:dunibound}
        d_{b,k}(x,y) = \frac{d_{k}(x,y)}{1+d_{k}(x,y)}, \quad \text{for all}\ x,y \in X_{k}.
      \end{equation}
      It is easy to verify that $\vect{d}_{b}$ is a uniformly bounded metric sequence and that $\vect{d}$ and $\vect{d}_{b}$ are uniformly equivalent.
      Hence the following result follows by Theorem~\ref{thm:pretopuniequiv}.
      \begin{cor}
        Given an NDS $(\vect{X},\vect{d},\vect{T})$, let $\vect{d}_{b}$ be the metric sequence given by \eqref{eq:dunibound}.
        Then for all equicontinuous $\vect{f} \in \vect{C}(\vect{X},\R)$,
        $$
        P_{\vect{d}}(\vect{T},\vect{f},Z)=P_{\vect{d}_{b}}(\vect{T},\vect{f},Z),
        $$
        where $P \in \{\prebow, \prepac, \preL, \preU\}$.
      \end{cor}

\end{document}